\documentclass[11pt]{amsart}
\usepackage{amssymb,latexsym,amsmath,amscd,amsthm,amsfonts, enumerate}
\usepackage{multirow}
\usepackage{color}
\usepackage[all]{xy}
\usepackage{caption}
\usepackage{soul}
\usepackage{float}
\usepackage{titletoc}
\usepackage[normalem]{ulem}
\usepackage{graphicx}
\usepackage{enumitem}

\usepackage[dvipsnames,svgnames]{xcolor}

\usepackage{tikz,tikz-cd}
\usepackage{mathrsfs}
\usepackage[all]{xy}
\usepackage{tikz}
\usepackage{extarrows}
\usepackage{tikz-cd}
\usetikzlibrary{calc}
\usetikzlibrary{matrix,arrows}
\usetikzlibrary{decorations.pathmorphing}
\usetikzlibrary{shapes.geometric,positioning}
\usetikzlibrary{positioning,shapes,shadows}

\usetikzlibrary{cd} 
\usetikzlibrary{arrows.meta}

\usepackage[colorlinks=true,pagebackref,hyperindex]{hyperref}
\newcommand\myshade{85}
\colorlet{mylinkcolor}{BlueViolet}
\colorlet{mycitecolor}{red}
\colorlet{myurlcolor}{green}

\hypersetup{
  linkcolor  = mylinkcolor!\myshade!black,
  citecolor  = mycitecolor!\myshade!black,
  urlcolor   = myurlcolor!\myshade!black,
  colorlinks = true,
}

\usepackage{tabularx}
\newcolumntype{E}{>{\hsize=0.5cm \centering\arraybackslash}X}%
\newcolumntype{C}[1]{>{\hsize=#1\hsize \centering\arraybackslash}X}%

\numberwithin{equation}{section}
\newtheorem{theorem}{Theorem}[section]
\newtheorem{proposition}[theorem]{Proposition}

\newtheorem{lemma}[theorem]{Lemma}
\newtheorem{question}[theorem]{Question}
\newtheorem{theoremA}{Theorem}

\theoremstyle{definition}
\newtheorem{remark}[theorem]{Remark}
\newtheorem{example}[theorem]{Example}
\newtheorem{definition}[theorem]{Definition}
\newtheorem{notation}[theorem]{Notation}

\definecolor{dark-green}{RGB}{14,150,2}
\definecolor{red}{RGB}{250,0,0}

\newcommand{\bpoint}{\circ}
\newcommand{\rpoint}{\color{red}{\bullet}}

\usepackage{tikz}
\usetikzlibrary{arrows.meta}

\tikzset{
    blackarrow/.style={
        bend left=30,
        line width=0.8pt,
        black,
        ->,
        >={Stealth[length=2.4pt, width=3.5pt]}
    },
    bluearrow/.style={
        line width=1.2pt,
        blue,
        ->,
        >={Triangle[length=3pt, width=4pt]}
    }
}

\begin{document}

\title[On Geometric Models of String Algebras]{On Geometric Models of String Algebras: Uniqueness of Surfaces and Existence of Red Punctures}

\author[Zheng Xin]{Zheng Xin}
\address{Zheng Xin,
	School of Mathematics and Statistics, Shaanxi Normal University, Xi'an 710062, China}
\email{xinzheng1314@yeah.net}

\author[LingChun Zhang]{LingChun Zhang}
\address{LingChun Zhang,
	School of Mathematics and Statistics, Shaanxi Normal University, Xi'an 710062, China}
\email{zlcxpy912@163.com}

\thanks{
}

\keywords{string algebras; geometric model; labelled tiling algebra; red punctures}

\thanks{}

%

\subjclass[2020]{05E10, 
16G20, 
05C10}

\begin{abstract}
A geometric model for string algebras was recently established in \cite{BC24}. Building upon this framework, we characterize the class of string algebras whose geometric models are unique up to equivalence of labelled tiled surfaces. 

Moreover, we provide a necessary and sufficient condition for all geometric models of a string algebra to be entirely free of red punctures, and further give a combinatorial description of string algebras with a common red puncture across all geometric models. In addition, we derive a sufficient condition for a string algebra guaranteeing the presence of red punctures in all geometric models.
\end{abstract}

\maketitle
\setcounter{tocdepth}{2} 

\tableofcontents

\section*{Introduction}\label{Introductions}

In recent years, topological and geometric methods have been extensively applied to the representation theory of algebras.
Numerous studies have investigated geometric models for various categories associated with an algebra. Such geometric models allow one to exploit combinatorial and topological tools from surfaces to obtain a variety of structural results on the corresponding categories.

The geometric interpretations of gentle algebras have been thoroughly developed, encompassing cluster theory \cite{ABCJP12,BZ11,CCS06,CS17,L09}, as well as Fukaya categories \cite{HKK17,LP20}. 
The geometric model for the module category of a gentle algebra was established in \cite{BC21}, and the derived category of a graded gentle algebra can be realized as partially wrapped Fukaya categories of graded marked surfaces; see \cite{HKK17,LP20,OPS18}.
In \cite{C26}, the author deformed the geometric model for the module category given in \cite{BC21}, and then embedded it into the geometric model of the derived category given in \cite{OPS18}.

As a generalization of gentle algebras, string algebras are introduced in \cite{BR87}.
Recently, Baur and Coelho Simões provided geometric models for string algebras and their module categories \cite{BC24}. In fact, a string algebra can be regarded as a quotient algebra of a locally gentle algebra, which can be realized by using so-called dual cellular dissections of oriented marked surfaces \cite {PPP19}. 
Then a string algebra can be realized on the surface by adding labels on the dual cellular dissection, and the authors in \cite{BC24} realized indecomposable modules by permissible curves crossing the dual dissection, where they call such a surface model a labeled tiled surface. 

However, the existence of different choices of locally gentle algebras associated to a string algebra means that the surface associated to a string algebra is not unique in
general (Example \ref{ex:not unique} illustrates this). The reason is that some vertices of string algebras are so-called \emph{non-gentle vertices}  (see Definition \ref{def:non-gentle vertices}). A natural question thus emerges: \\

{\bf Which string algebra admits a unique geometric model up to equivalence?}\\

The main aim of this note is to answer this question. Throughout this paper, we only consider \emph{saturated} labelled tiled surfaces introduced in Definition \ref{def:saturated_model}, which arise from saturated locally gentle covers defined by maximal quadratic ideals of the string algebra.

\begin{theoremA}[Theorem \ref{main-theorem}]\label{Mtheorem:object}
    Let $A$ be a string algebra. The geometric model of $A$ is unique if and only if every non-gentle vertex of Type $\mathrm{(I)}$ or Type $\mathrm{(III)}$ (if any) satisfies the three conditions of Lemma \ref{three-conditions}.
  \end{theoremA}
  
As mentioned above, there may exist several geometric models for a given string algebra, and some of them may contain red punctures, while some of them may not. 
The second result of this note describes a string algebra whose surface models have no red punctures. To this end, we rely on a specific combinatorial structure in the quiver, termed a \textit{gentle-bounded cycle} (see Definition \ref{def:gentle-bounded}).

\begin{theoremA}[Theorem \ref{thm:no-puncture}]\label{thm:B}
	Let \(A=\mathbf{k}Q/I\) be a string algebra. Then all geometric models of \(A\) contain no red puncture if and only if every primitive oriented cycle in \(Q\), if any, is
gentle-bounded.
\end{theoremA}
The third main result of this paper gives sufficient conditions for a string algebra under which all its geometric models contain red punctures, and provides a combinatorial characterization of those algebras admitting a common red puncture across all geometric models. For this purpose, we introduce three special subquiver structures:\textit{essential cycle}, \textit{shuttle cycle}, and \textit{multi‑coupled cycle} (see Definitions \ref{def:essential_cycle}, \ref{def:shuttle_cycle}, and \ref{def:multi-coupled-cycle}).

\begin{theoremA}[Theorem \ref{thm:common-red-essential} and \ref{thm: all-red puncture}]\label{thm:C}
Let \(A=\mathbf{k}Q/I\) be a string algebra. 
\begin{itemize}
    \item[(1)] All geometric models of \(A\) have a common red puncture if and only if \(Q\) contains an essential cycle.
    \item[(2)] If \(Q\) contains an
essential cycle, a shuttle cycle, or a multi-coupled cycle, then every geometric model of \(A\) contains a red puncture.
\end{itemize}
\end{theoremA}

\section*{Acknowledgments}
We are grateful to our supervisor, Prof. Wen Chang, for his  guidance and assistance during our weekly seminars. 
We are also thankful to K. Baur, R. Coelho Simões and B. Dequene for their suggestions on this topic. 
Zheng Xin would like to thank R. Coelho Simões for her help in answering questions about \cite[Example 4.13]{BC24}.
Zheng Xin also thanks Ping He and  Zixu Li for answering his questions on geometric models and for further insightful discussions. 
In addition, Zheng Xin is grateful to M. P. Fonseca for many helpful explanations on fundamental properties of string algebras.

\section{String algebras and their geometric realization}\label{String algebras and their geometric realization}

Throughout this paper, we assume that $\mathbf{k}$ is an algebraically closed field. A \textit{quiver} $Q = (Q_0, Q_1, s, t)$ is a directed graph, which we always assume to be finite and connected, with $Q_0$ the set of vertices, $Q_1$ the set of arrows, and $s, t : Q_1 \to Q_0$ two functions, sending an arrow to its start and target respectively. A \textit{loop} at a vertex $v \in Q_0$ is an arrow $\varepsilon \in Q_1$ with $s(\varepsilon) = t(\varepsilon)= v$. 
A \emph{path} in \( Q \) of length \( n \geq 1 \) is a sequence \(p = a_1 \cdots a_n \) of arrows such that \( t(a_i) = s(a_{i+1}) \) for \( 1 \leq i \leq n - 1 \). The \emph{inverse} of the path $p$ is \(p^{-1} = a_n^{-1} \dots a_1^{-1}\), which satisfies \(s(a_i^{-1}) = t(a_i) = s(a_{i + 1}) = t(a_{i + 1}^{-1})\) for \(1 \leq i \leq n - 1\). 

An \emph{oriented closed walk} is a positive-length path $c=\alpha_1\alpha_2\cdots\alpha_n$ satisfying $t(\alpha_n)=s(\alpha_1)$. It is \emph{primitive} if it is not a proper power \(q^m\) of a strictly shorter oriented closed walk \(q\), where \(m>1\). This condition is independent of the ideal. Two primitive oriented closed walks are \emph{cyclically equivalent} if one is obtained from the other by a cyclic rotation. A quiver is \emph{acyclic} if it contains no oriented closed walk.

By \( \mathbf{k}Q \) we denote the \emph{path algebra} of \( Q \). An ideal \( I \) of \( \mathbf{k}Q \) is called \emph{admissible} if \( R_Q^m \subseteq I \subseteq R_Q^2 \) for some \( m \geq 2 \), where \( R_Q \) denotes the ideal of \( \mathbf{k}Q \) generated by \( Q_1 \). The quotient algebra \( \mathbf{k}Q / I \) is finite dimensional whenever \( I \) is admissible. 
The finite dimensional algebra $A = \mathbf{k}Q / I$ is said to be \emph{monomial} if $I$ is generated by paths of length at least two.
We refer to a path in $I$ of length $s$ a $s$-relation. 
A primitive oriented closed walk \(c=a_1\cdots a_n\) is called a \emph{\(J\)-permitted cycle} for a quadratic ideal $J$ whenever \(a_ia_{i+1}\notin J\) for each \(1 \le i \le n\), with indices taken modulo $n$.

\subsection{String algebras}
We recall the definition of string algebras. 
\begin{definition}\cite{BR87}\label{definition:string algebras}
	We call an algebra $A=\mathbf{k}Q/I$ a \emph{string algebra}, if $Q$ is a quiver and $I$ is an admissible ideal of $\mathbf{k}Q$ satisfying the following conditions:
	\begin{enumerate}[label=(\arabic*), label={(S\arabic*)}]
		\item Each vertex in $Q_0$ is the source of at most two arrows and the target of at most two arrows.\label{S1}
		
		\item For each arrow $a$ in $Q_1$, there is at most one arrow $b$ such that  $ab\notin I$; at most one arrow $c$ such that  $ ca\notin I$.\label{S2} 
		
		\item $I$ is generated by paths of length at least two, i.e. $A$ is monomial. \label{S3}
	\end{enumerate}
\end{definition}

\begin{definition}\label{definition:gentle algebras}
	An algebra $A = \mathbf{k}Q/I$ is called \emph{gentle} if it is a string algebra satisfying the following additional conditions:
	\begin{enumerate}[label=(\arabic*), label={(G\arabic*)}]
		\item For each arrow $a$ in $Q_1$, there is at most one arrow $b'$ such that $ab'\in I$; at most one arrow $c'$ such that $c'a\in I$. \label{G1}
		
		\item $I$ is generated by paths of length two.\label{G2}
	\end{enumerate}
\end{definition}

A vertex $v$ satisfying \ref{S1} and for which the paths of length 2 going through $v$ satisfy
\ref{S2} and \ref{G1} is called a \emph{gentle vertex}. Otherwise, it is called a \emph{non-gentle vertex}, see Definition \ref{def:non-gentle vertices}.

\begin{definition}\label{definition:locally gentle algebras}
	A \emph{locally gentle algebra} is an algebra which admits a presentation $\mathbf{k}Q/I$ satisfying the conditions in Definition \ref{definition:gentle algebras} but which is not necessarily finite dimensional, i.e. there may not exist \( m \geq 2 \) such that \( R_Q^m \subseteq I\).
\end{definition}

\begin{example}\label{ex:string-gentle algebra}
	Consider the quiver $Q$
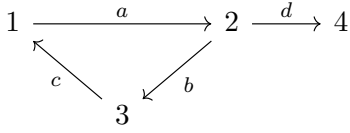
	\begin{center}
		\begin{tikzcd}
			1 \arrow[rr, "a"] &                   & 2 \arrow[ld, "b"] \arrow[r, "d"] & 4 \\
			& 3 \arrow[lu, "c"] &                                  &  
		\end{tikzcd}
	\end{center}
	and the following ideals:
	\begin{center}
		$I_{1}=<bca,ad,ab>, I_{2}=<ad> , I_{3}=<ab>$.
	\end{center}
	Then $\mathbf{k}Q/I_{1}$ is a string algebra which is not gentle nor locally gentle. The algebra $\mathbf{k}Q/I_{2}$ is locally gentle but not gentle, and $\mathbf{k}Q/I_{3}$ is gentle.
	
\end{example}

In (skew-)gentle algebras, each vertex has at most one loop, see \cite[Lemma 1.2]{HZZ23}. By contrast, string algebras allow a vertex to carry up to two loops, as shown below.

\begin{proposition}\label{pro:string_loops}
    Let $A = \mathbf{k}Q/I$ be a string algebra. For any vertex $v \in Q_0$, there are at most two distinct loops at $v$. Furthermore, if there are exactly two loops at $v$, then there are no other arrows in $Q_1$ starting or ending at $v$. Additionally, for any loop $\epsilon$, there exists an integer $m \ge 2$ such that $\epsilon^m \in I$.
\end{proposition}
\begin{proof}
Let $v \in Q_0$ be a vertex. By the definition of a string algebra, $A$ satisfies \ref{S1}, which requires that every vertex in $Q_0$ is the source of at most two arrows and the target of at most two arrows. 

A loop $\varepsilon$ at $v$ satisfies $s(\varepsilon) = t(\varepsilon) = v$. Therefore, each loop at $v$ contributes exactly one to the out-degree of $v$ and exactly one to the in-degree of $v$. Suppose there are $k$ distinct loops at $v$. These loops collectively contribute $k$ to both the out-degree and the in-degree of $v$. By  \ref{S1}, we must have $k \le 2$. Thus, there are at most two loops at $v$.

If there are exactly two distinct loops at $v$ (i.e., $k = 2$), the out-degree and in-degree of $v$ contributed by these loops are already exactly two. Since the maximum out-degree and in-degree permitted by \ref{S1} are both two, no additional arrows can have $v$ as their source or target. Hence, no other arrows can enter or leave $v$.

Finally, since the ideal $I$ is admissible, there exists an integer $N \ge 2$ such that $R_Q^N \subseteq I$. For any loop $\varepsilon$, the path $\varepsilon^N$ has length $N$, which implies $\varepsilon^N \in I$. Because $I$ is a  ideal generated by paths, the generating relation contained within $\varepsilon^N$ must be of the form $\varepsilon^m$ for some integer $2 \le m \le N$. Therefore, $\varepsilon^m \in I$.
\end{proof}

\subsection{Geometric models of string algebras }\label{subsection: geo-module categories}

In \cite{BC24}, the authors realize the string algebras as well as their module categories by using labeled tiled surfaces, which we will recall in this subsection.

\begin{definition}\label{def:marked_surface}
A marked surface is a pair $(\mathcal{S}, \mathcal{M})$, where

(1) $\mathcal{S}$ is an oriented surface with (possibly empty) boundaries with connected components $\partial\mathcal{S}=\sqcup_{i=1}^{b}\partial_{i}\mathcal{S}$;

(2) $\mathcal{M}=\mathcal{M}_{\circ}\cup\mathcal{P}_{\circ}\cup\mathcal{M}_{\rpoint}\cup\mathcal{P}_{\rpoint}$ is a finite set of marked points on $\mathcal{S}$. We refer to the points in $\mathcal{M}_{\circ}\cup\mathcal{P}_{\circ}$ (represented by the symbol $\circ$) as \emph{white marked points}, and the points in $\mathcal{M}_{\rpoint}\cup\mathcal{P}_{\rpoint}$ (represented by the symbol $\rpoint$) as \emph{red marked points}. Each connected component $\partial_{i}\mathcal{S}$ contains at least one marked point of each color from $\mathcal{M}_{\circ}\cup\mathcal{M}_{\rpoint}$, with the white points and the red points alternatingly appearing. The elements in $\mathcal{P}_{\circ}\cup\mathcal{P}_{\rpoint}$ are in the interior of $\mathcal{S}$. Specifically, the interior red points $p \in \mathcal{P}_{\rpoint}$, are formally called \textit{red punctures}.
\end{definition}

We will often simply write $\mathcal{S}$ for $(\mathcal{S},\mathcal{M})$. Let $\mathcal{S}$ be a marked surface. An $\rpoint$-\textit{arc} $\gamma$ is a non-contractible curve with endpoints in $\mathcal{M}_{\rpoint}\cup\mathcal{P}_{\rpoint}$.
On the surface, all $\rpoint$-arcs are considered up to homotopy with respect to the boundary components and the punctures.

\begin{definition}\label{def:dissection}
  A collection \( \mathsf{P} \) of $\rpoint$-arcs is called a \textit{dissection} of \( \mathcal{S} \) if the $\rpoint$-arcs have no interior intersections and they cut the surface into polygons, each of which contains exactly one $\bpoint$-point from $\mathcal{M}_{\circ}\cup\mathcal{P}_{\circ}$, see Figure \ref{fig:tiles}.
  If \( \mathsf{P} \) is a dissection of \( \mathcal{S} \), the pair \( (\mathcal{S}, \mathsf{P}) \) is called a \textit{tiled surface}.
  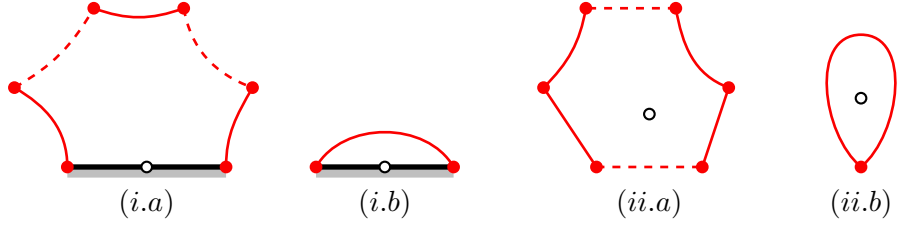
\begin{figure}[htbp]
	
	\begin{tikzpicture}[>=stealth,scale=0.7]
		
		\draw [line width=4pt, gray!50] (-6,-0.1) to (-3,-0.1);
		\draw [line width=2pt, black] (-6,0) to (-3,0);
		\draw [line width=4pt, gray!50] (-1.3,-0.1) to (1.3,-0.1);
		\draw [line width=2pt, black] (-1.3,0) to (1.3,0);

		\draw[red, line width=1pt] (-6,0) to[out=90, in=-30](-7,1.5);
        \draw[red, line width=1pt,dashed](-7,1.5)to[out=30, in=-120] (-5.5,3);
        \draw[red, line width=1pt](-5.5,3)to[out=-20,in=-160](-3.8,3);
        \draw[red, line width=1pt,red,dashed](-3.8,3)to[out=-80,in=160](-2.5,1.5);
        \draw[red, line width=1pt] (-2.5,1.5) to[out=-120, in=90](-3,0) ;
		
		\draw[red, line width=1pt] (-1.3,0) to[out=60, in=120](1.3,0)  ;
		
		\draw[red, line width=1pt] (6,0) to (6.5,1.5);
        \draw[red, line width=1pt](6.5,1.5) to[out=160, in=-80] (5.5,3);
        \draw[red, line width=1pt,dashed](5.5,3)to(3.8,3);
        \draw[red, line width=1pt](3.8,3)to[out=-100,in=45](3,1.5);
        \draw[red, line width=1pt](3,1.5) to (4,0);
        \draw[red, line width=1pt,dashed](4,0) to(6,0) ;
		
		\draw[red, line width=1pt,red] (9,0) to[out=140, in=180](9,2.5)  to[out=0, in=40](9,0)  ;

		\draw[thick,black, fill=white ] (-4.5,0) circle (0.1);
		\draw[thick,black, fill=white ] (0,0) circle (0.1);
		\draw[thick,black, fill=white ] (5,1) circle (0.1);
		\draw[thick,black, fill=white ] (9,1.3) circle (0.1);

		\draw[thick,red, fill=red ] (-6,0) circle (0.1);
		\draw[thick,red, fill=red ] (-3,0) circle (0.1);
		\draw[thick,red, fill=red ] (-1.3,0) circle (0.1);
		\draw[thick,red, fill=red ] (1.3,0) circle (0.1);
		\draw[thick,red, fill=red ] (4,0) circle (0.1);
		\draw[thick,red, fill=red ] (6,0) circle (0.1);

		\draw[thick,red, fill=red ] (-7,1.5) circle (0.1);
		\draw[thick,red, fill=red ] (-5.5,3) circle (0.1);
		\draw[thick,red, fill=red ] (-3.8,3) circle (0.1);
		\draw[thick,red, fill=red ] (-2.5,1.5) circle (0.1);
		\draw[thick,red, fill=red ] (5.5,3) circle (0.1);
		\draw[thick,red, fill=red ] (3.8,3) circle (0.1);
		\draw[thick,red, fill=red ] (3,1.5) circle (0.1);
		\draw[thick,red, fill=red ] (6.5,1.5) circle (0.1);

		\draw[thick,red, fill=red ] (9,0) circle (0.1);
		
		\draw (-4.5,-0.7) node {$(i.a)$};
		\draw (0,-0.7) node {$(i.b)$};
		\draw (5,-0.7) node {$(ii.a)$};
		\draw (9,-0.7) node {$(ii.b)$};
	\end{tikzpicture}
	
	\caption{The  types of polygons in a dissection of Definition \ref{def:dissection}. Figure $(i.b)$ is the special case of a polygon with boundary segments and of size $3$. Figure $(ii.b)$ is the special case of a polygon of size $1$.}
	\label{fig:tiles}
	
\end{figure}
\end{definition}

Let $p \in \mathcal{M}_{\rpoint}\cup\mathcal{P}_{\rpoint}$. If $p \in \mathcal{M}_{\rpoint}$, let $p'$ and $p''$ be two points in the same boundary component such that $p', p'' \not\in \mathcal{M}$ and $p$ is the only marked point in $\mathcal{M}$ lying in the boundary segment between $p'$ and $p''$. Consider a curve $\delta$ isotopic to this boundary segment but such that the only intersection with the boundary of $\mathcal{S}$ is at its endpoints. If $p \in \mathcal{P}_{\rpoint}$, consider a closed simple curve $\delta$ in the interior of $\mathcal{S}$ around $p$ and with no intersections with $\mathcal{M}$.
The {\it complete fan at $p$} is defined to be the sequence of all arcs of $\mathsf{P} $ that $\delta$ crosses in the clockwise order. A {\it fan at $p$} is any subsequence of consecutive arcs of the complete fan at $p$.
Any two consecutive crossings of the curve $\delta$ with arcs in $\mathsf{P}$ define a triangle whose vertices are these two crossings and the marked point $p$; see Figure~\ref{fig:angle}. This triangle is called an {\it angle at $p$} (also called {\it angle of $\mathsf{P} $}), and we say the curve $\delta$ {\it cuts this angle at $p$}.
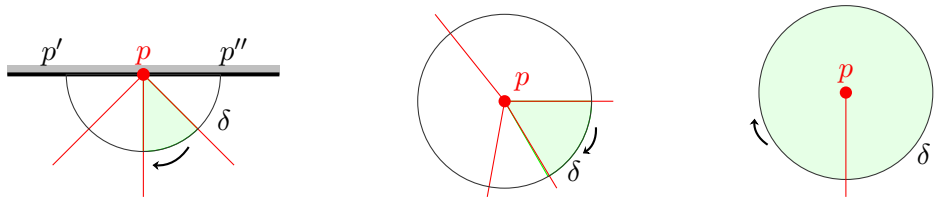
\begin{figure}[htbp]
    \centering
\begin{tikzpicture}[scale=1.2]
   \fill[gray!50] (-1.5,0) rectangle (1.5,0.1);
   \draw[line width=1.5pt] (-1.5,0) -- (1.5,0);
    \node at (-1,0.25) {$p'$};
    \node at (1,0.25) {$p''$};
     \draw[green, fill=green!10]
        (0,0) -- (0.6,-0.6)
        arc(-45:-90:0.85)
        -- cycle;
     \node[right] at (0.7,-0.5) {$\delta$};
     \draw[blackarrow](0.5,-0.8)to(0.1,-1);
     \draw[red] (0,0) -- (1,-1) ;
     \draw[red] (0,0)--(-1,-1) ;
     \draw[red] (0,0)--(0,-1.35) ;
    \draw[black!80] (0,0) -- (0.85,0) arc(0:-180:0.85) -- cycle;
     \fill[red] (0,0) circle (2pt) node[above] {$p$};
\end{tikzpicture}
\hspace{4em}
\begin{tikzpicture}[scale=1.15]
     \node[right] at (0.6,-0.8) {$\delta$};
    \draw[green, fill=green!10] (0,0) -- (1,0) arc(0:-60:1) -- cycle;
    \draw[blackarrow](1.05,-0.3)to(0.9,-0.6);
    \draw[red] (0,0) -- (1.25,0) ;
    \draw[red] (0,0) -- (0.6,-1) ;
    \draw[red] (0,0) -- (-0.8,1) ;
    \draw[red] (0,0) -- (-0.2,-1.1);
    \fill[red] (0,0) circle (2pt) node[above right] {$p$};
    \draw [black!80](0,0) circle (1);
\end{tikzpicture}
\hspace{4em}
\begin{tikzpicture}[scale=1.15]
    \draw[black!80, fill=green!10] (0,0) circle (1);
    \fill[red] (0,0) circle (2pt) node[above] {$p$};
    \draw[red] (0,0) -- (0,-1.2) ;
    \draw[blackarrow](-0.9,-0.6)to(-1.05,-0.3);
    \node[right] at (0.7,-0.7) {$\delta$};
\end{tikzpicture}
    \caption{The shaded triangle is an angle at $p$. A puncture $p \in \mathcal{P}_{\rpoint}$ can have a unique angle which is a self-folded triangle, see the rightmost picture.}
    \label{fig:angle}
\end{figure}

\begin{definition}
 Let \( (\mathcal{S}, \mathsf{P}) \) be a tiled surface. We define \( Q := Q(\mathcal{S}, \mathsf{P}) \) to be the quiver defined as follows:
 
(1) The vertices in $Q$ are in bijection with the arcs in $\mathsf{P}$.

(2) The arrows in $Q$ are in bijection with the angles at the marked points in $ \mathcal{M}_{\rpoint}\cup\mathcal{P}_{\rpoint}$. In other words, there is an arrow $a\colon x_i \to x_j$ in $Q$ if and only if the corresponding arcs ${x}_i$ and ${x}_j$ share an endpoint $p_a \in \mathcal{M}_{\rpoint}\cup\mathcal{P}_{\rpoint}$ and ${x}_j$ is the immediate successor of ${x}_i$ in the complete fan at $p_a$.

\end{definition}
  
\begin{definition}\label{def:label}
    Let \( (\mathcal{S}, \mathsf{P}) \) be a tiled surface.

(1) A {\it label} at a point $p\in \mathcal{M}_{\rpoint}\cup\mathcal{P}_{\rpoint}$ is a finite fan at $p$ of length $\geq 3$.

(2) Given two labels $\ell$ and $\ell'$, if $\ell$ is a subsequence of $\ell'$, then we say the label $\ell'$ is {\it redundant}.

\end{definition}

\medskip
\noindent\textbf{Convention.}
Through this paper, all geometric models are assumed to be \emph{reduced}: their labels are non-redundant and correspond to minimal monomial generators. Accordingly, we omit the adjective “reduced” from our terminology.

\begin{notation}
We will draw a label $\ell$ as a blue curve $\delta$ around the corresponding marked point with an orientation such that $\delta (0)$ is a point in the first arc of $\ell$, $\delta (1)$ is a point in the last arc of $\ell$ and if $\ell$ contains an arc ${x}$ multiple times, the corresponding intersections of $\delta$ with arc ${x}$ are at distinct points. This is in order to make it clear what the label is from $\delta$, especially in the case when the label starts and ends at the same arc, see Example~\ref{ex:not unique} and Figure~\ref{fig:non-homotopic}. In particular, let \( q \in \mathcal{P}_{\rpoint} \). If labels exist at \( q \) but the number of labels is unspecified, we represent them by blue dashed curves, see Figure \ref{fig: shuttle cycles} and Figure \ref{fig:geo-sing algebra}.
\end{notation}
\begin{definition}\label{def:tiled surface}

Let $Q = Q(\mathcal{S}, \mathsf{P})$ and consider the path algebra $\mathbf{k}Q$. 
We define $I_{\mathsf{P},\mathsf{L}}$ to be the ideal of $\mathbf{k}Q$ generated by the paths (see Figure \ref{fig:relation}):

\begin{enumerate}[label=(\text{R}\arabic*), ref=(\text{R}\arabic*)]
    \item \label{rel:R1} $ab$ of length 2 such that $p_a \neq p_b$ or $p_a = p_b$ and $t(a) = s(b)$ corresponds to a \textit{loop arc}.
    
    \item \label{rel:R2} linearly oriented paths in $Q$ corresponding to a label in $\mathsf{L}$.
\end{enumerate}
 
Relations of type \ref{rel:R1} come from the tiles in \( (\mathcal{S}, \mathsf{P}) \), as certain compositions of two arrows within a polygon, and these are exactly the relations used in the construction of (locally) gentle algebras via surfaces. For string algebras, we need also to consider labels and corresponding relations of type \ref{rel:R2}.
 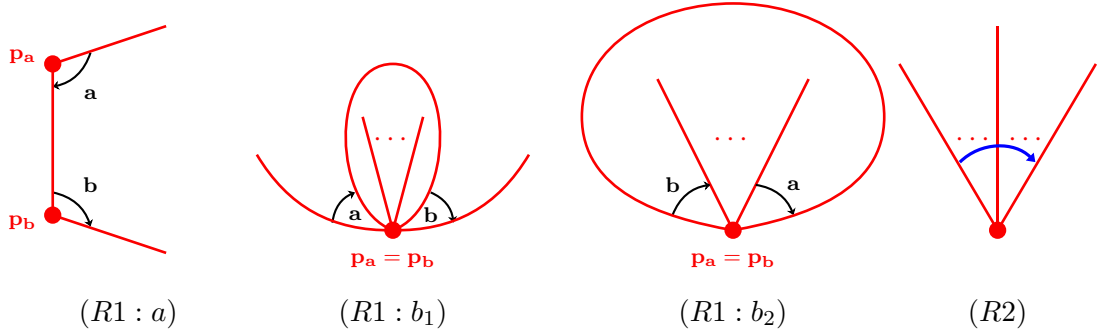
\begin{figure}[htbp]
  \begin{tikzpicture}[>=stealth,scale=1]
   \draw [line width=1pt, red] (0,1.8) to (1,-0.2)to (2,1.8);
   \draw [line width=1pt, red] (3.2,2)to (4.5,-0.2) to (4.5,2.5)to (4.5,-0.2)to (5.8,2);
   \draw [blackarrow] (-7.5,2.15) to (-8,1.7);
   \draw [blackarrow] (-8,0.3) to (-7.5,-0.15);
   \draw [blackarrow] (-4.3,-0.1) to (-4,0.3);
   \draw [blackarrow] (-3,0.3) to (-2.7,-0.1);
   \draw [blackarrow] (0.2,0) to (0.7,0.4);
   \draw [blackarrow] (1.3,0.4) to (1.8,0);
   \draw [bluearrow,bend left=45] (4,0.7) to (5,0.7);

   \draw[line width=1pt,red] (-6.5,-0.5) to(-8,0)to(-8,2)to(-6.5,2.5);
            \draw[line width=1pt,red] (-3.9,1.3) to(-3.5,-0.2)to(-3.1,1.3);
   \draw[bend right, line width=1pt,red] (-5.3,0.8) to(-3.5,-0.2);
   \draw[bend right, line width=1pt,red] (-3.5,-0.2) to(-1.7,0.8);
   \draw[red, line width=1pt,red] (-3.5,-0.2) to[out=160, in=180](-3.5,2)to[out=0, in=20](-3.5,-0.2);
   \draw[red, line width=1pt,red] (1,-0.2) to[out=170, in=-90](-1,1.3)to[out=90, in=180](1,2.8)to[out=0, in=90](3,1.3)to[out=-90, in=10](1,-0.2);
   \draw[thick,red, fill=red ] (-8,0) circle (0.1);
   \draw[thick,red, fill=red ] (-8,2) circle (0.1);
   \draw[thick,red, fill=red ] (-3.5,-0.2) circle (0.1);
   \draw[thick,red, fill=red ] (1,-0.2) circle (0.1);
   \draw[thick,red, fill=red ] (4.5,-0.2) circle (0.1);
   \draw (-3.5,1) node[red] {$\cdots$};
   \draw (4.2,1) node[red] {$\cdots$};
            \draw (4.9,1) node[red] {$\cdots$};
         \draw (1,1) node[red] {$\cdots$};
   \draw (-7.5,1.6) node[font=\scriptsize] {$\mathbf{a}$};
   \draw (-4,0) node[font=\scriptsize] {$\mathbf{a}$};
   \draw (-7.5,0.4) node[font=\scriptsize] {$\mathbf{b}$};
   \draw (-3,0) node[font=\scriptsize] {$\mathbf{b}$};
   \draw (0.2,0.4) node[font=\scriptsize] {$\mathbf{b}$};
   \draw (1.8,0.4) node[font=\scriptsize] {$\mathbf{a}$};
   \draw (-8.4,2.1) node[red,font=\scriptsize] {$\mathbf{p_a}$};
   \draw (-8.4,-0.1) node[red,font=\scriptsize] {$\mathbf{p_b}$};
   \draw (-3.5,-0.6) node[red,font=\scriptsize] {$\mathbf{p_{a}=p_{b}}$};
   \draw (1,-0.6) node[red,font=\scriptsize] {$\mathbf{p_{a}=p_{b}}$};
   \draw (-7,-1.3) node {$(R1:a)$};
   \draw (-3.5,-1.3) node {$(R1:b_{1})$};
   \draw (1,-1.3) node {$(R1:b_{2})$};
   \draw (4.5,-1.3) node {$(R2)$};
  \end{tikzpicture}
  \caption{The ideal \( I_{\mathsf{P},\mathsf{L}} \) of $\mathbf{k}Q$ associated with a tiled surface $(\mathcal{S},\mathsf{P})$, as given in Definition \ref{def:tiled surface}.}
  \label{fig:relation}
 \end{figure}
\end{definition}
\begin{definition}
  Let \( \mathcal{S} = (\mathcal{S}, \mathsf{P}) \) be a tiled surface and let \( Q = Q(\mathcal{S}, \mathsf{P}) \). Let \( \mathsf{L} \) be a finite set of labels on \( \mathcal{S} \) such that for every \( q \in \mathcal{P}_{\rpoint} \), \( \mathsf{L} \) contains a label at \( q \). The \textit{labelled tiling algebra} \( A_{\mathsf{P},\mathsf{L}} \) associated to the data \( (\mathcal{S}, \mathsf{P}, \mathsf{L}) \) is the bound quiver algebra \( A_{\mathsf{P},\mathsf{L}} = \mathbf{k}Q/I_{\mathsf{P},\mathsf{L}} \).
\end{definition}

\begin{definition}\label{def:saturated_model}
    Let $A = \mathbf{k}Q/I$ be a string algebra. A quadratic  ideal $J \subseteq I$ is a \textit{saturated locally gentle cover} of $A$ if $\mathbf{k}Q/J$ is locally gentle and there is no quadratic  ideal $J^\prime$ such that 
    \[J \subsetneq J^\prime \subseteq I \quad\text{and}\quad \mathbf{k}Q/J^\prime \text{ is locally gentle}.\]
    A labelled tiled surface model is called \textit{saturated} if its underlying locally gentle
bound quiver comes from such a cover. 
In what follows, all labelled tiled surfaces considered in this paper are assumed to be saturated.
\end{definition} 

To rigorously formalize the notion of uniqueness, we introduce the following equivalence relation on labelled tiled surfaces:

\begin{definition}\label{def:equivalence_models}
Two labelled tiled surfaces $(\mathcal{S}, \mathsf{P}, \mathsf{L})$ and $(\mathcal{S}', \mathsf{P}', \mathsf{L}')$ are said to be \emph{equivalent} if there exists an orientation-preserving homeomorphism $h: \mathcal{S} \rightarrow \mathcal{S}'$ satisfying the following conditions:
\begin{itemize}
    \item $h$ preserves the colours of the marked points;
    \item $h$ maps the dissection $\mathsf{P}$ to $\mathsf{P}';$
    \item $h$ maps each ordered label in $\mathsf{L}$ to the corresponding ordered label in $\mathsf{L}'.$
\end{itemize}

Throughout this paper, two geometric models are considered equivalent, and the uniqueness of a geometric model for a string algebra is always understood, up to this equivalence of labelled tiled surfaces.
\end{definition}

The following result shows that the class of labelled tiling algebras coincides with
the class of string algebras.
\begin{theorem}\cite[Theorem 3.1]{BC24}\label{string algebra--labelled tiling algebra}
 Let \( A = \mathbf{k}Q/I \) be a finite-dimensional monomial algebra. The following are equivalent.
 \begin{enumerate}
  \item \( A \) is a string algebra.
  \item \( A = (\mathbf{k}Q/J)/(I/J) \), where \( \mathbf{k}Q/J \) is a locally gentle algebra.
  \item \( A \) is a labelled tiling algebra of a marked surface.
 \end{enumerate}
\end{theorem}
By Theorem \ref{string algebra--labelled tiling algebra}, the existence of different choices of locally gentle algebras associated with a string algebra implies that the corresponding surface is not unique in general, see Example \ref{ex:not unique}. 
\begin{example}\label{ex:not unique}
 Consider the quiver $Q$ and the admissible ideal $I_{1}=<bca,ad,ab>$ in Example \ref{ex:string-gentle algebra}. The locally gentle algebras associated with $\mathbf{k}Q/I_1$ admit two distinct choices: the finite-dimensional algebra $B_1 = \mathbf{k}Q/J_1$ with $J_1 = \langle ab\rangle$ and the infinite-dimensional algebra $B_2 = \mathbf{k}Q/J_2$ with $J_2 = \langle ad\rangle$. The geometric model corresponding to $B_{1}$ and $B_{2}$ are illustrated in Figure \ref{fig:non-homotopic}, and are not homotopic to each other.
 \begin{figure}[htbp]
  \begin{center}
   \begin{tikzpicture}[scale=0.35]
    \begin{scope}[shift={(-9,0)}]
     \draw[line width=1.5pt,fill=white] (0,0) circle (6cm);
     \draw[line width=1.5pt,fill=gray!50] (0,0) circle (1cm);
     \path
     (0:1) coordinate (b1)
     (-90:1) coordinate (b2)
     (-180:1) coordinate (b3)
     (-90:6) coordinate (b4)
     (90:6) coordinate (b5)
     (0:6) coordinate (r1)
     (180:6) coordinate (r2)
     (90:1) coordinate (r3);
     \draw[bend left,line width=1pt,red] (b4) to (b3);
     \draw[bend right,line width=1pt, red] (b4) to (b1);
     \draw[red,line width=1pt,red] (b4) to[out=30,in=-90](3.5,0)to[out=90,in=0](0,4)to[out=180,in=90](-3.5,0) to[out=-90,in=150](b4);
     \draw[red,line width=1pt,red] (b4) to[out=10,in=-90](5.2,0)to[out=90,in=-10](b5);

     \draw[blackarrow] (-1.5,-5) to (-0.9,-4.5);
     \draw[blackarrow] (-0.9,-4.5) to (0.9,-4.5);
     \draw[blackarrow] (0.9,-4.5) to (1.5,-5);
     \draw[blackarrow] (1.5,-5) to (2.2,-5.3);
     \draw[bend left,bluearrow] (-2.2,-4) to (2.3,-4);
     \draw[bend left,bluearrow] (1.4,-3) to (4.15,-3.4);

     \draw[thick, red ,fill=red]
     (b1) circle (0.2cm)
     (b3) circle (0.2cm)
     (b4) circle (0.2cm)
     (b5) circle (0.2cm);
     \draw[thick,black, fill=white]
     (r1) circle (0.2cm)
     (b2) circle (0.2cm)
     (r2) circle (0.2cm)
     (r3) circle (0.2cm);
     \draw(-1.8,-1) node[black,font=\scriptsize] {$\mathbf{x_{3}}$};
     \draw(2,-2) node[black,font=\scriptsize] {$\mathbf{x_{1}}$};
     \draw(0,3) node[black,font=\scriptsize] {$\mathbf{x_{2}}$};
     \draw(4.7,0) node[black,font=\scriptsize] {$\mathbf{x_{4}}$};
     \draw(-1.5,-4.2) node[black,font=\scriptsize] {$\mathbf{b}$};
     \draw(0,-4) node[black,font=\scriptsize] {$\mathbf{c}$};
     \draw(1.5,-4.2) node[black,font=\scriptsize] {$\mathbf{a}$};
     \draw(2.2,-4.7) node[black,font=\scriptsize] {$\mathbf{d}$};
    \end{scope}
    \begin{scope}[shift={(6,0)}]
     \draw[line width=1.5pt,fill=white] (0,0) circle (6cm);
     \path
     (0:6) coordinate (b1)
     (-90:6) coordinate (b2)
     (180:6) coordinate (b3)
     (90:6) coordinate (b4)
     (2,0) coordinate (b6)
     (45:6) coordinate (r1)
     (135:6) coordinate (r2)
     (-135:6) coordinate (r3)
     (-45:6) coordinate (r4);
     \draw[line width=1pt,red] (b6) to (b1);
     \draw[line width=1pt,red] (b6) to (b2);
     \draw[line width=1pt,red] (b6) to (b4);
     \draw[line width=1pt,red] (b4) to (b3);

     \draw[bluearrow] (1.9,-0.4) arc[start angle=-120, end angle=-360, radius=0.45] ;
     \draw[bluearrow] (1.75,0.6) arc[start angle=120, end angle=-80, radius=0.75] arc[start angle=-80, end angle=-245, radius=0.85] ;

     \draw[thick,red ,fill=red]
     (b1) circle (0.2cm)
     (b2) circle (0.2cm)
     (b3) circle (0.2cm)
     (b4) circle (0.2cm)
     (b6) circle (0.2cm);
     \draw[thick,black, fill=white]
     (r1) circle (0.2cm)
     (r2) circle (0.2cm)
     (r3) circle (0.2cm)
     (r4) circle (0.2cm);
     \draw(0.25,-3.5) node[black] {\scriptsize$\mathbf{x_{1}}$};
     \draw(0.45,3) node[black] {\scriptsize$\mathbf{x_{2}}$};
     \draw(4.8,0.5) node[black] {\scriptsize$\mathbf{x_{3}}$};
     \draw(-2,3) node[black] {\scriptsize$\mathbf{x_{4}}$};
    \end{scope}
   \end{tikzpicture}
  \end{center}
  \caption{Surfaces that are not equivalent but correspond to the same string algebra $\mathbf{k}Q/I_{1}$ in Example \ref{ex:string-gentle algebra}.}
  \label{fig:non-homotopic}
 \end{figure}
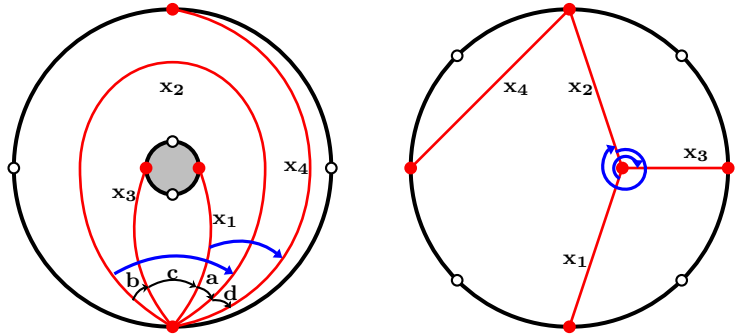
\end{example}

\section{String algebras with unique geometric models}\label{section: string algebras with unique geometric models}

In this section, we characterize a class of special string algebras whose geometric model is unique up to equivalence of labelled tiled surfaces. 

\subsection{Non-gentle vertices} In this subsection, we study non-gentle vertices in detail, which are crucial for classifying the string algebra with a unique surface model described in next subsection. 

\begin{definition}\label{def:non-gentle vertices}
	
Let $A = \mathbf{k}Q/I$ be a string algebra and $v \in Q_0$ a non-gentle vertex. We can delete some relations at $v$ to make it a gentle vertex. This process is referred to as \emph{transforming non-gentle vertices into gentle vertices}. There are exactly three types of non-gentle vertices in a string algebra $A$:
\begin{enumerate}[
    label=\textbf{Type (\Roman*)}, 
    ref=Type \Roman*, 
    leftmargin=2.2cm,        
    labelwidth=2.0cm,        
    align=left               
]
    \item \label{type:I} The vertex $v$ is the target of two arrows $a, b$ and the source of two arrows $c, d$, such that $ac, ad, bc, bd \in I$. 
    
    \item \label{type:II} The vertex $v$ is the target of two arrows $a, b$ and the source of two arrows $c, d$, such that exactly three of the paths $ac, ad, bc, bd$ belong to $I$. For instance, $ac, ad, bd \in I$.
    
    \item \label{type:III} The vertex $v$ has degree three. For Type (III.a), $v$ is the target of two arrows $a, b$ and the source of one arrow $c$, such that $ac, bc \in I$. Type (III.b) is defined symmetrically by reversing the arrows.
\end{enumerate}

\end{definition}

\begin{remark}\label{remark:transform non-gentle}
	 Let $v \in Q_{0}$ be  a non-gentle vertex.

  (1) There are three cases for transforming it into a gentle vertex.
	
	 $\bullet$ If $v$ is a $(\mathrm{I})$-non-gentle vertex, there are two choices to transform it into a gentle vertex: one consists in deleting the relations \(\{ac, bd\}\), the other in removing \(\{ad, bc\}\). 
	
	 $\bullet$ If $v$ is a $(\mathrm{II})$-non-gentle vertex, there is only one choice to transform it into a gentle vertex, namely deleting the relation \(\{ad\}\). 
	
	 $\bullet$ If $v$ is a $(\mathrm{III}.a)$-non-gentle vertex, there are also two choices to transform it into a gentle vertex: deleting the relation \(\{ac\}\) or deleting the relation \(\{bc\}\). The Type $(\mathrm{III}.b)$ is similar, so when analyzing Type $(\mathrm{III})$ non-gentle vertices below, we only discuss the case shown in Type $(\mathrm{III}.a)$.
   
    (2) If the vertex $v$ admits a loop $\varepsilon$, we may classify the type of the non-gentle vertex $v$ by regarding $\varepsilon$ as an arrow entering $v$ and an arrow leaving $v$.

\end{remark}

Suppose $A=\mathbf{k}Q/I$ is a string algebra. Let $n_1,n_2,n_3$ denote the numbers of $(\mathrm{I})$-non-gentle vertices, $(\mathrm{II})$-non-gentle vertices and $(\mathrm{III})$-non-gentle vertices, respectively, and let $n_0$ be the number of relations of length strictly greater than 2. In the aforementioned construction of associated locally gentle algebras, we obtain such algebras by removing a fixed number of relations from the string algebra, and this fixed number is exactly $n_0+2n_1+n_2+n_3$.

Although the number of deleted relations is an invariant of the string algebra, the specific choice of relations to delete at non-gentle vertices of Type (I) and Type (III) is not unique. Since each such vertex admits exactly two choices, there are $2^{n_1+n_3}$ ways to construct the corresponding locally gentle algebras, yielding up to $2^{n_1+n_3}$ distinct geometric models.

\begin{example}\label{ex:two_loops}
Let $A = \mathbf{k}Q/I$ be a string algebra containing a vertex $v$ with exactly two distinct loops, $a$ and $b$. By Proposition \ref{pro:string_loops}, $v$ admits no other incident arrows. Thus, the set of all length-2 paths passing through $v$ is exactly $\{a^2, ab, ba, b^2\}$.\\
$(1)$ The vertex $v$ is gentle if and only if either of the following two conditions holds:
    \begin{itemize}
        \item $a^2, b^2 \in I$ (while $ab, ba \notin I$), provided there exists a finite alternating path of length $\ge 3$ (composed of $a$ and $b$) contained in $I$.
        \item $ab, ba \in I$ (while $a^2, b^2 \notin I$), provided $a^m, b^n \in I$ for some integers $m, n \ge 3$.
    \end{itemize}
$(2)$ The vertex $v$ is non-gentle if and only if either of the following two conditions holds:
    \begin{itemize}
        \item Type (I): All four length-2 paths belong to $I$ (i.e., $a^2, b^2, ab, ba \in I$).
        \item Type (II): Exactly three of the length-2 paths belong to $I$ (e.g., $a^2, b^2, ab \in I$, while $ba \notin I$).
        
    \end{itemize}

\end{example}

\subsection{Admissible deletion datum}

As shown in the foregoing discussion, the uniqueness of geometric models of a string algebra $A$ is determined by its associated locally gentle algebra together with the associated relation of Type \ref{rel:R2}. The associated locally gentle algebra and the associated labelled relation are closely related to the non‑gentle transforms of $A$. We therefore need to specify the concrete choices made at non‑gentle vertices.

\begin{definition}
Let $A=\mathbf{k}Q/I$ be a string algebra. An \emph{admissible deletion datum} $\delta$ consists of the following two parts:
\begin{enumerate}
    \item a choice of 2-relations to be deleted at every non-gentle vertex of $A$;
    \item the deletion of all generating relations of length at least $3$ contained in $I$,
\end{enumerate}
such that the remaining relations generate an ideal $J_\delta\subseteq I$, which defines a locally gentle algebra $B_\delta=\mathbf{k}Q/J_\delta$. Let $\mathfrak{D}(A)$ denote the set of all admissible deletion data for $A$.
\end{definition}

For each $\delta \in \mathfrak{D}(A)$, the deleted  relations naturally correspond to a set of labels on the geometric model, which we denote by $\mathcal{L}_\delta$. Thus, each $\delta$ yields a geometric pair $(J_\delta, \mathcal{L}_\delta)$.

\begin{definition}\label{def:deletion_data_equiv}
Two admissible deletion data $\delta, \delta' \in \mathfrak{D}(A)$ are said to be \emph{equivalent} (denoted by $\delta \sim \delta'$) if there exists a bound-quiver isomorphism $\varphi: (Q, J_\delta) \xrightarrow{\sim} (Q, J_{\delta'})$ such that the induced map on paths perfectly preserves the label data, i.e., $\varphi(\mathcal{L}_\delta) = \mathcal{L}_{\delta'}$.
\end{definition}

\begin{remark}\label{rem:label-perserve}
    In Definition~\ref{def:deletion_data_equiv}, the equality $\phi(\mathcal L_{\delta})=\mathcal L_{\delta'}$ is always understood in the following local sense. Let \(v\) be a non-gentle vertex at which \(\delta\) and \(\delta'\) make different deletion choices, and let \(\Sigma(v)\) denote the symmetric region involved in these two local choices. On \(\Sigma(v)\), the isomorphism \(\phi\) maps each label in
\(\mathcal L_{\delta}\) to its symmetrically corresponding label in \(\mathcal L_{\delta'}\) according to the prescribed local symmetry. Outside \(\Sigma(v)\), every label, together with its supporting path, remains unchanged.

If a labelled path meets both \(\Sigma(v)\) and its complement, then the part supported on \(\Sigma(v)\) is mapped according to the prescribed local symmetry, whereas the part supported on the complement remains unchanged. Moreover, the order of all occurrences along the labelled path is preserved.

Consequently, the equality
$\phi(\mathcal L_{\delta})=\mathcal L_{\delta'}$ 
does not mean that the labels may be permuted arbitrarily. Rather, the induced label map is required to be compatible, at every ordered occurrence, with the prescribed local symmetry at each non-gentle vertex. In particular, the restriction of any global label-preserving bound-quiver isomorphism to \(\Sigma(v)\) induces a label-preserving bound-quiver isomorphism between the two local deletion choices at \(v\). This convention is compatible with the
equivalence relation on  labelled tiled surfaces given in Definition~\ref{def:equivalence_models}.
\end{remark}

To ensure that the geometric model is unique, the local subquivers around these vertices must exhibit structural symmetry. We have the following definition.

\begin{definition}\label{def:symmetric_vertex}
Let $A = \mathbf{k}Q/I$ be a string algebra, and let $\mathfrak{D}(A)$ denote the set of all admissible deletion data. For any $\delta \in \mathfrak{D}(A)$ and any non-gentle vertex $u$, let $\delta(u)$ denote the specific length-$2$ relations deleted at $u$ under the choice $\delta$.

Let $v \in Q_0$ be a non-gentle vertex of Type (I) or Type (III). We say that the bound quiver $(Q, I)$ is \emph{globally symmetric with respect to $v$} if, for any two deletion data $\delta_1, \delta_2 \in \mathfrak{D}(A)$ that differ exclusively at the vertex $v$ (i.e., $\delta_1(u) = \delta_2(u)$ for all $u \neq v$, and $\delta_1(v) \neq \delta_2(v)$), there exists a global bound-quiver automorphism $\Psi: (Q, I) \xrightarrow{\sim} (Q, I)$ satisfying the following condition:

The automorphism $\Psi$ induces a bijection on the set of relations, mapping the entire deletion datum $\delta_1$ exactly to $\delta_2$ as sets. Specifically, this requires:
\begin{enumerate}
    \item[$(\mathrm{i})$] $\Psi$ maps $v$ to some non-gentle vertex $v'$ (where possibly $v' = v$), such that the deleted relations at $v$ are mapped to the deleted relations at $v'$, i.e., $\Psi(\delta_1(v)) = \delta_2(v')$;
    \item[$(\mathrm{ii})$] $\Psi$ permutes the remaining deletion choices, meaning $\Psi(\delta_1 \setminus \delta_1(v)) = \delta_2 \setminus \delta_2(v')$ as sets.
\end{enumerate}
Consequently, this global automorphism ensures that the induced label data are equivalent, satisfying $\Psi(\mathcal{L}_{\delta_1}) = \mathcal{L}_{\delta_2}$, and yielding equivalent geometric models.
\end{definition}

\begin{remark}\label{rem:complementary_bound_subquiver}
Let $v$ be a non-gentle vertex of Type~\textup{(I)} or
Type~\textup{(III)}. Throughout this section, a branch associated
with an arrow incident with $v$ means the full connected branch of
the underlying unoriented quiver $Q\setminus\{v\}$
determined by that arrow.

Let $\Sigma(v)$ be the full subquiver induced by $v$ together with
the connected branches supporting a local symmetry prescribed by
Condition~\ref{U1} or~\ref{U2} of
Lemma~\ref{three-conditions}. Let $C(v)$ be the full subquiver
induced by $v$ together with the remaining connected branches. We
call $C(v)$ the \emph{complementary bound subquiver} of
$\Sigma(v)$. By construction,
\[
Q=\Sigma(v)\cup C(v),
\qquad
\Sigma(v)\cap C(v)=\{v\}.
\]

Suppose that the two local deletion choices at $v$ give rise to
the bound quivers $(Q,I_1)$ and $(Q,I_2)$, and that
\[
\varphi_v:
\bigl(\Sigma(v),I_1|_{\Sigma(v)}\bigr)
\xrightarrow{\sim}
\bigl(\Sigma(v),I_2|_{\Sigma(v)}\bigr)
\]
is the local bound-quiver isomorphism induced by the prescribed
symmetry. Assume that:
\begin{enumerate}[label=\textup{(\arabic*)}]
  \item $\varphi_v(v)=v$;
  \item the restrictions of $I_1$ and $I_2$ to $C(v)$ coincide;
  \item after extending $\varphi_v$ by the identity on $C(v)$,
        every relation containing arrows from both
        $\Sigma(v)$ and $C(v)$ is mapped from $I_1$ to $I_2$;
  \item the corresponding label data are preserved.
\end{enumerate}
Then $\varphi_v$ and the identity map on $C(v)$ agree on their
common vertex $v$, and hence define a global label-preserving
bound-quiver isomorphism
\[
\widetilde{\varphi}_v:(Q,I_1)\xrightarrow{\sim}(Q,I_2)
\]
given by
\[
\widetilde{\varphi}_v|_{\Sigma(v)}=\varphi_v,
\qquad
\widetilde{\varphi}_v|_{C(v)}
=\operatorname{id}_{C(v)}.
\]

Thus, the statement that the complementary bound subquiver remains
unchanged means not only that its underlying quiver is unchanged,
but also that its relations, the relations crossing the common
vertex $v$, and the corresponding label data are compatible with
the local symmetry.
\end{remark}

\subsection{Unique geometric model}

Combining Remark \ref{remark:transform non-gentle} with Theorem \ref{string algebra--labelled tiling algebra}, the conclusion can be drawn as follows:
\begin{lemma}\label{Lem:geo-unique}
	Let $A = \mathbf{k}Q/I$ be a string algebra. There is a one-to-one correspondence between the equivalence classes of geometric models of $A$ and the equivalence classes of admissible deletion data $\mathfrak{D}(A)/\sim$. 
\end{lemma}

\begin{proof}
Let $\mathfrak{G}(A)$ denote the set of equivalence classes of geometric models of $A$. We define a map
\[
\Theta\colon \mathfrak{D}(A)/{\sim}\longrightarrow \mathfrak{G}(A)
\]
by sending the class of an admissible deletion datum $\delta$ to the
class of the labelled tiled surface associated with $(Q,J_\delta,\mathcal L_\delta)$.

We first show that $\Theta$ is well defined. By Remark~\ref{remark:transform non-gentle}, the local deletion
choices make \(J_\delta\) maximal among the quadratic ideals \(J\subseteq I\) for which \(\mathbf{k}Q/J\) is locally gentle. Thus, \(J_\delta\) is a saturated locally gentle cover of \(A\). Since the deleted minimal relations are encoded by \(\mathcal L_\delta\), the pair \((J_\delta,\mathcal L_\delta)\) determines a geometric model of \(A\). If \(\delta\sim\delta'\), then we have a label‑preserving bound‑quiver isomorphism \((Q,J_\delta,\mathcal L_\delta) \xrightarrow{\sim} (Q,J_{\delta'},\mathcal L_{\delta'})\) that induces an equivalence of the associated labelled tiled surfaces. Hence \(\Theta\) is well defined.

We next prove surjectivity. Let \((S,P,\mathcal L)\) be a geometric model of \(A\). After identifying its associated labelled tiling algebra with \(A=\mathbf{k}Q/I\), its underlying locally gentle bound quiver has the form \((Q,J)\), where \(J\subseteq I\) is a saturated locally gentle cover. By Remark~\ref{remark:transform non-gentle}, \(J\) determines an admissible local deletion choice at every non-gentle vertex, while \(\mathcal L\) records the remaining deleted minimal relations, including those of length at least three. Hence these data define some \(\delta\in\mathfrak D(A)\) such that
$J_\delta=J$ and $\mathcal L_\delta=\mathcal L$.
Therefore, \((S,P,\mathcal L)\) lies in the image of \(\Theta\).

Finally, suppose that $\Theta([\delta])=\Theta([\delta'])$. An
equivalence between the corresponding labelled tiled surfaces maps
one dissection to the other and preserves all ordered labels. It induces a
label-preserving bound-quiver isomorphism
$(Q,J_\delta,\mathcal L_\delta)\xrightarrow{\sim}
(Q,J_{\delta'},\mathcal L_{\delta'})$.
This satisfies the equivalence criterion for admissible deletion data in Definition \ref{def:deletion_data_equiv}, hence \(\delta\sim\delta'\), which proves injectivity of \(\Theta\).

Thus, we obtain the bijection $\mathfrak{G}(A) \cong \mathfrak{D}(A)/{\sim}$.
\end{proof}

By Lemma \ref{Lem:geo-unique}, the geometric model of $A$ is unique if and only if all admissible deletion data are equivalent. Algebraically, this means the automorphism group of the bound quiver acts transitively on $\mathfrak{D}(A)$, ensuring that any two choices of deleted relations $(J_\delta, \mathcal{L}_\delta)$ and $(J_{\delta'}, \mathcal{L}_{\delta'})$ are isomorphic via a label-preserving bound-quiver isomorphism.

If the string algebra $A$ contains no non-gentle vertices, the uniqueness of its geometric model is trivial. We first consider the case where A contains non-gentle vertices, specifically of Type (II).

\begin{lemma} \label{lem:non-II}
Let $A=\mathbf{k}Q/I$ be a string algebra. If all non-gentle vertices of $A$ are of Type (II), then the geometric model of $A$ is unique.
\end{lemma}

\begin{proof}
By Remark \ref{remark:transform non-gentle}, $\mathfrak{D}(A)$ contains precisely one element in this case. The assertion follows directly from Lemma \ref{Lem:geo-unique}.
\end{proof}

Observe that under the assumptions of Lemma \ref{lem:non-II}, the geometric model of $A$ is absolutely unique, with no need to pass to equivalence under labelled tiled surfaces. Next, we consider the case where the algebra contains non-gentle vertices of Type (I) and/or Type (III). The following lemma gives a sufficient condition for the geometric model to be unique.

\begin{lemma}\label{three-conditions}
Let $A = \mathbf{k}Q/I$ be a string algebra and $A$ contains non-gentle vertices of Type (I) and/or Type (III). Then the geometric model of $A$ is unique, if the bound quiver $(Q, I)$ satisfies the following conditions:
\begin{enumerate}[label=(\arabic*), label={(U\arabic*)}]
    \item\label{U1} If $v \in Q_0$ is a Type (I) non-gentle vertex, then the local structure at $v$ satisfies(see Figure \ref{non-gentle-sys-1}):
    \begin{enumerate}
        \item[$(\mathrm{a1})$] There exists a label-preserving bound-quiver isomorphism either between the full subquivers $(Q^{(1)}, I|_{Q^{(1)}})$ and $(Q^{(2)}, I|_{Q^{(2)}})$, or between $(Q^{(3)}, I|_{Q^{(3)}})$ and $(Q^{(4)}, I|_{Q^{(4)}})$, mapping the local deleted relations symmetrically;
        \item[$(\mathrm{a2})$] The full subquivers \(Q^{(1)}\cup Q^{(2)}\) and \(Q^{(3)}\cup Q^{(4)}\) lie in distinct connected components of the underlying unoriented quiver after removing \(v\).
        
    \end{enumerate}
    \item\label{U2} If $v \in Q_0$ is a Type (III) non-gentle vertex, then there exists a label-preserving bound-quiver isomorphism between $(Q^{(1)}, I|_{Q^{(1)}})$ and $(Q^{(2)}, I|_{Q^{(2)}})$.
    \begin{figure}[htbp]
		\begin{center}
				\begin{tikzpicture}[>=stealth, scale=0.5]
					\begin{scope}[shift={(-9, 0)}]
						\path 
						(-2.3,2)  coordinate (1_1)
						(-2.3,-2) coordinate (2_1)
						(0,0) coordinate (3_1)
						(2.3,2)  coordinate (4_1)
						(2.3,-2)  coordinate (5_1);
						
						\draw 
						(1_1) node[purple] {$Q^{(3)}$}
						(2_1) node[purple] {$Q^{(2)}$}
						(3_1) node[blue] {$v$}
						(4_1) node[purple] {$Q^{(1)}$}
						(5_1) node[purple] {$Q^{(4)}$};
						
						\draw [thick,->] (-0.15,0.15) to (-1.85,1.85);
						\draw [thick,->] (-1.65,-1.65) to (-0.15,-0.15);
						\draw [thick,->] (0.15,-0.15) to (1.7,-1.7);
						\draw [thick,->] (1.7,1.7) to (0.15,0.15) ;
						
						\draw[,line width=1pt, dashed, red] (0.6,0) arc[start angle=0, end angle=-370, radius=0.6]  ;	
						\draw (-1,-1.4) node {$b$};
						\draw (-1,1.4) node {$c$};
						\draw (1,1.4) node {$a$};
						\draw (1,-1.4) node {$d$};
                        
                        \draw (0,-3.5) node {$ac,ad,bc,bd \in I$};
				\draw (0,-5) node {\textbf{Type (I)}};
					\end{scope}
					
					\begin{scope}[shift={(0, 0)}]
						\path 
						(-2.3,2)  coordinate (1_2)
						(-2.3,-2) coordinate (2_2)
						(0,0) coordinate (3_2)
						(2.6,0)  coordinate (4_2);
						
						\draw 
						(1_2) node[purple] {$Q^{(1)}$}
						(2_2) node[purple] {$Q^{(2)}$}
						(3_2) node[blue] {$v$}
						(4_2) node[purple] {$Q^{(3)}$};
						
						\draw [thick,->] (-1.85,1.85) to (-0.15,0.15);
						\draw [thick,->] (-1.65,-1.65) to (-0.15,-0.15);
						\draw [thick,->] (0.2,0) to (1.85,0);
						
						\draw[, line width=1pt,red, dashed] (-0.5,0.5) to[out=30, in=90](0.5,0);
						\draw[, line width=1pt,red, dashed] (-0.5,-0.5) to[out=-30, in=-90](0.5,0);
						
						\draw (-1,-1.4) node {$b$};
						\draw (-1,1.4) node {$a$};
						\draw (1,0.2) node {$c$};

                        \draw (0,-3.5) node {$ac,bc \in I$};
                        \draw (0,-5) node {\textbf{Type (III.a)}};
					\end{scope}
					
					\begin{scope}[shift={(9, 0)}]
						\path 
						(-2.5,0)  coordinate (1_3)
						(0,0) coordinate (2_3)
						(2.6,2) coordinate (3_3)
						(2.6,-2)  coordinate (4_3);
						
						\draw 
						(1_3) node[purple] {$Q^{(3)}$}
						(2_3) node[blue] {$v$}
						(3_3) node[purple] {$Q^{(1)}$}
						(4_3) node[purple] {$Q^{(2)}$};
						
						\draw [thick,->] (0.15,-0.15) to (1.85,-1.85);
						\draw [thick,->] (0.15,0.15) to (1.85,1.85);
						\draw [thick,->] (-1.85,0) to (-0.2,0);
						
						\draw[, line width=1pt,red, dashed] (-0.7,0) to[out=90, in=150](0.5,0.5);
						\draw[, line width=1pt,red, dashed] (-0.7,0) to[out=-90, in=-150](0.5,-0.5);
						
						\draw (1,1.4) node {$b$};
						\draw (1,-1.4) node {$c$};
						\draw (-1,0.2) node {$a$};

                        \draw (0,-3.5) node {$ab,ac \in I$};
                        \draw (0,-5) node {\textbf{Type (III.b)}};
					\end{scope}
				\end{tikzpicture}

		\end{center}
        \caption{Local quiver configurations for (I)-non-gentle vertex $v$ satisfying condition \ref{U1} and (III)-non-gentle vertex $v$ satisfying condition \ref{U2} in Lemma \ref{three-conditions}.}
		\label{non-gentle-sys-1}
	\end{figure}
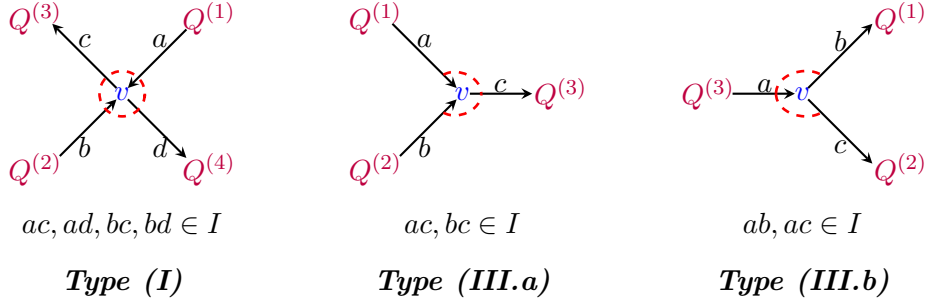
    \item\label{U3} For any two Type (III) non-gentle vertices, none of the four subquiver configurations shown in Figure \ref{two III vertices} exists in $(Q, I)$.
    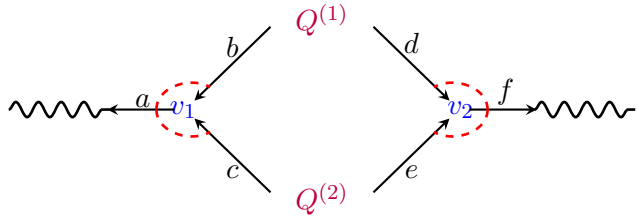
\begin{figure}[htbp]
		\begin{center}
			\begin{tikzpicture}[>=stealth, scale=0.592]
				\path 
				
				(-7,0)  coordinate (1)
				(-4.75,0) coordinate (2)
				(-3.1,0) coordinate (3)
				(-1,2)  coordinate (4)
				(-1,-2)  coordinate (5)
				(1,2) coordinate (6)
				(1,-2) coordinate (7)
				(3.1,0)  coordinate (8)
				(4.75,0)  coordinate (9)
				(7,0) coordinate (10)
				;

				\draw 
				(3) node[blue] {$v_{1}$}
				(8) node[blue] {$v_{2}$}
				;
                
				\draw 
						(0,2) node[purple] {$Q^{(1)}$}
						(0,-2) node[purple] {$Q^{(2)}$};
				\draw[decorate,line width=1pt,decoration={snake,amplitude=1mm,segment length=3mm}] (1)--(2);
				\draw[decorate,line width=1pt,decoration={snake,amplitude=1mm,segment length=3mm}] (9)--(10);
				
				\draw [thick,->] (-1.15,-1.85) to node[below right ]{}(-2.85,-0.2);
				\draw [thick,->] (-1.15,1.85) to node[above right ]{}(-2.85,0.2);
				\draw [thick,->] (-3.3,0) to (-4.8,0);
				
				\draw [thick,->] (1.15,-1.85) to node[below right ]{}(2.85,-0.2);
				\draw [thick,->] (1.15,1.85) to node[above right ]{}(2.85,0.2);
				\draw [thick,->] (3.3,0) to (4.8,0);

				\draw[, line width=1pt,red, dashed] (-3.7,0) to[out=90, in=150](-2.5,0.5);
				\draw[, line width=1pt,red, dashed] (-3.7,0) to[out=-90, in=-150](-2.5,-0.5);
				\draw[, line width=1pt,red, dashed] (3.7,0) to[out=90, in=30](2.5,0.5);
				\draw[, line width=1pt,red, dashed] (3.7,0) to[out=-90, in=-30](2.5,-0.5);

				\draw (-4,0.2) node {$a$};
				\draw (-2,1.4) node {$b$};
				\draw (-2,-1.4) node {$c$};
				\draw (4.1,0.4) node {$f$};
				\draw (2,1.5) node {$d$};
				\draw (2,-1.4) node {$e$};

			\end{tikzpicture}
		\end{center}
		\caption{The non-gentle vertices \(v_1\) and \(v_2\) may be of either Type $(\mathrm{III}.a)$ or Type $(\mathrm{III}.b)$.}
		\label{two III vertices}
	\end{figure}
\end{enumerate}
\end{lemma}

\begin{proof}
Let $n$ be the number of non-gentle vertices of
Type~\textup{(I)} or Type~\textup{(III)} in $Q_0$.
The deletion of every relation of length at least three,
as well as the deletion at each Type~\textup{(II)} non-gentle
vertex, is uniquely determined. We therefore fix all these
deletions throughout the proof. By
Lemma~\ref{Lem:geo-unique}, it suffices to prove that all
admissible deletion data in $\mathfrak D(A)$ are equivalent.

For each non-gentle vertex $v$ of Type~\textup{(I)} or
Type~\textup{(III)}, fix a local symmetry provided by
Condition~\ref{U1}\textup{(a1)} or~\ref{U2}, and use the
subquivers $\Sigma(v)$ and $C(v)$ defined in Remark \ref{rem:complementary_bound_subquiver}.
If $v$ is of Type~\textup{(I)}, then
Condition~\ref{U1}\textup{(a2)} states that, after removing $v$
from the underlying unoriented quiver, the two branches forming
$\Sigma(v)$ and the remaining branches forming $C(v)$ lie in
distinct connected components.

Define a directed graph $G_D$ whose vertex set consists exactly of the Type~\textup{(I)} and Type~\textup{(III)} non‑gentle vertices in $Q_0$. For two distinct vertices $u,v$ of $G_D$, introduce a directed edge $u\longrightarrow v$ if and only if there exist $\delta,\delta'\in\mathfrak{D}(A)$ satisfying
\[
\delta(w)=\delta'(w)\quad\forall\,w\neq u,\qquad \delta(u)\neq\delta'(u),
\]
and, with the deletion choices at all vertices other than $u$ held fixed, replacing the local deletion choice $\delta(u)$ at $u$ by $\delta'(u)$ destroys the extendability of the prescribed local symmetry at $v$ to a global label‑preserving bound‑quiver isomorphism. Consequently, whenever no edge $u\to v$ exists in $G_D$, altering the deletion datum at $u$ cannot destroy the global extendability of the local symmetry at $v$, no matter which deletion choices are fixed at the remaining vertices.

Suppose that $u\to v$. For the deletion data $\delta,\delta'$
witnessing this dependency edge, the two local deletion choices
described in Remark \ref{remark:transform non-gentle} give
$\delta(u)\cap\delta'(u)=\varnothing.$
Hence $\delta(u)\cup\delta'(u)$ consists precisely of the
2-relations whose deletion status changes when the choice
at $u$ is switched. At least one 2-relation in this union
contains an arrow belonging to $\Sigma(v)$. Otherwise, replacing
$\delta(u)$ by $\delta'(u)$ would change neither $\Sigma(v)$ nor
the relations and labels carried by it. By Remark \ref{rem:complementary_bound_subquiver},
the local isomorphism at $v$ would still glue with the identity on
$C(v)$, contradicting $u\to v$.

Inspecting one by one the local connection patterns permitted by
Figure~\ref{non-gentle-sys-1}, and combining this inspection with \ref{S1}, \ref{U1}\textup{(a2)}, and \ref{U3}, yields the following local conclusion: every
dependency edge $u\to v$ is of exactly one of the following two types.

\textup{(i)} \textbf{Single-branch type.}
        The two symmetric branches at $u$ both enter the same
        symmetric branch at $v$; equivalently, as viewed from the
        side of $v$, the vertex $v$ is connected to $u$ through a
        non-symmetric branch at $u$. In this case,
        \begin{equation}\label{eq:2.1}
        \Sigma(u)\subsetneq\Sigma(v).
        \end{equation}

\textup{(ii)} \textbf{Double-branch type.}
        The two symmetric branches at $u$ are connected,
        respectively, to the two symmetric branches at $v$. In
        this case, the full subquiver $\Sigma(v)$ contains $u$,
        while $\Sigma(u)$ contains $v$.

There is no third connection pattern, since otherwise the
in-degree or out-degree of one of the relevant vertices would
exceed two. A direct inspection of the two local deletion choices
in Figure~\ref{non-gentle-sys-1} further shows that, in the
double-branch case, if $u$ is of Type~\textup{(I)}, then
simultaneously interchanging both pairs of branches preserves each
local deletion choice at $u$ and therefore does not produce a
dependency edge $u\to v$. If both $u$ and $v$ are of
Type~\textup{(III)}, then the two pairs of symmetric branches form
one of the configurations excluded by Condition~\ref{U3} and
displayed in Figure~\ref{two III vertices}. Thus, a genuine
double-branch dependency edge can only be of the form
\begin{equation}\label{eq:2.2}
\textup{(III)}\longrightarrow\textup{(I)}.
\end{equation}

Suppose, for a contradiction, that $G_D$ contains a directed
cycle, and choose one having the smallest possible number of
vertices:
\begin{equation}\label{eq:2.3}
v_1\longrightarrow v_2\longrightarrow\cdots
\longrightarrow v_m\longrightarrow v_1,
\qquad m\geq2.
\end{equation}
This cycle has no repeated vertices. If a consecutive segment of
the cycle consists of single-branch edges,
\[
x_0\longrightarrow x_1\longrightarrow\cdots
\longrightarrow x_r,
\]
then \eqref{eq:2.1} gives
\[
\Sigma(x_0)\subsetneq\Sigma(x_1)\subsetneq\cdots
\subsetneq\Sigma(x_r).
\]
It therefore remains to distinguish three cases according to the
number of double-branch edges on the cycle.

If the cycle has no double-branch edge, then \eqref{eq:2.1} gives a closed
chain of strict inclusions, a contradiction. Suppose next that the
cycle has exactly one double-branch edge $u\to v$. By \eqref{eq:2.2}, $u$ is
of Type~\textup{(III)} and $v$ is of Type~\textup{(I)}. The two
branch connections corresponding to this double-branch edge reach
$u$ through the two branches contained in $\Sigma(v)$, whereas all
remaining edges on the cycle are of single-branch type. By the
local form of a single-branch edge, the connected branches
determined by these remaining edges return to $u$ through the
remaining branches contained in $C(v)$. Hence, in
$Q\setminus\{v\}$, the branches forming $\Sigma(v)$ remain
connected to the remaining branches forming $C(v)$, contrary to
Condition~\ref{U1}\textup{(a2)}.

Finally, suppose that the cycle has at least two double-branch
edges. Choose, in the directed order of the cycle, two such edges
$u_1\longrightarrow v_1$ and
$u_2\longrightarrow v_2$
so that the segment of the cycle from $v_1$ to $u_2$ contains no
other double-branch edge. Every edge on this segment is therefore
of single-branch type. By \eqref{eq:2.2}, $u_1$ and $u_2$ are of
Type~\textup{(III)}, whereas $v_1$ and $v_2$ are of
Type~\textup{(I)}. Trace, along the single-branch edges from $v_1$
to $u_2$, the corresponding connected branches in
Figure~\ref{non-gentle-sys-1}. If the two paths obtained by
starting from the two symmetric branches have a common vertex
other than their endpoints, then, in $Q\setminus\{v_1\}$, the
branches forming $\Sigma(v_1)$ remain connected to the remaining
branches forming $C(v_1)$, contrary to
Condition~\ref{U1}\textup{(a2)}. If the two paths have no common
vertex other than their endpoints, then they connect the two
Type~\textup{(III)} vertices $u_1$ and $u_2$. By
Condition~\ref{S1}, the four possible orientations at the two ends
are precisely the four configurations displayed in
Figure~\ref{two III vertices} and excluded by
Condition~\ref{U3}. All three cases are impossible. Therefore,
$G_D$ contains no directed cycle; that is, $G_D$ is a finite
directed acyclic graph.

We proceed by induction on the total number $n$ of  non-gentle vertices of Type (I) and Type (III) in $Q$.
If $n=0$, there is no deletion choice to make, and the assertion is
immediate. Let $n\geq 1$, and assume that the assertion holds when
the number of non-gentle vertices of Type~\textup{(I)} and
Type~\textup{(III)} under consideration is smaller than $n$.

Since $G_D$ is a finite directed acyclic graph, it has a sink,
that is, a vertex from which no directed edge starts. Fix such a
sink and denote it by $v^*$. Notice that $v^*$ is an actual vertex
of $Q$, so $v^*\in Q_0$. Making a deletion choice at $v^*$ changes
only the specified 2-relations passing through $v^*$ and
does not change the underlying quiver $Q$.

By the existential quantifier in the definition of a dependency
edge and its negation, the fact that $v^*$ is a sink means that,
regardless of which deletion choices have already been fixed at
the other vertices, switching the choice at $v^*$ does not destroy
the global extendability of the local symmetry at any remaining
non-gentle vertex. The connectivity condition in
Condition~\ref{U1}\textup{(a2)} is unchanged because $Q$ is
unchanged, and the underlying local configurations forbidden by
Condition~\ref{U3} cannot be created by this deletion. Therefore,
after either choice is made at $v^*$, the remaining $n-1$
non-gentle vertices still satisfy
Conditions~\ref{U1}--\ref{U3}.

Let $A_1=\mathbf{k}Q/I_1$ and $A_2=\mathbf{k}Q/I_2$
be the two intermediate string algebras arising from the two
choices at $v^*$. Condition~\ref{U1} or~\ref{U2} provides a local
label-preserving isomorphism interchanging these two choices. The
complementary bound subquivers on the two sides are identical, the
 relations crossing $v^*$ correspond under the local
symmetry, and the label data are preserved. Hence, by the
preceding remark, the local isomorphism glues to a global
label-preserving bound-quiver isomorphism
\begin{equation}\label{eq:2.4}
\Phi:(Q,I_1)\xrightarrow{\sim}(Q,I_2).
\end{equation}

In each of the two intermediate cases, only $n-1$ non-gentle
vertices remain to be treated. We may therefore apply the
induction hypothesis separately to both cases: all admissible
deletion data obtained by completing the remaining deletions in
$A_1$ are equivalent, and all admissible deletion data obtained by
completing the remaining deletions in $A_2$ are likewise
equivalent. Moreover, \eqref{eq:2.4} maps every remaining local
configuration, deletable relation, and label in $(Q,I_1)$ to the
corresponding local configuration, deletable relation, and label
in $(Q,I_2)$. Transporting the remaining deleted relations through
$\Phi$ therefore gives a bijection between the complete deletion
data on the two sides, and corresponding deletion data are
equivalent in the sense of
Definition~\ref{def:deletion_data_equiv}. Hence the two equivalence
classes obtained by induction are in fact the same equivalence
class.

Thus, $\delta\sim\delta'$ for all
$\delta,\delta'\in\mathfrak D(A)$. By
Lemma~\ref{Lem:geo-unique}, all
geometric models of $A$ are equivalent. Therefore, the geometric
model of $A$ is unique.
\end{proof}

\begin{remark}
Several remarks concerning Lemma \ref{three-conditions} are as follows:
\begin{enumerate}
    \item The quiver configuration illustrated in Figure \ref{non-gentle-sys-1} represents the most general case at each non-gentle vertex and its geometric realization is illustrated in Figure \ref{type II}.

    \item In condition \ref{U1}, $Q^{(1)}$ and $Q^{(2)}$, as well as $Q^{(3)}$ and $Q^{(4)}$, may share common vertices; in conditions \ref{U2} and \ref{U3}, $Q^{(1)}$ and $Q^{(2)}$ may share common vertices.

    \item Under Condition \ref{U1}:
    \begin{itemize}
        \item If only arrows $a$ and $c$ are identified, i.e., there exists a loop at vertex $v$, then vertex $v$ fails to satisfy conditions (a1) and (a2);
        \item If arrows $a$ and $c$ are identified, and arrows $b$ and $d$ are identified, i.e., there exist two loops at vertex $v$, then vertex $v$ fails to satisfy condition (a2).
    \end{itemize}

    \item Under Condition \ref{U2}, if arrows $a$ and $c$ (or $b$ and $c$) are identified, i.e., there exists a loop at vertex $v$, then vertex $v$ fails to satisfy Condition \ref{U2}.
\end{enumerate}
\begin{figure}[htbp]
	\begin{center}
    \begin{tikzpicture}[>=stealth, scale=0.8]
	\draw[line width=1pt, red] (0,2) --(0,-2);
	\draw[line width=2pt, red, dash dot] (-0.8,1.5) --(-1.5,2);
	\draw[line width=2pt, red, dash dot] (0.8,1.5) --(1.5,2);
	\draw[line width=2pt, red, dash dot] (-0.8,-1.5) --(-1.5,-2);
	\draw[line width=2pt, red, dash dot] (0.8,-1.5) --(1.5,-2);
	
	\draw[line width=1pt, red] (-2,1) to[out=0, in=-120](0,2)to[out=-60, in=180](2,1);
	\draw[line width=1pt, red] (-2,-1) to[out=0, in=120](0,-2)to[out=60, in=180](2,-1);
	
	\draw[blackarrow](0.4,1.5)to(0,1.5);
	\node [] at (0.2,1.3) {$\mathbf{a}$};	

	\draw[blackarrow](0,1.5)to(-0.4,1.5);
	\node [] at (-0.2,1.3) {$\mathbf{d}$};

	\draw[blackarrow](-0.4,-1.5)to(0,-1.5);
	\node [] at (-0.2,-1.25) {$\mathbf{b}$};	

	\draw[blackarrow](0,-1.5)to(0.4,-1.5);
	\node [] at (0.2,-1.3) {$\mathbf{c}$};
	
	\draw[bend left,bluearrow](0.7,1.3)to(-0.7,1.3);
	\draw[bend left,bluearrow](-0.7,-1.3)to(0.7,-1.3);
	
	\draw[red,thick,fill=red] (0,2) circle (0.07);
	\draw[red,thick,fill=red] (0,-2) circle (0.07);
	
	\node [purple] at (2,2) {$Q^{(1)}$};
	\node [purple] at (-2,-2) {$Q^{(2)}$};
	\node [purple] at (-2,2) {$Q^{(4)}$};
	\node [purple] at (2,-2) {$Q^{(3)}$};
\end{tikzpicture}
	\begin{tikzpicture}[>=stealth, scale=0.8]
	\draw[line width=1pt, red] (0,2) --(0,-2);
	\draw[line width=2pt, red, dash dot] (-0.8,1.5) --(-1.5,2);
	\draw[line width=2pt, red, dash dot] (0.8,1.5) --(1.5,2);
	\draw[line width=2pt, red, dash dot] (-0.8,-1.5) --(-1.5,-2);
	\draw[line width=2pt, red, dash dot] (0.8,-1.5) --(1.5,-2);
	
	\draw[line width=1pt, red] (-2,1) to[out=0, in=-120](0,2)to[out=-60, in=180](2,1);
	\draw[line width=1pt, red] (-2,-1) to[out=0, in=120](0,-2)to[out=60, in=180](2,-1);
	
	\draw[blackarrow](0.4,1.5)to(0,1.5);
	\node [] at (0.2,1.3) {$\mathbf{a}$};	

	\draw[blackarrow](0,1.5)to(-0.4,1.5);
	\node [] at (-0.2,1.3) {$\mathbf{d}$};

	\draw[blackarrow](-0.4,-1.5)to(0,-1.5);
	\node [] at (-0.2,-1.25) {$\mathbf{b}$};	

	\draw[blackarrow](0,-1.5)to(0.4,-1.5);
	\node [] at (0.2,-1.3) {$\mathbf{c}$};
	
	\draw[bend left,bluearrow](0.7,1.3)to(-0.7,1.3);
	
	\draw[red,thick,fill=red] (0,2) circle (0.07);
	\draw[red,thick,fill=red] (0,-2) circle (0.07);
	
	\node [purple] at (2,2) {$Q^{(1)}$};
	\node [purple] at (-2,-2) {$Q^{(2)}$};
	\node [purple] at (-2,2) {$Q^{(4)}$};
	\node [purple] at (2,-2) {$Q^{(3)}$};
\end{tikzpicture}
\begin{tikzpicture}[>=stealth, scale=0.8]
	\draw[line width=1pt, red] (0,2) --(0,-2);
	\draw[line width=2pt, red, dash dot] (0.8,1.5) --(1.5,2);
	\draw[line width=2pt, red, dash dot] (-0.8,-1.5) --(-1.5,-2);
	\draw[line width=2pt, red, dash dot] (0.8,-1.5) --(1.5,-2);
	
	\draw[, line width=1pt, red] (0,2)to[out=-60, in=180](2,1);
	\draw[, line width=1pt, red] (-2,-1) to[out=0, in=120](0,-2)to[out=60, in=180](2,-1);
	
	\draw[blackarrow](0.4,1.5)to(0,1.5);
	\node [] at (0.3,1.2) {$\mathbf{a}$};	

	\draw[blackarrow](-0.4,-1.5)to(0,-1.5);
	\node [] at (-0.2,-1.25) {$\mathbf{b}$};	

	\draw[blackarrow](0,-1.5)to(0.4,-1.5);
	\node [] at (0.2,-1.3) {$\mathbf{c}$};
	
	\draw[bend left, bluearrow](-0.7,-1.3)to(0.7,-1.3);
	
	\draw[red,thick,fill=red] (0,2) circle (0.07);
	\draw[red,thick,fill=red] (0,-2) circle (0.07);
	
	\node [purple] at (2,2) {$Q^{(1)}$};
	\node [purple] at (-2,-2) {$Q^{(2)}$};
	\node [purple] at (2,-2) {$Q^{(3)}$};
\end{tikzpicture}
		
	\end{center}
	\caption{Geometric realization of non-gentle vertices.}
	\label{type II}
\end{figure}
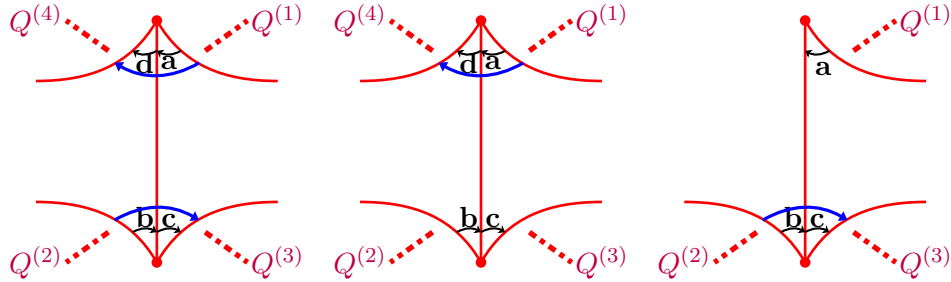
\end{remark}

Conversely, the sufficient conditions given in Lemma \ref{three-conditions} are also necessary for the uniqueness of the geometric model. We prove this by showing that when any one of these conditions is not met, the geometric model obtained will not be unique.

\begin{lemma}\label{reverse of three conditions}
Let \(A=\mathbf{k}Q/I\) be a string algebra containing non-gentle
vertices of Type~\textup{(I)} and/or Type~\textup{(III)}. If the
geometric model of \(A\) is unique, then the non-gentle vertices of
\(A\) satisfy Conditions~\ref{U1}, \ref{U2}, and~\ref{U3} of
Lemma~\ref{three-conditions}.
\end{lemma}

\begin{proof}
We argue by contraposition. Suppose that at least one of
Conditions~\ref{U1}, \ref{U2}, and~\ref{U3} fails. We shall prove
that there exist two inequivalent admissible deletion data $\delta,\delta'\in\mathfrak D(A)$.
By the preceding classification of admissible local deletions at
non-gentle vertices, each local deletion choice made below can be
combined with fixed deletion choices at all the remaining non-gentle
vertices. Thus, all the deletion data constructed below belong to
\(\mathfrak D(A)\) and determine saturated locally gentle covers
$(Q,J_{\delta},\mathcal L_{\delta})$ and $(Q,J_{\delta'},\mathcal L_{\delta'})$.

We first specify the local invariant used in the proof. Consider the
distinguished full branches appearing in Figure~\ref{non-gentle-sys-1}
or Figure~\ref{two III vertices}. Their \emph{permitted-connection
graph} is the undirected graph whose vertices are these full branches
and whose edges are the local permitted connections arising from
non-relations in the relevant locally gentle ideal, together with the
fixed exterior permitted connections. Two branches will be called
\emph{permitted-connected} if the corresponding vertices lie in the
same connected component of this graph. Orientations are ignored in
this auxiliary graph. Thus, the orientation of a wavy arrow denoting
a permitted connection is purely schematic and may be reversed. In
the configuration of Figure~\ref{two III vertices}, its actual
orientation depends on whether each of the vertices \(v_1\) and
\(v_2\) is of Type~\textup{(III.a)} or Type~\textup{(III.b)}; this
does not affect the underlying undirected permitted connection.

A bound-quiver isomorphism preserves both relations and
non-relations. It therefore induces an isomorphism between the
corresponding permitted-connection graphs and preserves their
partitions into connected components. Moreover, by the
label-preserving convention in
Definition~\ref{def:deletion_data_equiv} and
Remark~\ref{rem:label-perserve}, once the deletion choices and label
data outside the local configuration have been fixed, every global
label-preserving bound-quiver isomorphism must be compatible with
these fixed data and hence induce an isomorphism of the local
permitted-connection graphs.

\smallskip
\noindent
\emph{Case 1: Condition~\ref{U1}\textup{(a1)} or
Condition~\ref{U2} fails.}

There exists a non-gentle vertex \(v\) of Type~\textup{(I)} or
Type~\textup{(III)} at which the local label-preserving symmetry
required in Lemma~\ref{three-conditions} does not exist. Choose
\(\delta,\delta'\in\mathfrak D(A)\) so that they make different local
deletion choices only at \(v\), while their deletion choices at every
other non-gentle vertex coincide.

By the classification of local deletions at vertices of
Types~\textup{(I)} and~\textup{(III)}, any label-preserving
isomorphism between the two labelled local bound quivers determined
by these choices would give precisely the local label-preserving
symmetry required by Condition~\ref{U1}\textup{(a1)} or
Condition~\ref{U2}. This contradicts the present assumption. Hence
the two labelled local bound quivers are not isomorphic.

If \(\delta\sim\delta'\), then, by
Definition~\ref{def:deletion_data_equiv}, there would exist a global
label-preserving bound-quiver isomorphism
\[
  \varphi\colon
  (Q,J_{\delta},\mathcal L_{\delta})
  \xrightarrow{\sim}
  (Q,J_{\delta'},\mathcal L_{\delta'}).
\]
Since \(\delta\) and \(\delta'\) agree at every non-gentle vertex
other than \(v\), the label-preserving convention forces the
restriction of \(\varphi\) to the local configuration at \(v\) to
induce a label-preserving isomorphism between the two local bound
quivers, a contradiction. Therefore, $\delta\not\sim\delta'$.

\smallskip
\noindent
\emph{Case 2: Condition~\ref{U1}\textup{(a2)} fails.}

If Condition~\ref{U1}\textup{(a1)} also fails, the conclusion follows
from Case~1. We may therefore assume that \(v\) satisfies
Condition~\ref{U1}\textup{(a1)} but fails
Condition~\ref{U1}\textup{(a2)}. Condition~\ref{U1}\textup{(a1)}
provides the local label-preserving symmetry of one of the two pairs
of branches in Figure~\ref{non-gentle-sys-1}. Since
Condition~\ref{U1}\textup{(a2)} fails, the two pairs of branches are
not separated in \(Q\setminus\{v\}\).

Checking the local connections allowed by
Figure~\ref{non-gentle-sys-1}, together with the degree restriction
in Condition~\ref{S1}, shows that the given symmetry propagates along
the exterior connections and determines the corresponding pairing of
the other two branches. After relabelling the symmetric branches if
necessary, there are two possible symmetric exterior pairings:
\[
\begin{array}{ll}
\textup{(A)} &
Q^{(1)}\rightsquigarrow Q^{(3)},\qquad
Q^{(2)}\rightsquigarrow Q^{(4)};\\[1.5mm]
\textup{(B)} &
Q^{(1)}\rightsquigarrow Q^{(4)},\qquad
Q^{(2)}\rightsquigarrow Q^{(3)}.
\end{array}
\]
Interchanging the indices \(3\) and \(4\) transforms one case into
the other, so it is enough to consider case~\textup{(A)}. Denote the
two fixed exterior permitted connections by
\[
  P_{31}\colon Q^{(3)}\rightsquigarrow Q^{(1)},
  \qquad
  P_{42}\colon Q^{(4)}\rightsquigarrow Q^{(2)}.
\]

Fix the deletion choices at all the other non-gentle vertices, and
choose
\[
  \delta_{(v)}=\{ca,bd\},
  \qquad
  \delta'_{(v)}=\{cd,ba\}.
\]
Here \(\delta_{(v)}\) and \(\delta'_{(v)}\) denote the 2-relations deleted from \(I\) at \(v\). Consequently,
\[
\begin{aligned}
  &ca,bd\notin J_{\delta},
  &&cd,ba\in J_{\delta},\\
  &cd,ba\notin J_{\delta'},
  &&ca,bd\in J_{\delta'}.
\end{aligned}
\]

For \(J_{\delta}\), the fixed exterior connections and the local
permitted connections at \(v\), arising respectively from the
non-relations \(ca\) and \(bd\), are
\[
  E_{\mathrm{ext}}=\bigl\{\{1,3\},\{2,4\}\bigr\},
  \qquad
  E_{\mathrm{loc}}^{\delta}
  =\bigl\{\{1,3\},\{2,4\}\bigr\}.
\]
Hence the corresponding permitted-connection graph has precisely the
two connected components
$\bigl\{Q^{(1)},Q^{(3)}\bigr\}$,
$\bigl\{Q^{(2)},Q^{(4)}\bigr\}$.
Equivalently, its connections are
\[
  Q^{(1)}\xrightarrow{\,ca\,}Q^{(3)}
  \xrightarrow{\,P_{31}\,}Q^{(1)},
  \qquad
  Q^{(2)}\xrightarrow{\,bd\,}Q^{(4)}
  \xrightarrow{\,P_{42}\,}Q^{(2)}.
\]

For \(J_{\delta'}\), the local permitted connections at \(v\),
arising from the non-relations \(cd\) and \(ba\), are
\[
  E_{\mathrm{loc}}^{\delta'}
  =\bigl\{\{1,4\},\{2,3\}\bigr\}.
\]
Together with the fixed exterior connections
\(E_{\mathrm{ext}}\), they place all four vertices in a single
connected component. Explicitly, the connection is
\[
  Q^{(1)}\xrightarrow{\,cd\,}Q^{(4)}
  \xrightarrow{\,P_{42}\,}Q^{(2)}
  \xrightarrow{\,ba\,}Q^{(3)}
  \xrightarrow{\,P_{31}\,}Q^{(1)}.
\]

Thus, the local permitted-connection graphs determined by the two
deletion choices have two and one connected components,
respectively. On the other hand, \(\delta\) and \(\delta'\) agree
away from \(v\), so all relations and label data outside the local
configuration remain fixed. If \(\delta\sim\delta'\), the resulting
global label-preserving bound-quiver isomorphism would induce an
isomorphism between these two local permitted-connection graphs.
This is impossible because a graph isomorphism preserves the number
of connected components. Hence, $\delta\not\sim\delta'$.

\smallskip
\noindent
\emph{Case 3: Condition~\ref{U3} fails.}

In this case, \((Q,I)\) contains two interacting non-gentle vertices
\(v_1\) and \(v_2\) of Type~\textup{(III)}, in one of the
configurations shown in Figure~\ref{two III vertices}. Fix the
deletion choices at all the remaining non-gentle vertices, and take
\[
  \delta_{(v_1)}=\{ca\},
  \qquad
  \delta'_{(v_1)}=\{ba\},
  \qquad
  \delta_{(v_2)}=\delta'_{(v_2)}=\{ef\}.
\]
It follows that
\[
\begin{aligned}
  &ca,ef\notin J_{\delta},
  &&ba,df\in J_{\delta},\\
  &ba,ef\notin J_{\delta'},
  &&ca,df\in J_{\delta'}.
\end{aligned}
\]

Let \(P_{bd}\) and \(P_{ce}\) denote the two fixed branches in
Figure~\ref{two III vertices} joining \(b\) to \(d\) and \(c\) to
\(e\), respectively, and let \(B_a\) and \(B_f\) denote the two
exterior branches determined by the arrows \(a\) and \(f\). In
\(J_{\delta}\), the two-paths \(ca\) and \(ef\) are permitted, so
the local permitted-connection graph contains the edges
\[
  \{B_a,P_{ce}\},
  \qquad
  \{P_{ce},B_f\}.
\]
Thus, its connected-component partition is
\[
  \bigl\{B_a,P_{ce},B_f\bigr\}
  \mathbin{\sqcup}
  \bigl\{P_{bd}\bigr\};
\]
in particular, its component sizes are \(3+1\). For the representative
arrow orientation displayed in Figure~\ref{two III vertices}, the
corresponding permitted paths are
\[
  B_a\xleftarrow{\,ca\,}P_{ce}
  \xrightarrow{\,ef\,}B_f.
\]
If either \(v_1\) or \(v_2\) is of the other type among
\textup{(III.a)} and \textup{(III.b)}, the corresponding arrow is
reversed. This does not affect the argument, which uses only the
induced undirected connections.

In \(J_{\delta'}\), the two-paths \(ba\) and \(ef\) are permitted,
and the local permitted-connection graph therefore contains the
edges
\[
  \{B_a,P_{bd}\},
  \qquad
  \{P_{ce},B_f\}.
\]
By Conditions~\ref{S1}, \ref{S2}, and~\ref{G1}, these are the only
permitted connections among the four distinguished full branches.
Thus, its connected-component partition is
\[
  \bigl\{B_a,P_{bd}\bigr\}
  \mathbin{\sqcup}
  \bigl\{P_{ce},B_f\bigr\},
\]
whose component sizes are \(2+2\). The two local
permitted-connection graphs are therefore not isomorphic.

Furthermore, \(\delta\) and \(\delta'\) both delete \(ef\) at
\(v_2\), and they make identical deletion choices at every other
non-gentle vertex. Hence the ordered label support corresponding to
\(ef\), as well as all relations and label data outside the local
configuration, remains fixed. A label-preserving bound-quiver
isomorphism must preserve this fixed label support. Since \(df\) is
the other relation retained at \(v_2\), the uniqueness of
path imposed by the string conditions also fixes the
relative positions of \(P_{bd}\), \(P_{ce}\), \(B_a\), and \(B_f\).
Therefore, every global label-preserving bound-quiver isomorphism
would induce an isomorphism between the two local
permitted-connection graphs, which is impossible by the preceding
comparison. Thus, $\delta\not\sim\delta'$.
The other three arrow orientations represented in
Figure~\ref{two III vertices} are treated identically after making
the corresponding substitutions of arrows and relations.

We have shown that whenever at least one of
Conditions~\ref{U1}, \ref{U2}, and~\ref{U3} fails, there exist two
inequivalent admissible deletion data. By
Lemma~\ref{Lem:geo-unique}, these deletion data determine two
inequivalent saturated  labelled geometric models in the
sense of Definition~\ref{def:equivalence_models}. Hence, if the
geometric model of \(A\) is unique, every non-gentle vertex of
Type~\textup{(I)} or Type~\textup{(III)} must satisfy
Conditions~\ref{U1}, \ref{U2}, and~\ref{U3}.
\end{proof}

We can directly obtain the following result and the proof of Theorem \ref{Mtheorem:object}  is complete.

\begin{theorem}\label{main-theorem}
Let $A$ be a string algebra. The geometric model of $A$ is unique if and only if every non-gentle vertex of Type $\mathrm{(I)}$ or Type $\mathrm{(III)}$ (if any) satisfies the three conditions of Lemma \ref{three-conditions}.
\end{theorem}

\begin{proof}
The sufficiency follows directly from Lemma \ref{lem:non-II} and Lemma \ref{three-conditions}. Conversely, the necessity is immediately derived from Lemma \ref{reverse of three conditions}. This completes the proof.
\end{proof}

\section{String algebras over geometric models without red punctures}
\label{section:String algebras over geometric models of red punctures}
In this section, we show exactly when the geometric models of a string algebra have no red punctures, and when they always have them.

The red punctures appearing in Definition \ref{def:marked_surface} are in bijection with
infinite cyclic straight walks, see \cite[Remark~4.11(i)]{PPP19}. Within the framework of saturated labelled tiled surface models adopted in this paper, this correspondence is realized as a bijection between red punctures and \(J\)-permitted cycles, where \(J\) is a saturated locally gentle cover of \(A\) in the sense of
Definition~\ref{def:saturated_model}.

 Since the geometric model of a gentle algebra must contain no red punctures, it follows that for every geometric model of a string algebra to be free of red punctures, the algebra obtained after transforming all non-gentle vertices must be a gentle algebra.
\subsection{Gentle-bounded cycles}

 According to \cite{PPP19}, the red punctures in the geometric model of the locally gentle algebra are given by the corresponding cycles in the quiver, see Figure \ref{fig:cycle}. Consequently, only the cases of cycles need to be considered.

 \begin{figure}[htbp]
 	\begin{center}
 		\begin{tikzpicture}[>=stealth, scale=1]
 			\path 
 			(-5,1.5)  coordinate (1)
 			(-4,2) coordinate (2)
 			(-3,1.5) coordinate (3)
 			(-4,0)  coordinate (4)
 			(-6,-1.5) coordinate (5)
 			(-2,-1.5) coordinate (6)
 			(2,0)  coordinate (7)
 			(3.5,2) coordinate (8)
 			(5,0) coordinate (9)
 			(2.5,-2) coordinate (10)
 			(4.5,-2) coordinate (11)
 			;
 			
 			\draw 
 			(1) node[red] {$\bullet$}
 			(2) node[red] {$\bullet$}
 			(3) node[red] {$\bullet$}
 			(4) node[red] {$\bullet$}
 			(5) node[red] {$\bullet$}
 			(6) node[red] {$\bullet$}
 			(7) node[blue] {$v_{1}$}
 			(8) node[blue] {$v_{2}$}
 			(9) node[blue] {$v_{3}$}
 			(10) node[blue] {$v_{n}$}
 			(11) node[blue] {$v_{4}$}
 			;
 			
 			\draw[line width=1pt,red ] (4) to(1);
 			\draw[line width=1pt,red ] (4) to(2);
 			\draw[line width=1pt,red ] (4) to(3);
 			\draw[line width=1pt,red ] (4) to(5);
 			\draw[line width=1pt,red ] (4) to(6);
 			
 			\draw[line width=1pt,black,->] (2.2,0.2) to(3.3,1.8);
 			\draw[line width=1pt,black,->] (3.7,1.8) to(4.8,0.2);
 			\draw[line width=1pt,black,->] (5,-0.2) to(4.5,-1.8);
 			\draw[line width=1pt,black,->] (2.5,-1.8) to(2,-0.2);
            \draw[line width=1pt,black,->,dashed] (4.2,-2) to(2.8,-2);

 			\draw (-4.6,0.5) node[black,font=\tiny] {$\bf{x_{1}}$}; 
 			\draw (-4.25,1.2) node[black,font=\tiny] {$\bf{x_{2}}$}; 
 			\draw (-3.3,0.5) node[black,font=\tiny] {$\bf{x_{3}}$}; 
 			\draw (-2.7,-0.7) node[black,font=\tiny] {$\bf{x_{4}}$};
 			\draw (-5.3,-0.7) node[black,font=\tiny] {$\bf{x_{n}}$};
 			\draw (-4,-1) node[red] {\huge$\cdots$};

 			\draw (-4,-2.5) node {$\mathcal{S} = (\mathcal{S}, \mathsf{P})$}; 
 			\draw (4,-2.5) node {$Q = Q(\mathcal{S}, \mathsf{P})$};

 		\end{tikzpicture}
 	\end{center}
 	\caption{A red puncture in the geometric model arises from a cycle in the associated quiver $Q$.}
 	\label{fig:cycle}
 \end{figure}
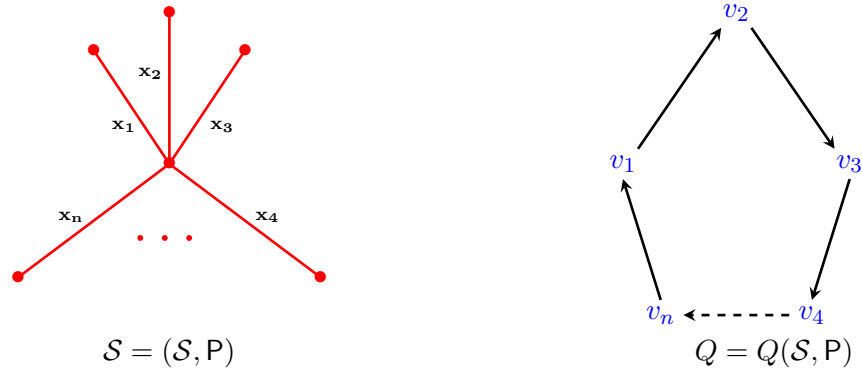
 
We will introduce a class of special cycles in the following definition.

\begin{definition}\label{def:gentle-bounded}
    Let $A = \mathbf{k}Q/I$ be a string algebra. A primitive oriented cycle $\mathcal{C} = a_1 a_2 \dots a_n$ in $Q$ is called \textit{gentle-bounded} if there exists at least one index $i$ (with $1 \le i \le n$ and indices taken modulo $n$) such that $a_{i-1}a_i \in I$, and the vertex $v_i = t(a_{i-1}) = s(a_i)$ satisfies one of the following two conditions:
    \begin{enumerate}[label=(GB\arabic*), ref=(GB\arabic*)]
    \item $v_i$ is a gentle vertex; \label{GB1}
    
    \item $v_i$ is a non-gentle vertex of Type (II), and $a_{i-1}a_i$ is not the deleted relation when transforming $v_i$ into a gentle vertex, see Figure \ref{fig:gentle-bounded}. \label{GB2}
\end{enumerate}
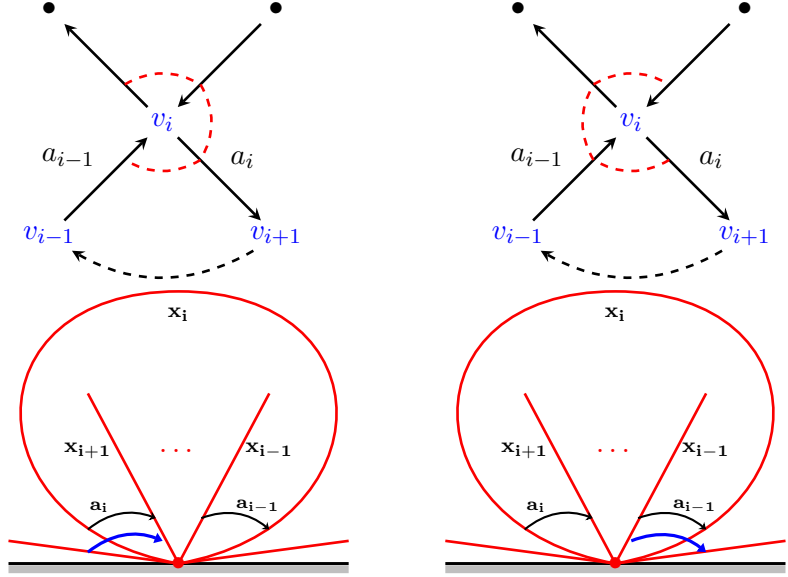
\begin{figure}[htbp]
    \centering
    \begin{tikzpicture}[>=stealth, scale=1]
        \path 
        (-1.5,1.5)  coordinate (1)
        (1.5,1.5) coordinate (2)
        (0,0) coordinate (3)
        (-1.5,-1.5)  coordinate (4)
        (1.5,-1.5) coordinate (5)
        ;
        
        \draw 
        (1) node[] {$\bullet$}
        (2) node[] {$\bullet$}
        (3) node[blue] {$v_{i}$}
        (4) node[blue] {$v_{i-1}$}
        (5) node[blue] {$v_{i+1}$}
        ;
        
        \draw[line width=1pt,->] (-0.2,0.2) to(-1.3,1.3);
        \draw[line width=1pt,->] (1.3,1.3) to(0.2,0.2);
        
        \draw[line width=1pt,->] (-1.3,-1.3) to node[above left] {$a_{i-1}$} (-0.2,-0.2);
        \draw[line width=1pt,->] (0.2,-0.2) to node[above right] {$a_{i}$} (1.3,-1.3);

\draw[bend left,line width=1pt, dashed, ->] (1.2,-1.7) to (-1.2,-1.7);

\draw[bend left,line width=1pt, dashed, red] (0.5,-0.5) to (-0.5,-0.5);
\draw[bend left,line width=1pt, dashed, red] (-0.5,0.5) to (0.5,0.5);
\draw[bend left,line width=1pt, dashed, red] (0.5,0.5) to (0.5,-0.5); 
\end{tikzpicture}
    \hspace{2cm}
\begin{tikzpicture}[>=stealth, scale=1]
        \path 
        (-1.5,1.5)  coordinate (1)
        (1.5,1.5) coordinate (2)
        (0,0) coordinate (3)
        (-1.5,-1.5)  coordinate (4)
        (1.5,-1.5) coordinate (5)
        ;   
        \draw 
        (1) node[] {$\bullet$}
        (2) node[] {$\bullet$}
        (3) node[blue] {$v_{i}$}
        (4) node[blue] {$v_{i-1}$}
        (5) node[blue] {$v_{i+1}$}
        ;
        
        \draw[line width=1pt,->] (-0.2,0.2) to(-1.3,1.3);
        \draw[line width=1pt,->] (1.3,1.3) to(0.2,0.2);
        
        \draw[line width=1pt,->] (-1.3,-1.3) to node[above left] {$a_{i-1}$} (-0.2,-0.2);
        \draw[line width=1pt,->] (0.2,-0.2) to node[above right] {$a_{i}$} (1.3,-1.3);

\draw[bend left,line width=1pt, dashed, ->] (1.2,-1.7) to (-1.2,-1.7);

\draw[bend left,line width=1pt, dashed, red] (0.5,-0.5) to (-0.5,-0.5);
\draw[bend left,line width=1pt, dashed, red] (-0.5,0.5) to (0.5,0.5); 
\draw[bend left,line width=1pt, dashed, red] (-0.5,-0.5) to (-0.5,0.5); 
\end{tikzpicture}  
\begin{tikzpicture}[>=stealth, scale=1.5]
        \draw[line width=3pt,gray!50] (-1.5,-0.05) to (1.5,-0.05);
        \draw[line width=1.1pt,black] (-1.5,0) to (1.5,0);

        \path 
        (-1.5,0.2])  coordinate (1)
        (1.5,0.2) coordinate (2)
        (0,0) coordinate (3)
        (-0.8,1.5)  coordinate (4)
        (0.8,1.5) coordinate (5)
        ;
        
        \draw 
        (3) node[] {$\rpoint$}
        (0,1) node[red] {$\dots$};
        \draw (0.8,1) node[black,font=\scriptsize] {$\bf{x_{i-1}}$};
        \draw (0,2.2) node[black,font=\scriptsize] {$\bf{x_{i}}$};
        \draw (-0.8,1) node[black,font=\scriptsize] {$\bf{x_{i+1}}$};

        \draw[line width=1pt, red] (0,0) to[out=170, in=180, looseness=2] (0,2.4) to[out=0, in=10, looseness=2] (0,0);
       \draw[ line width=1pt, red] (3) to(1);
        \draw[ line width=1pt, red] (3) to(2);
       \draw[line width=1pt, red] (3) to(4);
        \draw[line width=1pt, red] (3) to(5);
  
        \draw[blackarrow] (0.2,0.4) to (0.8,0.3);
        \draw[blackarrow] (-0.8,0.3) to (-0.2,0.4);
        
        \draw (0.7,0.5) node[] {\tiny$\mathbf{a_{i-1}}$};
        \draw (-0.7,0.5) node[] {\tiny$\mathbf{a_{i}}$};   
       \draw[bend left,bluearrow] (-0.8,0.1) to(-0.14,0.2);
\end{tikzpicture}
\begin{tikzpicture}[>=stealth, scale=1.5]
        \draw[line width=3pt,gray!50] (-1.5,-0.05) to (1.5,-0.05);
        \draw[line width=1.1pt,black] (-1.5,0) to (1.5,0);

        \path 
        (-1.5,0.2])  coordinate (1)
        (1.5,0.2) coordinate (2)
        (0,0) coordinate (3)
        (-0.8,1.5)  coordinate (4)
        (0.8,1.5) coordinate (5)
        ;
        
        \draw 
        (3) node[] {$\rpoint$}
        (0,1) node[red] {$\dots$};
        \draw (0.8,1) node[black,font=\scriptsize] {$\bf{x_{i-1}}$};
        \draw (0,2.2) node[black,font=\scriptsize] {$\bf{x_{i}}$};
        \draw (-0.8,1) node[black,font=\scriptsize] {$\bf{x_{i+1}}$};

        \draw[line width=1pt, red] (0,0) to[out=170, in=180, looseness=2] (0,2.4) to[out=0, in=10, looseness=2] (0,0);
       \draw[ line width=1pt, red] (3) to(1);
        \draw[ line width=1pt, red] (3) to(2);
       \draw[line width=1pt, red] (3) to(4);
        \draw[line width=1pt, red] (3) to(5);
  
        \draw[blackarrow] (0.2,0.4) to (0.8,0.3);
        \draw[blackarrow] (-0.8,0.3) to (-0.2,0.4);
        
        \draw (0.7,0.5) node[] {\tiny$\mathbf{a_{i-1}}$};
        \draw (-0.7,0.5) node[] {\tiny$\mathbf{a_{i}}$};   
       \draw[bend left,bluearrow] (0.14,0.2) to(0.8,0.1);

\end{tikzpicture}

    \caption{Exactly two situations for a gentle-bounded cycle at a  (II)-non-gentle vertex $v_i$. The top panels depict the quivers, while the bottom panels illustrate the corresponding local geometric models.}
    \label{fig:gentle-bounded}
\end{figure}
\end{definition}

\begin{example}
    Consider the cycle $abc$ in $Q$ of Example \ref{ex:string-gentle algebra}. It is gentle-bounded under the ideal $I_3$ but not under $I_1$ or $I_2$, because the required relation of length two on the cycle is completely absent in $I_2$, and factors through a non-gentle vertex in $I_1$.
\end{example}
\begin{remark}
    It is worth noting that if a cycle $\mathcal{C}$ in $Q$ is not gentle-bounded, it only guarantees that $\mathcal{C}$ corresponds to the complete fan of a red puncture in at least one geometric model of $A$. This does not mean that $\mathcal{C}$ will correspond to the complete fan of a red puncture in every geometric model  of $A$, see Example \ref{ex:cycle-every surface}.
\end{remark}

After defining gentle-bounded cycles, we now look at how they affect the geometric models. A cycle that is not gentle-bounded gives a red puncture in at least one geometric model of $A$. However, if every cycle is gentle-bounded, then all geometric models of $A$ will have no red punctures. We have the following proposition.

\begin{proposition}\label{pro:cycle}
 	Let $A = \mathbf{k}Q/I$ be a string algebra. If  every cycle in $Q$ is gentle-bounded, then every geometric model of $A$ contains no red punctures.
 \end{proposition}

\begin{proof}
Let $\delta \in \mathfrak{D}(A)$ be an arbitrary admissible deletion datum, and let $B_\delta = \mathbf{k}Q/J_\delta$ be the associated locally gentle algebra. We will show that $B_\delta$ is finite-dimensional by examining the permitted paths in its bound quiver.

Suppose that every primitive oriented cycle in $Q$ is gentle-bounded. Let $\mathcal{C} = a_1 a_2 \cdots a_n$ be an arbitrary primitive oriented cycle in $Q$. By Definition \ref{def:gentle-bounded}, there exists an index $i$ such that $a_{i-1}a_i \in I$, and the vertex $v_i = t(a_{i-1}) = s(a_i)$ satisfies either \ref{GB1} or \ref{GB2}. 
\begin{enumerate}
    \item If $v_i$ satisfies \ref{GB1}, it is a gentle vertex. The relations at gentle vertices are never removed during the construction of $B_\delta$; thus, $a_{i-1}a_i \in J_\delta$.
    \item If $v_i$ satisfies \ref{GB2}, it is a non-gentle vertex of Type (II), and $a_{i-1}a_i$ is strictly distinct from the relation chosen to be deleted at $v_i$. Thus, $a_{i-1}a_i \in J_\delta$.
\end{enumerate}

In either case, every primitive oriented cycle $\mathcal{C}$ contains at least one length-2 subpath that belongs to $J_\delta$. This implies that under any deletion datum $\delta$, there are no primitive permitted cycles in the bound quiver $(Q, J_\delta)$.

Since the quiver $Q$ is finite, and the ideal $J_\delta$ is generated by 2-relations, the total absence of primitive permitted cycles guarantees that there cannot exist permitted directed paths of arbitrary length in $(Q, J_\delta)$. Consequently, the locally gentle algebra $B_\delta = \mathbf{k}Q/J_\delta$ is finite-dimensional.

By the geometric correspondence established in Theorem \ref{string algebra--labelled tiling algebra}, the geometric model associated with a finite-dimensional locally gentle algebra contains no red punctures. Since $\delta$ was chosen arbitrarily, every geometric model of $A$ is free of red punctures.
\end{proof}

\begin{remark}
Note that a gentle-bounded cycle $\mathcal{C}$ cannot correspond to the complete fan of a red puncture, but it may still belong to one. In this case, the arrows of $\mathcal{C}$ only form a part of a complete fan, while the full complete fan of the red puncture must be given by another cycle that is not gentle-bounded, see Figure \ref{fig:cycle-complete fan}. 

\begin{figure}[htbp]
	\begin{center}
\begin{tikzpicture}[scale=0.45]
   \path 
    (140:6) coordinate (b1)
    (90:6)  coordinate (b2)
    (-50:6) coordinate (b3)
    (-140:6)coordinate (b4)
    (0:0) coordinate (b5)
    (25:2.5) coordinate (r5);
    
    \draw[thick,black, fill=white] 
    (r5) circle (0.15cm);
    
     \draw[thick,red, fill=red] 
    (b1) circle (0.15cm) (b2) circle (0.15cm)
    (b3) circle (0.15cm) (b4) circle (0.15cm)
     (b5) circle (0.15cm);

    \draw[line width=1pt,red] (b1) to (b5);
    \draw[line width=1pt,red] (b2) to (b5);
    \draw[line width=1pt,red] (b3) to (b5);
    \draw[line width=1pt,red] (b4) to (b5);
    \draw[line width=1pt,red]
    (b5) to[out=80,in=120] (4,2)
    to[out=-60,in=-30](b5);

    \draw(1,4) node[black] {$\mathbf{x_{i-1}}$};
    \draw(4.5,0) node[black] {$\mathbf{x_{i}}$};
   \draw(1.6,-3.7) node[black] {$\mathbf{x_{i+1}}$};
    \draw (-3,0) node[red] {\Huge\vdots};
 	\draw (-1.5,3) node[red] {\Huge\ldots};
 	\draw (0,-2.5) node[red] {\Huge\ldots};

    \draw (0,-7) node {$\mathcal{S} = (\mathcal{S}, \mathsf{P})$}; 
 	 
\end{tikzpicture}
\hspace{2cm}
\begin{tikzpicture}[>=stealth, scale=1]
        \path 
       
        (0,0) coordinate (3)
        (-1.5,-1.5)  coordinate (4)
        (1.5,-1.5) coordinate (5)
        ;
        
        \draw 
        (3) node[blue] {$v_{i}$}
        (4) node[blue] {$v_{i-1}$}
        (5) node[blue] {$v_{i+1}$}
        ;
 
        \draw[line width=1pt, ->] (-0.2,0.2) to[out=135, in=180] node[above left] {$\epsilon$} (0,1.5) to[out=0, in=45] (0.2,0.2);
   
        \draw[line width=1pt,->] (-1.3,-1.3) to node[above left] {$a_{i-1}$} (-0.2,-0.2);
        \draw[line width=1pt,->] (0.2,-0.2) to node[above right] {$a_{i}$} (1.3,-1.3);

\draw[bend left,line width=1pt, dashed, ->] (1.1,-1.7) to node[below] {$a_{i+1} \ldots a_{i-2}$}(-1.2,-1.7);

\draw[bend left,line width=1pt, dashed, red] (0.5,-0.5) to (-0.55,-0.5);
\draw[bend left,line width=1pt, dashed, red] (-0.4,0.6) to (0.45,0.6);
\draw (0,-3.6) node {$Q = Q(\mathcal{S}, \mathsf{P})$};
\end{tikzpicture}
	\end{center}
	\caption{ The cycle $\mathcal{C}_1 = a_1 \dots a_{i-1} a_i \dots a_n$ is gentle-bounded, which may form a fan on $\mathcal{S}$. In contrast, the cycle $\mathcal{C}_2 = a_1 \dots a_{i-1} \epsilon a_i \dots a_n$ is not gentle-bounded and forms a complete fan on $\mathcal{S}$.} \label{fig:cycle-complete fan}
\end{figure}
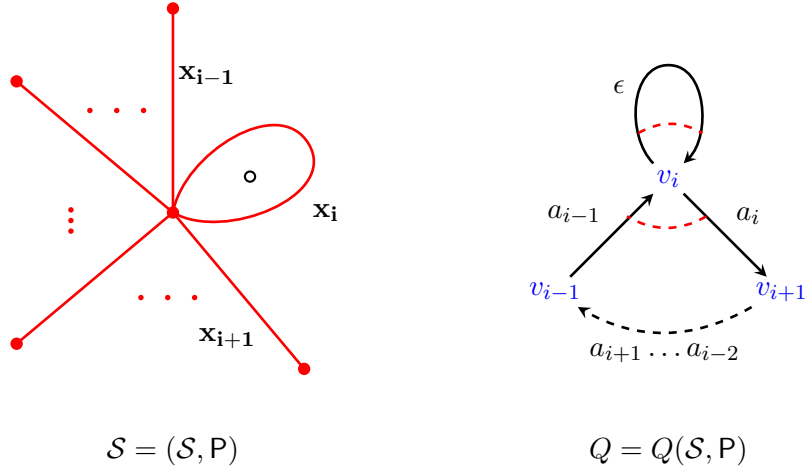

\end{remark}

\subsection{Geometric models without red punctures}
 
 In this section, we will characterize the string algebras without red punctures in all their surface models. 
 
 We first consider the case where the quiver is acyclic, which leads to the following lemma.

 \begin{lemma}\label{lem:acyclic}
 	Let $A = \mathbf{k}Q/I$ be a string algebra. If $Q$ is acyclic, then every geometric model of \( A \) contains no red punctures. 
 \end{lemma}
\begin{proof}
As established in \cite{PPP19}, every red puncture in the geometric model corresponds to a cycle in the quiver $Q$ (see Figure \ref{fig:cycle}). Consequently, if $Q$ is acyclic, the absence of cycles ensures that the geometric model of $A$ admits no red punctures.
\end{proof}

 Conversely, if all geometric models are free of red punctures, the cycles in the quiver must satisfy a specific condition, as follows.
 
\begin{lemma}\label{lem:no-puncture}
Let $A = \mathbf{k}Q/I$ be a string algebra where $Q$ is not acyclic. If every geometric model of $A$ contains no red punctures, then every primitive oriented cycle in the underlying quiver $Q$ is gentle-bounded.
\end{lemma}
\begin{proof}
Suppose, for a contradiction, that every geometric model of \(A\)
contains no red puncture, but that \(Q\) contains a primitive
oriented cycle which is not gentle-bounded. Choose such a cycle of
minimal length and write
$\mathcal{C}=a_1a_2\cdots a_n$.

We first show that \(\mathcal{C}\) contains no repeated arrow.
Otherwise, after cyclically reindexing \(\mathcal{C}\), choose
\(a_i=a_j\), with \(1\leq i<j\leq n\), so that \(j-i\) is minimal
among all positive cyclic distances between two occurrences of the
same arrow. Then
$W=a_i a_{i+1}\cdots a_{j-1}$
is an oriented closed walk, since $t(a_{j-1})=s(a_j)=s(a_i)$. 

Every cyclic length-two subpath of \(W\) occurs in
\(\mathcal{C}\); in particular, its closing subpath satisfies $a_{j-1}a_i=a_{j-1}a_j$.
Hence \(W\) contains no witness for \ref{GB1} or
\ref{GB2}. Moreover, \(W\) is primitive: otherwise,
\(W=U^m\) for some \(m>1\), and the first arrow of \(W\) reappears after the subwalk $U$ with $\ell(U)<\ell(W)$, contrary to the choice of \(W\). Thus \(W\) is a primitive oriented cycle which is not gentle-bounded and satisfies \(\ell(W)<\ell(\mathcal{C})\), contradicting the minimality of \(\mathcal{C}\). Therefore, \(\mathcal{C}\) contains no repeated arrow.

We now construct an admissible deletion datum
$\delta\in\mathfrak{D}(A)$.
First delete all generating relations of length at least three. It therefore remains only to specify the local quadratic deletions at
the non-gentle vertices.

At any Type~\textup{(II)} non-gentle vertex \(v_i\) on
\(\mathcal{C}\), take its unique local deletion. If
\(a_{i-1}a_i\in I\), then \(a_{i-1}a_i\) must be precisely the
relation deleted at \(v_i\); otherwise, Condition~\ref{GB2} would
make \(\mathcal{C}\) gentle-bounded. If
\(a_{i-1}a_i\notin I\), then
\(a_{i-1}a_i\notin J_\delta\) .

At any Type~\textup{(I)} or Type~\textup{(III)} non-gentle vertex
\(v_j\) on \(\mathcal{C}\), the local classification gives $a_{j-1}a_j\in I$, and we choose a valid local deletion containing this relation. At all remaining non-gentle vertices, choose an arbitrary valid local deletion.

It remains to check that these local requirements are compatible.
Since \(\mathcal{C}\) contains no repeated arrow and \ref{S1} bounds both the in-degree and the out-degree by
\(2\), the cycle \(\mathcal{C}\) traverses each vertex at most
twice. A \textup{(III)}-non-gentle vertex has a unique
incoming or outgoing arrow and therefore cannot be traversed
twice. If \(\mathcal{C}\) traverses a \textup{(I)}-non-gentle vertex twice, the two passages use distinct incoming
arrows \(a,b\) and distinct outgoing arrows \(c,d\); hence the two
relations to be deleted are either \{ac,bd\} or 
\{ad,bc\}, which is exactly one of the two valid local deletion choices. Loops are counted simultaneously as incoming and outgoing arrows,
and the same argument applies. Finally, a \textup{(II)}-non-gentle vertex has only one possible deleted relation, so no conflict can arise there.

Since every 2-relation has a unique middle vertex, the choices made at distinct non-gentle vertices are independent. By the preceding classification of admissible local deletions, these
compatible choices, together with the already fixed deletions of
all generating relations of length $\ge 3$, determine a well-defined admissible deletion datum $\delta\in\mathfrak{D}(A)$.

Let $B_\delta=\mathbf{k}Q/J_\delta$ be the associated locally gentle algebra. By construction, no length-2 subpath of \(\mathcal{C}\) passing through a non-gentle
vertex belongs to \(J_\delta\). If \(v_k\) is a gentle vertex on
\(\mathcal{C}\), then $a_{k-1}a_k\notin I$; otherwise, \ref{GB1} would make \(\mathcal{C}\)
gentle-bounded. Since \(J_\delta\subseteq I\), it follows that no
cyclic length-2 subpath of \(\mathcal{C}\) belongs to
\(J_\delta\). Thus \(\mathcal{C}\) is a primitive
\(J_\delta\)-permitted cycle.

By the correspondence between red punctures and cyclic
equivalence classes of primitive \(J_\delta\)-permitted cycles,
the geometric model determined by \(\delta\) contains a red
puncture. This contradicts the assumption that every geometric
model of \(A\) contains no red puncture. Therefore, every primitive
oriented cycle in \(Q\) is gentle-bounded.
\end{proof}

Combining the results established above, we are now ready to prove Theorem \ref{thm:B}.

\begin{theorem}\label{thm:no-puncture}
Let \(A=\mathbf{k}Q/I\) be a string algebra. Then all geometric models of \(A\) contain no red puncture if and only if every primitive oriented cycle in \(Q\), if any, is
gentle-bounded.
\end{theorem}
\begin{proof}
If \(Q\) is acyclic, the assertion follows from Lemma \ref{lem:acyclic}, since \(Q\) contains no primitive oriented cycle. Suppose that \(Q\) is not acyclic. Then the sufficiency follows from Proposition \ref{pro:cycle}, whereas the necessity follows from Lemma \ref{lem:no-puncture}. This completes the proof.
\end{proof}

\section{String algebras over geometric models with red punctures}\label{section:String algebras over geometric models with red punctures}
In this section, we provide sufficient conditions ensuring every geometric model of a string algebra contains a red puncture.
 
For a string algebra, all its geometric models contain red punctures if and only if the algebra obtained by transforming all non-gentle vertices is an infinite-dimensional locally gentle algebra. According to \cite{PPP19}, only the cases of cycles need to be considered.

\subsection{Essential cycles and Common red punctures}

In this subsection, we first introduce \emph{essential cycles}, which guarantee red punctures appear in all geometric models. We further introduce common red punctures and show they are exactly determined by essential cycles.

\subsubsection{Essential cycles}
To define this structure, we  initially introduce the concept of a stable vertex.

\begin{definition}\label{def:stable_vertex}

Let $A = \mathbf{k}Q/I$ be a string algebra and $\mathcal{C} = v_1 \xrightarrow{a_1} v_2 \xrightarrow{a_2} \cdots \xrightarrow{a_{n-1}} v_n \xrightarrow{a_n} v_1$ be a primitive oriented cycle in $Q$. A vertex $v_i$ is called a \emph{stable vertex} with respect to $\mathcal{C}$ if it satisfies one of the following conditions:
\begin{enumerate}[label=(E\arabic*), ref=(E\arabic*)]
    \item $v_i$ is a gentle vertex, and $a_{i-1}a_i \notin I$.  \label{E1}
    \item $v_i$ is a non-gentle vertex of Type (II) with $a_{i-1}a_i \in I$, having the local configuration shown in the left panel of Figure \ref{fig:red puncture}. \label{E2}
    \item $v_i$ is a non-gentle vertex of Type (II) with $a_{i-1}a_i \notin I$, having the local configuration shown in the right panel of Figure \ref{fig:red puncture}. \label{E3}
\end{enumerate}
\begin{figure}[htbp]
    \centering

\begin{tikzpicture}[>=stealth, scale=0.8] 
        \path 
        (-1.5,1.5)  coordinate (1)
        (1.5,1.5) coordinate (2)
        (0,0) coordinate (3)
        (-1.5,-1.5)  coordinate (4)
        (1.5,-1.5) coordinate (5)
        (0,-2.5) coordinate (6)
        (0,-1.25) coordinate (7)
        ;
        
        \draw 
        (1) node[] {$\bullet$}
        (2) node[] {$\bullet$}
        (3) node[blue] {$v_{i}$}
        (4) node[blue] {$v_{i-1}$}
        (5) node[blue] {$v_{i+1}$}
        (7) node[purple] {$\mathcal{C}$}
        ;
        
         \draw[line width=1pt,->] (-0.2,0.2) to node[above right] {$\mathbf{a}$} (-1.3,1.3);
\draw[line width=1pt,->] (1.3,1.3) to node[above left] {$\mathbf{b}$} (0.2,0.2);

\draw[line width=1pt,->] (-1.3,-1.3) to node[above left] {$\mathbf{c}$} (-0.2,-0.2);
\draw[line width=1pt,->] (0.2,-0.2) to node[above right] {$\mathbf{d}$} (1.3,-1.3);
        \draw[bend left,line width=1pt, dashed, ->] (1.2,-1.7) to (-1.2,-1.7);

\draw[bend left,line width=1pt, dashed, red] (0.5,-0.5) to (-0.5,-0.5);
\draw[bend left,line width=1pt, dashed, red] (0.5,0.5) to (0.5,-0.5); 
\draw[bend left,line width=1pt, dashed, red] (-0.5,-0.5) to (-0.5,0.5); 
\draw (0,-2.7) node {\ref{E2}};
\end{tikzpicture} 
\hspace{2cm}
\begin{tikzpicture}[>=stealth, scale=0.8] 
        \path 
        (-1.5,1.5)  coordinate (1)
        (1.5,1.5) coordinate (2)
        (0,0) coordinate (3)
        (-1.5,-1.5)  coordinate (4)
        (1.5,-1.5) coordinate (5)
        (0,-2.5) coordinate (6)
        (0,-1.25) coordinate (7)
        ;
        
        \draw 
        (1) node[] {$\bullet$}
        (2) node[] {$\bullet$}
        (3) node[blue] {$v_{i}$}
        (4) node[blue] {$v_{i-1}$}
        (5) node[blue] {$v_{i+1}$}
        (7) node[purple] {$\mathcal{C}$}
        ;
        
         \draw[line width=1pt,->] (-0.2,0.2) to node[above right] {$\mathbf{a}$} (-1.3,1.3);
\draw[line width=1pt,->] (1.3,1.3) to node[above left] {$\mathbf{b}$} (0.2,0.2);

\draw[line width=1pt,->] (-1.3,-1.3) to node[above left] {$\mathbf{c}$} (-0.2,-0.2);
\draw[line width=1pt,->] (0.2,-0.2) to node[above right] {$\mathbf{d}$} (1.3,-1.3);
        \draw[bend left,line width=1pt, dashed, ->] (1.2,-1.7) to (-1.2,-1.7);

\draw[bend left,line width=1pt, dashed, red] (-0.5,0.5) to (0.5,0.5);
\draw[bend left,line width=1pt, dashed, red] (0.5,0.5) to (0.5,-0.5); 
\draw[bend left,line width=1pt, dashed, red] (-0.5,-0.5) to (-0.5,0.5); 
\draw (0,-2.7) node {\ref{E3}};
\end{tikzpicture} 

\begin{tikzpicture}[scale=0.3]
   \path 
    (140:6) coordinate (b1)
    (115:6)  coordinate (b2)
    (-50:6) coordinate (b3)
    (-140:6)coordinate (b4)
    (0:0) coordinate (b5)
    (25:6) coordinate (b6);

     \draw[thick,red, fill=red] 
    (b2) circle (0.15cm)
    (b3) circle (0.15cm) (b4) circle (0.15cm)
     (b5) circle (0.15cm) (b6) circle (0.15cm);

    \draw[line width=1pt,red] (b2) to (b5);
    \draw[line width=1pt,red] (b3) to (b5);
    \draw[line width=1pt,red] (b4) to (b5);
    \draw[line width=1pt,red] (b6) to (b5);
    \draw[bend left, line width=1pt,red] (b6) to (3,5);
    \draw[bend right, line width=1pt,red] (b6) to (6,-1);
    \draw[bend left,bluearrow](-0.5,1.3)to(1,-1.2);

    \draw(-0.5,4) node[black] {\scriptsize$\mathbf{x_{i-1}}$};
    \draw(3.5,0.9) node[black] {\scriptsize$\mathbf{x_{i}}$};
   \draw(1.6,-3.7) node[black] {\scriptsize$\mathbf{x_{i+1}}$};
    \draw (-2,1) node[red] {\Huge\vdots};
 	\draw (0,-2.5) node[red] {\Huge\ldots};

\end{tikzpicture}
\hspace{2cm}
\begin{tikzpicture}[scale=0.3]
   \path 
    (140:6) coordinate (b1)
    (115:6)  coordinate (b2)
    (-50:6) coordinate (b3)
    (-140:6)coordinate (b4)
    (0:0) coordinate (b5)
    (25:6) coordinate (b6);

     \draw[thick,red, fill=red] 
    (b2) circle (0.15cm)
    (b3) circle (0.15cm) (b4) circle (0.15cm)
     (b5) circle (0.15cm) (b6) circle (0.15cm);

    \draw[line width=1pt,red] (b2) to (b5);
    \draw[line width=1pt,red] (b3) to (b5);
    \draw[line width=1pt,red] (b4) to (b5);
    \draw[line width=1pt,red] (b6) to (b5);
    \draw[bend left, line width=1pt,red] (b6) to (3,5);
    \draw[bend right, line width=1pt,red] (b6) to (6,-1);
    \draw[bend left, bluearrow](5.1,1)to(4.1,3.1);

    \draw(-0.5,4) node[black] {\scriptsize$\mathbf{x_{i-1}}$};
    \draw(3.5,0.9) node[black] {\scriptsize$\mathbf{x_{i}}$};
   \draw(1.6,-3.7) node[black] {\scriptsize$\mathbf{x_{i+1}}$};
    \draw (-2,1) node[red] {\Huge\vdots};
 	\draw (0,-2.5) node[red] {\Huge\ldots};

\end{tikzpicture}

\caption{This figure shows the local quiver and geometric models for a stable (II)-non-gentle vertex $v_i$ on a cycle $\mathcal{C}$. Reversing all arrows at such vertices yields configurations satisfying the same conditions.}
\label{fig:red puncture}
\end{figure}
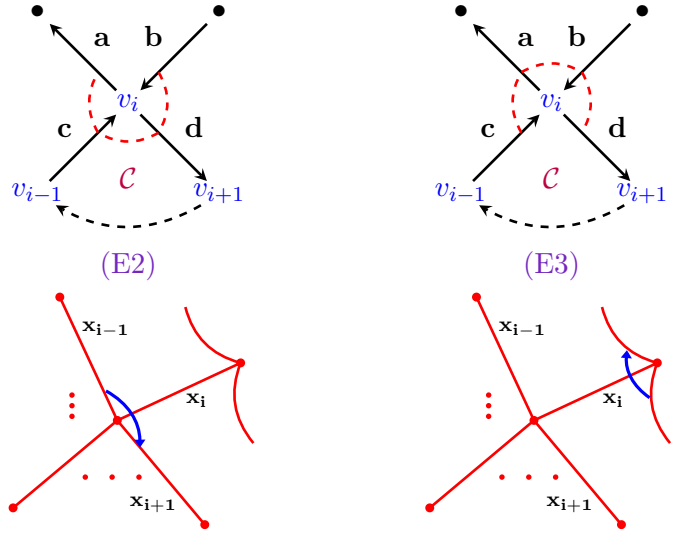
\end{definition}

\begin{remark}
Under condition \ref{E2} in Definition \ref{def:stable_vertex}, for a (II)-non-gentle vertex $v_i$ on $\mathcal{C}$, the 2-relation $a_{i-1}a_i$ is precisely the one chosen to be deleted to transform $v_i$ into a gentle vertex. This uniquely deleted relation corresponds to a label (Type \ref{rel:R2}) in the geometric model, as illustrated in Figure \ref{fig:red puncture}.
\end{remark}

Once the stable vertices are established, we can define the essential cycle  composed entirely of such vertices.

\begin{definition} \label{def:essential_cycle}
Let $A=\mathbf{k}Q/I$ be a string algebra. A primitive oriented cycle  $\mathcal{C} = v_1 \xrightarrow{a_1} v_2 \xrightarrow{a_2} \cdots \xrightarrow{a_{n-1}} v_n \xrightarrow{a_n} v_1$  in $Q$ is called an \emph{essential cycle} if every vertex in \(\mathcal{C}\) is a stable vertex. 
\end{definition}

Next, we analyze the geometric realization of the essential cycle, showing that its existence yields a red puncture in every geometric model of $A$. Specifically, we have the following proposition:

\begin{proposition}\label{prop: essential cycle}
    Let $A=\mathbf{k}Q/I$ be a string algebra. If the quiver $Q$ contains an essential cycle $\mathcal{C}$, then every geometric model of $A$ contains a red puncture. 
\end{proposition} 

\begin{proof}
Let $\mathcal{C} = v_1 \xrightarrow{a_1} v_2 \xrightarrow{a_2} \cdots \xrightarrow{a_{n-1}} v_n \xrightarrow{a_n} v_1$ be an essential cycle in $Q$. Let $B=\mathbf{k}Q/J$ be any locally gentle algebra associated with $A$. We will show that no length-2 subpath of $\mathcal{C}$ belongs to $J$. By Definition \ref{def:essential_cycle}, every vertex on $\mathcal{C}$ is a stable vertex. Specifically:
\begin{itemize}
    \item At any gentle vertex $v_i$ on $\mathcal{C}$, condition \ref{E1} ensures $a_{i-1}a_i \notin I$, which naturally yields $a_{i-1}a_i \notin J$.
    \item At any non-gentle vertex $v_i$ on $\mathcal{C}$ satisfying condition \ref{E2} ($a_{i-1}a_i \in I$), the uniquely deleted relation chosen to transform $v_i$ into a gentle vertex lies precisely along $\mathcal{C}$. Thus, $a_{i-1}a_i$ is removed from $I$, ensuring $a_{i-1}a_i \notin J$ (corresponding to the label $r_1$ in Figure \ref{fig:geo-red puncture}).
    \item At any non-gentle vertex $v_i$ on $\mathcal{C}$ satisfying condition \ref{E3} ($a_{i-1}a_i \notin I$), the path naturally avoids the ideal, directly ensuring $a_{i-1}a_i \notin J$.
\end{itemize}

Since $J$ is generated strictly by relations of length 2, the fact that $a_{i-1}a_i \notin J$ for all $i$ guarantees that $\mathcal{C}$ is a primitive permitted cycle in the bound quiver $(Q, J)$. Consequently, $\mathcal{C}^m \neq 0$ in $B$ for all $m \ge 1$. Furthermore, any relations in $I$ of length $k \ge 3$ along $\mathcal{C}$ are removed during the construction of $B$ and merely manifest as labels (Type \ref{rel:R2}, see $r_2$ in Figure \ref{fig:geo-red puncture}). 

Thus, $B$ admits paths of arbitrary length and is infinite-dimensional. By the geometric correspondence, its geometric model must contain a red puncture. Since $J$ was chosen arbitrarily, every geometric model of $A$ contains a red puncture.
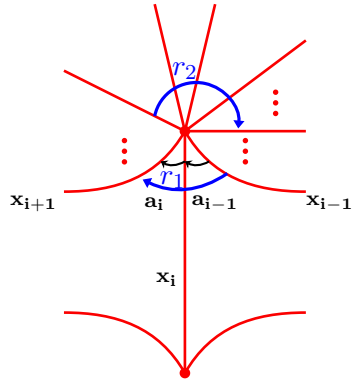
\begin{figure}[htbp]
    \begin{center}
          \begin{tikzpicture}[>=stealth, scale=0.8]
            \draw[line width=1pt, red] (0,2) --(0,-2);
            \draw[line width=1pt, red] (0,2) --(2,2);
            \draw[line width=1pt, red] (0,2) --(2,3.5);
            \draw[line width=1pt, red] (0,2) --(-0.5,4.1);
             \draw[line width=1pt, red] (0,2) --(0.5,4.1);
            \draw[line width=1pt, red] (0,2) --(-2,3);
            
            \draw (1,1.8) node[red] {\Huge\vdots};
             \draw (-1,1.8) node[red] {\Huge\vdots};
             \draw (1.5,2.6) node[red] {\Huge\vdots};

             \draw (-0.2,1.2) node[blue] {$r_1$};
             \draw (0,3) node[blue] {$r_2$};
             
             \draw[bluearrow] (-0.5, 2.25) arc[start angle=168, end angle=-5, radius=0.7];

            \draw[line width=1pt, red] (-2,1) to[out=0, in=-120](0,2)to[out=-60, in=180](2,1);
            \draw[, line width=1pt, red] (-2,-1) to[out=0, in=120](0,-2)to[out=60, in=180](2,-1);
            \draw[blackarrow](0.4,1.5)to(0,1.5);
            \node [] at (0.5,0.8) {\scriptsize$\mathbf{a_{i-1}}$};	

            \draw[blackarrow](0,1.5)to(-0.4,1.5);
            \node [] at (-0.5,0.8) {\scriptsize$\mathbf{a_{i}}$};

            \node [] at (2.4,0.8) {\scriptsize$\mathbf{x_{i-1}}$};
            \node [] at (-0.3,-0.45) {\scriptsize$\mathbf{x_{i}}$};
            \node [] at (-2.5,0.8) {\scriptsize$\mathbf{x_{i+1}}$};

            \draw[bend left,bluearrow](0.7,1.3)to(-0.7,1.2);
            
            \draw[red,thick,fill=red] (0,2) circle (0.07);
            \draw[red,thick,fill=red] (0,-2) circle (0.07);
            
        \end{tikzpicture}
        
    \caption{Local geometry of essential cycle \(\mathcal{C}\). \(r_2\): k-relation (\(k\ge3\)) on gentle vertices of \(\mathcal{C}\); \(r_1\): unique deleted 2-relation at Type (II) vertex \(v_i\). Both yield labels.}
    \label{fig:geo-red puncture}
    \end{center}
 \end{figure}

\end{proof}

\subsubsection{Common red punctures}
This gives rise to the notion of a \emph {common red puncture}, whose full characterization via essential cycles is provided below.
\begin{definition}\label{def:common-red-puncture}
Let \(A=\mathbf{k}Q/I\) be a string algebra. For each admissible deletion datum
\(\delta\in\mathfrak D(A)\), denote by \(J_\delta\) the corresponding
saturated locally gentle cover of $A$. Let \(\mathcal C\) be a primitive
oriented cycle in \(Q\). If \(\mathcal C\) is a
\(J_\delta\)-permitted cycle for every
\(\delta\in\mathfrak D(A)\), then its cyclic equivalence class $[\mathcal C]_{\mathrm{cyc}}$ is said to determine a \emph{common red puncture} of \(A\).
\end{definition}

\begin{theorem}\label{thm:common-red-essential}
Let \(A=\mathbf{k}Q/I\) be a string algebra. Then all geometric models of \(A\) have a common red puncture
if and only if \(Q\) contains an essential cycle.
\end{theorem}

\begin{proof}
The sufficiency follows directly from the proof of
Proposition~\ref{prop: essential cycle}: the same essential cycle
\(\mathcal C\) is \(J_\delta\)-permitted for every
\(\delta\in\mathfrak D(A)\).

Conversely, suppose that
\([\mathcal C]_{\mathrm{cyc}}\) determines a common red puncture, and
write
\[
    \mathcal C=
    v_1\xrightarrow{a_1}v_2\xrightarrow{a_2}\cdots
    \xrightarrow{a_{n-1}}v_n\xrightarrow{a_n}v_1,
\]
We proceed by analyzing the subpath \(a_{i-1}a_i\) along the cycle.

\noindent\textbf{Case 1.} \(a_{i-1}a_i\) does not belong to $I$.
  Since \(J_\delta \subseteq I\) holds for every admissible deletion datum \(\delta\), we have \(a_{i-1}a_i \notin J_\delta\).
  \begin{itemize}
      \item If \(v_i\) is a gentle vertex, then \(v_i\) satisfies \ref{E1}.
      \item If \(v_i\) is a non-gentle vertex of Type~\textup{(II)}, then \(v_i\) satisfies \ref{E3}.
  \end{itemize}

\noindent\textbf{Case 2.} \(a_{i-1}a_i\) belongs to $I$. 
  \begin{itemize}
    \item  \(v_i\) cannot be a gentle vertex. Assume for contradiction that \(v_i\) is gentle. Then \(a_{i-1}a_i \in J_\delta\), contradicting the premise that \(\mathcal{C}\) is \(J_\delta\)-permitted for all admissible deletion data \(\delta\). Therefore \(v_i\) is non-gentle.
    \item Since \(\mathcal{C}\) is \(J_\delta\)-permitted for all \(\delta\),  \(a_{i-1}a_i\) cannot belong to any \(J_\delta\).
     Consequently, \(a_{i-1}a_i\) must be the unique 2-relation deleted at \(v_i\) during the transformation, so \(v_i\) satisfies \ref{E2}.
\end{itemize}

Finally, \(v_i\) cannot be a  non-gentle vertex of Type~\textup{(I)} or
Type~\textup{(III)}. By Remark \ref{remark:transform non-gentle}, we can always choose a local deletion at \(v_i\) for which
$a_{i-1}a_i\in J_\delta$, a contradiction. The configurations involving loops described in Remark~\ref{remark:transform non-gentle} are also covered by this local classification.

Therefore, every vertex of \(\mathcal C\) is stable, and \(\mathcal C\) is an essential cycle.
\end{proof}

\subsection{Shuttle cycles and Coupled cycles}
In general, the primitive \(J_\delta\)-permitted cycle giving rise to such a puncture may depend on the admissible deletion datum \(\delta\). Different choices of \(\delta\) lead to distinct primitive oriented cycles, all of which are \(J_\delta\)-permitted and produce red punctures. We therefore introduce two types of cycles: \emph{shuttle cycles} and \emph{multi-coupled cycles}.

\subsubsection{Shuttle cycles}
We now introduce the definition of shuttle cycles.

\begin{definition}\label{def:shuttle_cycle}
Let $A = \mathbf{k}Q/I$ be a string algebra. A subquiver $\mathcal{Q}$ of $Q$ is called a \emph{shuttle cycle} (see Figure \ref{fig: shuttle cycles}), if:
\begin{itemize}
    \item It consists of a sequence of distinct vertices $V = \{v_1, v_2, \dots, v_n\}$ ($n \ge 3$) and a set of antiparallel arrows 
    \[
    E = \{a_i: v_i \to v_{i+1}, \; b_i: v_{i+1} \to v_i \mid 1 \le i \le n-1\};
    \]
    \item The vertex $v_1$ is a stable vertex with respect to the primitive oriented cycle formed by $a_1$ and $b_1$, and $v_n$ is a stable vertex with respect to the primitive oriented cycle formed by $a_{n-1}$ and $b_{n-1}$;
    \item Every intermediate vertex $v_i$ ($1 < i < n$) is either a gentle vertex, a Type (I) non-gentle vertex, or a Type (II) non-gentle vertex, with at least one such vertex being non-gentle.
\end{itemize}
\begin{figure}[htbp]
    \centering

\begin{tikzcd}
\mathcal{Q}: & v_1 \arrow[rr, "a_1", shift left=2] &  & v_2 \arrow[ll, "b_1", shift left] \arrow[rr, "a_2", shift left=2] &  & v_3 \arrow[ll, "b_2", shift left] \arrow[rrr, dashed, shift left=2] &  &  & v_{n-1} \arrow[lll, dashed, shift left] \arrow[rr, "a_{n-1}", shift left=2] &  & v_n \arrow[ll, "b_{n-1}", shift left]
\end{tikzcd}
\vspace{1.2cm}
\begin{tikzpicture}[scale=0.4]
   \path 
    (0,3)+(60:5) coordinate (b1)
    (0,3)+(120:5)  coordinate (b2)
    (0,3)+(0:6) coordinate (b3)
    (0,3)+(180:6)coordinate (b4)
    (0,3)+(-60:5) coordinate (b5)
    (0,3)+(-120:5) coordinate (b6)
    ;

     \draw[thick,red, fill=red] 
    (b1) circle (0.15cm) (b2) circle (0.15cm)
    (b3) circle (0.15cm) (b4) circle (0.15cm)
     ;

    \draw[line width=1pt,red] (b2) to (b1);
    \draw[,line width=1pt,red,dashed] (b3) to (b1);
    \draw[,line width=1pt,red,dashed] (b4) to (b2);
    \draw[bend left,line width=1pt,red] (b4) to (b6);
    \draw[bend right,line width=1pt,red] (b3) to (b5);
    
    \draw(0,6.5) node[black] {$\mathbf{x_{i}}$};
    \draw(4.3,1) node[black] {$\mathbf{x_{1}}$};
    \draw(-4.3,1) node[black] {$\mathbf{x_n}$};

\end{tikzpicture}
\hspace{2cm}
\begin{tikzpicture}[scale=0.4]
   \path 
    (0:0) coordinate (b1)
    (60:4)  coordinate (u1)
    (120:4) coordinate (u2)
    (-60:3)coordinate (u3)
    (-120:3) coordinate (u4)    
    
     ;
      
     \draw[thick,red, fill=red] 
    (b1) circle (0.15cm);

          \draw[,line width=1pt,red] (b1) to (u1);
          \draw[bend right,line width=1pt,red] (b1) to (u3);
          \draw[bend left,line width=1pt,red,dashed] (u1) to (u3);
        \draw[line width=1.5pt] (2, 3.4641) to[out=70, in=110, looseness=2.5] (3.0, 3.7641); 
        \draw[line width=1.5pt] (1.7, 3.8641) to[out=-60, in=-66, looseness=2] (2, 3.4641) to[out=-70, in=48, looseness=2] (3.0, 3.7641) to[out=48, in=50, looseness=2] (3.3,4.0641);
         \node[black] at (0,4) {\Huge\ldots};
         \draw[,line width=1pt,red] (b1) to (u2);
          \draw[bend left,line width=1pt,red] (b1) to (u4);
          \draw[bend right,line width=1pt,red,dashed] (u2) to (u4);
         
       \draw[line width=1.5pt] (-3, 3.1641) to[out=70, in=110, looseness=2.5] (-2, 3.4641); 
       \draw[line width=1.5pt] (-3.3, 3.5641) to[out=-60, in=-66, looseness=2] (-3, 3.1641) to[out=-70, in=48, looseness=2] (-2, 3.4641) to[out=48, in=50, looseness=2] (-1.7,3.7641);

     \coordinate (m1) at (5.5,5.2) ;
     \draw[bend right,line width=1pt,red] (b1) to (m1);
     \draw[bend left,line width=1pt,red] (b1) to (-5.5,5.2);
     \node[red] at (0.1,2.2) {\Huge\ldots};

     \draw[line width=1.2pt, blue, dashed, ->=bluearrow] (-0.5,0.25) arc[start angle=160, end angle=-100, radius=0.5] arc[start angle=-110, end angle=-165, radius=1] ;

    \draw(-5,2.4) node[black] {$\mathbf{x_{1}}$};
    \draw(-1.5,-3) node[black] {$\mathbf{x_{2}}$};
    \draw(1.5,-3) node[black] {$\mathbf{x_{n-1}}$};
    \draw(5,2.4) node[black] {$\mathbf{x_{n}}$};
      
\end{tikzpicture}
\vspace{-1cm}
\caption{The shuttle cycle and its geometric realization.}
\label{fig: shuttle cycles}
\end{figure}
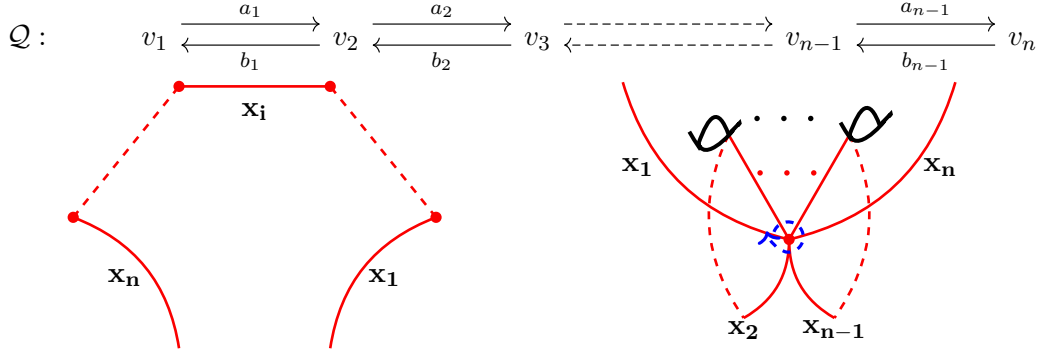
\end{definition}

\begin{remark}\label{remark:shuttle_n_ge_3}
We strictly require $n \ge 3$. The structural analogues for $n=1$ and $n=2$ naturally correspond to other established configurations: 
\begin{itemize}
    \item When $n=1$, the analogous subquiver is a single vertex with two loops, as discussed in Example \ref{ex:two_loops}.
    \item When $n=2$, the subquiver consists of exactly two vertices $v_1$ and $v_2$. By the shuttle cycle definition, both must be stable vertices with respect to the 2-cycle formed by $a_1$ and $b_1$. Thus, the entire subquiver is precisely an essential cycle.
\end{itemize}
\end{remark}

With the definition of shuttle cycles established, we now show that their presence guarantees a red puncture in every geometric model.

\begin{proposition}\label{prop:shuttle_cycle}
Let $A = \mathbf{k}Q/I$ be a string algebra. If the quiver $Q$ contains a shuttle cycle $\mathcal{Q}$, then every geometric model of $A$ contains a red puncture.
\end{proposition}

\begin{proof}
Let $B = \mathbf{k}Q/J$ be an arbitrary locally gentle algebra associated with $A$. We will construct an infinite non-zero path in $B$ strictly supported on $\mathcal{Q}$.

For any  vertex $v_i$ ($1 < i < n$) in $\mathcal{Q}$, there are exactly two incoming and two outgoing arrows within $\mathcal{Q}$. Regardless of whether \(v_i\) is gentle or non-gentle (Type I or II), the admissible deletion datum ensures that $(Q,J)$ satisfies \ref{G1}  and \ref{S2}. Consequently, for any incoming arrow at $v_i$, there exists a unique outgoing arrow in $\mathcal{Q}$ such that the resulting length-$2$ path does not belong to $J$. The endpoints $v_1$ and $v_n$ are stable vertices with respect to the primitive oriented cycles formed by $\{a_1,b_1\}$ and $\{a_{n-1},b_{n-1}\}$, respectively, which yields $b_1 a_1 \notin J$ and $a_{n-1} b_{n-1} \notin J$. 

Consequently, for any given arrow in $\mathcal{Q}$, there exists a uniquely determined subsequent arrow in $\mathcal{Q}$ such that their composition does not belong to $J$. Starting with an arbitrary arrow $p_1 \in \mathcal{Q}$, we can inductively append these unique valid paths to form an infinite sequence of arrows $p = p_1 p_2 p_3 \cdots$. 

Because $B$ is a locally gentle algebra, its ideal $J$ is generated strictly by relations of length 2. By our inductive construction, no length-2 subpath of $p$ belongs to $J$, meaning $p$ is an infinite non-zero path in $B$. This implies $B$ is infinite-dimensional, ensuring that its geometric model contains a red puncture. Since $B$ was chosen arbitrarily, every geometric model of $A$ inevitably contains a red puncture.
\end{proof}

\subsubsection{Multi-Coupled cycles}
We introduce multi-coupled cycles, which consist of multiple overlapping coupled cycles.

\begin{definition}\label{def:couple_cycle}
Let $A = \mathbf{k}Q/I$ be a string algebra. A subquiver $\mathcal{C}_{cou}$ of $Q$ is called a \emph{coupled cycle} if it consists of two primitive oriented cycles $C_1$ and $C_2$ traversing a set of shared vertices $\{v_1, \dots, v_k\}$ in one of the following two situations (see Figure \ref{fig:coupled cycle}), where $n, m \ge k \ge 1$:

\begin{enumerate}[label=(C\arabic*), ref=(C\arabic*)]
    \item \label{C1}  Both cycles traverse $v_1, \dots, v_k$ in the same direction:
    \begin{align*}
        \mathcal{C}_1 &= v_1 \xrightarrow{a_1} v_2 \xrightarrow{a_2} \dots \xrightarrow{a_{k-1}} v_k \xrightarrow{a_k} v_{k+1} \xrightarrow{a_{k+1}} \dots \xrightarrow{a_{n-1}} v_n \xrightarrow{a_n} v_1, \\
       \mathcal{C}_2 &= v_1 \xrightarrow{b_1} v_2 \xrightarrow{b_2} \dots \xrightarrow{b_{k-1}} v_k \xrightarrow{b_k} v'_{k+1} \xrightarrow{b_{k+1}} \dots \xrightarrow{b_{m-1}} v'_m \xrightarrow{b_m} v_1.
    \end{align*}

    \item \label{C2} Both cycles traverse $v_1, \dots, v_k$ in opposite directions:
    \begin{align*}
        \mathcal{C}_1 &= v_1 \xrightarrow{a_1} v_2 \xrightarrow{a_2} \dots \xrightarrow{a_{k-1}} v_k \xrightarrow{a_k} v_{k+1} \xrightarrow{a_{k+1}} \dots \xrightarrow{a_{n-1}} v_n \xrightarrow{a_n} v_1, \\
        \mathcal{C}_2 &= v_1 \xleftarrow{b_1} v_2 \xleftarrow{b_2} \dots \xleftarrow{b_{k-1}} v_k \xleftarrow{b_k} v'_{k+1} \xleftarrow{b_{k+1}} \dots \xleftarrow{b_{m-1}} v'_m \xleftarrow{b_m} v_1.
    \end{align*}
\end{enumerate}
Here, the vertex sets $\{v_1, \dots, v_k\}$, $\{v_{k+1}, \dots, v_n\}$, and $\{v'_{k+1}, \dots, v'_m\}$ are pairwise disjoint, and the explicitly labeled parallel or anti-parallel arrows within the shared path are strictly distinct ($a_i \neq b_i$ for all $1 \le i \le k-1$). Furthermore, $\mathcal{C}_{cou}$ must satisfy the following conditions:
\begin{itemize}
    \item At least one vertex among the shared vertices $v_1, \dots, v_k$ is a non-gentle vertex of Type (I) or Type (II).
    \item The non-shared vertices $\{v_{k+1}, \dots, v_n\}$ (if $n > k$) are stable vertices with respect to \ref{C1}, and $\{v'_{k+1}, \dots, v'_m\}$ (if $m > k$) are stable vertices with respect to \ref{C2}. If $n=k$ (resp., $m=k$), $a_k$ (resp., $b_k$) is an arrow directed from $v_k$ to $v_1$.
\end{itemize}

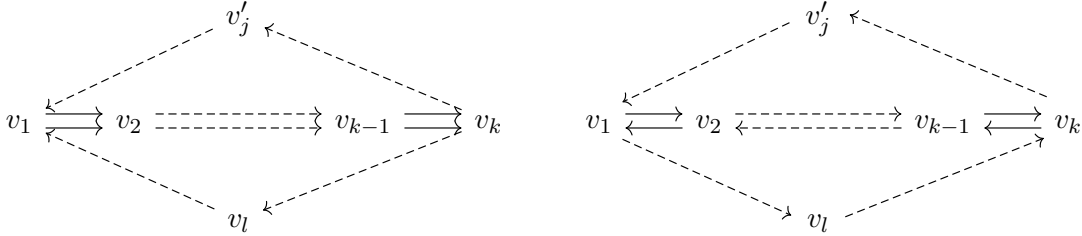
\begin{figure}[htbp]
    \centering
\begin{tikzcd}[column sep=2em, row sep=2em]
& & v'_j \arrow[lld, dashed] & & & & & v'_j \arrow[lld, dashed, shift right] & & \\
v_{1} \arrow[r, shift left] \arrow[r, shift right] & v_2 \arrow[rr, dashed, shift left] \arrow[rr, dashed, shift right] & & v_{k-1} \arrow[r, shift left] \arrow[r, shift right] & v_{k} \arrow[llu, dashed] \arrow[lld, dashed] & v_{1} \arrow[r, shift left] \arrow[rrd, dashed, shift right] & v_2 \arrow[rr, dashed, shift left] \arrow[l, shift left] & & v_{k-1} \arrow[r, shift left] \arrow[ll, dashed, shift left] & v_{k} \arrow[llu, dashed, shift right=2] \arrow[l, shift left] \\
& & v_l \arrow[llu, dashed] & & & & & v_l \arrow[rru, dashed, shift right] & &
\end{tikzcd}
    \caption{Quiver configurations for the two cases of a coupled cycle.}
    \label{fig:coupled cycle}
\end{figure}
\end{definition}

\begin{remark}\label{rem:couple_special_cases}
A coupled cycle exhibits rich combinatorial properties and may intrinsically encompass other structures:
\begin{itemize}
    \item When \(n = m = k\), the two cycles completely overlap without branching into non-shared stable paths; its geometric model is shown in Figure \ref{fig:geo-sing algebra}. 
    \item It may contain a shuttle cycle as a subquiver. Specifically, as illustrated in the right diagram of Figure \ref{fig:coupled cycle}, when $k\ge 4$, the bound full subquiver of $Q$ formed by vertices $v_2,\dots,v_{k-1}$ inherently constitutes a shuttle cycle.
   
\end{itemize}
\end{remark}

\begin{figure}[htbp]
    \begin{center}
     \begin{tikzpicture}[scale=0.35]
            \begin{scope}[xshift=0,yshift=10cm]
                \draw[line width=1pt,fill=white] (0,0) circle (7);
                
                \path (-130:3) coordinate (b1) (-90:5) coordinate (b2)
                (-80:7) coordinate (m1)
                (-90:7) coordinate (m2)
                (-110:7) coordinate (m3)
                (-135:7) coordinate (m4)
                ;
                \draw[thick,red, fill=red] (b1) circle (0.15) (b2) circle (0.15) ;
                \draw[thick,black, fill=white] (0,2.5) circle (0.15) ;
                                            
               \draw[line width=1pt,red] (b1) to (-4.2, 0.5);
               \draw[bend right,line width=1pt,red,dashed] (-4.2, 0.5) to (m1);
                
                \draw[line width=1.5pt] (-5.2, 0.2) to[out=70, in=110, looseness=2.5] (-4.2, 0.5); 
                \draw[line width=1.5pt] (-5.5,0.6) to[out=-60, in=-66, looseness=2] (-5.2, 0.2) to[out=-70, in=48, looseness=2] (-4.2, 0.5) to[out=48, in=50, looseness=2] (-3.9,0.8);
                \node[black] at (-2.65,0.5) {\Huge\ldots};

              \draw[line width=1pt,red] (b1) to (-0.5, 0.5);
                 \draw[,line width=1pt,red,dashed] (-0.5, 0.5) to [out=100,in=80,looseness=2] (-5.7,0.5) to [out=-100,in=150] (m2);
                \draw[line width=1.5pt] (-1.5, 0.2) to[out=70, in=110, looseness=2.5] (-0.5, 0.5); 
                \draw[line width=1.5pt] (-1.8,0.6) to[out=-60, in=-66, looseness=2] (-1.5, 0.2) to[out=-70, in=48, looseness=2] (-0.5, 0.5) to[out=48, in=50, looseness=2] (-0.2,0.8);
                
               \draw[bend right,line width=1pt,red] (b1) to (2, 0.5);
               \draw[,line width=1pt,red,dashed] (2, 0.5) to [out=100,in=80,looseness=2] (-6,0.5) to [out=-100,in=130] (m3);
                \draw[line width=1.5pt] (1.0, 0.2) to[out=70, in=110, looseness=2.5] (2.0, 0.5); 
                \draw[line width=1.5pt] (0.7,0.6) to[out=-60, in=-66, looseness=2] (1.0, 0.2) to[out=-70, in=48, looseness=2] (2.0, 0.5) to[out=48, in=50, looseness=2] (2.3,0.8);
                \node[black] at (3.3,0.5) {\Huge\ldots};
                
               \draw[bend right,line width=1pt,red] (b1) to (4.5,0.2);
               \draw[,line width=1pt,red,dashed] (4.5,0.2) to [out=100,in=80,looseness=2] (-6.3,0.5) to [out=-100,in=120] (m4);
                \draw[line width=1.5pt] (4.5, 0.2) to[out=70, in=110, looseness=2.5] (5.5, 0.5); 
                \draw[line width=1.5pt] (4.2,0.6) to[out=-60, in=-66, looseness=2] (4.5, 0.2) to[out=-70, in=48, looseness=2] (5.5, 0.5) to[out=48, in=50, looseness=2] (5.8,0.8);

               \draw[bend left,line width=1pt,red] (b2) to (m1);
               \draw[bend left,line width=1pt,red] (b2) to (m2);
               \draw[bend left,line width=1pt,red] (b2) to (m3);
               \draw[bend left,line width=1pt,red] (b2) to (m4);

                \draw[line width=1.2pt, blue, dashed, ->=bluearrow] (-1.4,-2.2) arc[start angle=-10, end angle=-200, radius=0.6]arc[start angle=-200, end angle=-360, radius=0.7];            

                \draw[line width=1.2pt, blue, dashed, ->=bluearrow] (-0.25,-5.4) arc[start angle=-120, end angle=-340, radius=0.5] arc[start angle=-340, end angle=-475, radius=0.7];

               \node[red] at (-4.7,-1.6) {\Huge$\kern-0.1em.\kern-0.1em.\kern-0.1em.$};
                \node[red] at (3,2) {\Huge$\kern-0.1em.\kern-0.1em.\kern-0.1em.$};
                
                \node[black] at (-1.8,-0.8) {\scriptsize$\mathbf{x_{n}}$};
                \node[black] at (1.8,-1) {\scriptsize$\mathbf{x_{1}}$};
                \node[black] at (-3.5,-1) {\scriptsize$\mathbf{x_{i}}$};
                \node[black] at (4.3,-1.5) {\scriptsize$\mathbf{x_{i-1}}$};
            \end{scope}
            
            \begin{scope}[xshift=17cm,yshift=10cm]
                \draw[line width=1pt,fill=white] (0,0) circle (7);
                \path (-90:4) coordinate (b1);
                \draw[thick,red, fill=red] (b1) circle (0.15) ;
                \draw[thick,black, fill=white] (0,2.5) circle (0.15) ;
                
                \draw[line width=1pt,red] (b1) to[out=-170,in=-90] (-6.5,1.65) to[out=90,in=110] (6,2) to[out=-70,in=20] (b1) ;

                \draw[line width=1.5pt] (-5, 0.2) to[out=70, in=110, looseness=2.5] (-4, 0.5); 
                \draw[line width=1.5pt] (-5.3,0.6) to[out=-60, in=-66, looseness=2] (-5, 0.2) to[out=-70, in=48, looseness=2] (-4, 0.5) to[out=48, in=50, looseness=2] (-3.7,0.8);
                \draw[line width=1pt, red] (b1) to(-4, 0.5);
               \path  (-45:7) coordinate (u1) 
               ;
                \draw[bend right=70,line width=1pt, red,dashed] (-4, 0.5) to (u1) ;
                \draw[,line width=1pt, red] (u1) to (b1);
                \node[black] at (-2.45,0.5) {\Huge\ldots};

                \draw[line width=1.5pt] (-1.3, 0.2) to[out=70, in=110, looseness=2.5] (-0.3, 0.5); 
                \draw[line width=1.5pt] (-1.6,0.6) to[out=-60, in=-66, looseness=2] (-1.3, 0.2) to[out=-70, in=48, looseness=2] (-0.3, 0.5) to[out=48, in=50, looseness=2] (0,0.8);
                \draw[line width=1pt, red]  (-0.3, 0.5) to (b1);
                 \draw[,line width=1pt, red,dashed] (-0.3, 0.5)  to [out=90,in=80,looseness=2] (-5.5,0.5);
               \path (-90:7) coordinate (u2) ;
                \draw[line width=1pt, red,dashed] (-5.5, 0.5) to [out=-100,in=150] (u2);
                \draw[bend left=20,line width=1pt, red] (b1) to(u2);
                
                \draw[line width=1.5pt] (1.0, 0.2) to[out=70, in=110, looseness=2.5] (2.0, 0.5); 
                \draw[line width=1.5pt] (0.7,0.6) to[out=-60, in=-66, looseness=2] (1.0, 0.2) to[out=-70, in=48, looseness=2] (2.0, 0.5) to[out=48, in=50, looseness=2] (2.3,0.8);
                \node[black] at (3.3,0.5) {\Huge\ldots};
                \draw[line width=1pt, red] (b1) to(1, 0.2);
                 \draw[,line width=1pt, red,dashed] (1,0.2) to [out=90,in=80,looseness=2] (-5.7,0.5);
                  \path   (-120:7) coordinate (u3) ;
               \draw[,line width=1pt, red,dashed] (-5.7,0.5) to [out=-100,in=130] (u3);
                \draw[bend left,line width=1pt, red] (b1) to (u3);

                \draw[line width=1.5pt] (4.5, 0.2) to[out=70, in=110, looseness=2.5] (5.5, 0.5); 
                \draw[line width=1.5pt] (4.2,0.6) to[out=-60, in=-66, looseness=2] (4.5, 0.2) to[out=-70, in=48, looseness=2] (5.5, 0.5) to[out=48, in=50, looseness=2] (5.8,0.8);
                \draw[line width=1pt, red] (b1) to(4.5, 0.2);
                 \draw[,line width=1pt, red,dashed] (4.5,0.2) to [out=90,in=80,looseness=1.55] (-6.1,0.5);
                \path  (-150:7) coordinate (u4) ;
               \draw[,line width=1pt, red,dashed] (-6.1,0.5) to [out=-100,in=100](u4);
               \draw[bend right,line width=1pt,red](u4) to (b1);

               \node[red] at (-0.9,-2.2) {\Huge$\kern-0.1em.\kern-0.1em.\kern-0.1em.$};
                \node[red] at (1.3,-2.2) {\Huge$\kern-0.1em.\kern-0.1em.\kern-0.1em.$};

                \draw[line width=1.2pt, blue, dashed, ->=bluearrow] (0.15,-3.55) arc[start angle=60, end angle=-110, radius=0.5] arc[start angle=-110, end angle=-200, radius=0.5] arc[start angle=-200, end angle=-265, radius=0.9];            
               
                 \node[black] at (0,5.66) {\scriptsize$\mathbf{x_{i}}$};
                \node[black] at (-0.7,-1) {\scriptsize$\mathbf{x_{n}}$};
                \node[black] at (1.2,-1) {\scriptsize$\mathbf{x_{1}}$};
                \node[black] at (-2.85,-1.8) {\scriptsize$\mathbf{x_{i+1}}$};
                \node[black] at (3.8,-1.5) {\scriptsize$\mathbf{x_{i-1}}$};
            \end{scope}
        \end{tikzpicture}  

        \vspace{0.1cm}
        
        \begin{tikzpicture}[scale=0.35]
            \begin{scope}[xshift=0,yshift=10cm]

                \draw[line width=1pt,fill=white] (0,0) circle (7);
                \path 
                    (90:6) coordinate (b1) 
                    (50:5.5)  coordinate (b2)  
                    (130:5.5) coordinate (b3) 
                    (170:4) coordinate (b4)
                    (10:4)  coordinate (b5)
                    (-30:3.5)  coordinate (b6)
                    (-150:3.5) coordinate (b7)  
                    (90:3) coordinate (r1)
                    (-90:4)  coordinate (r2)
                    (-30:7)    coordinate (u1)
                    (-150:7)  coordinate (u2)
                ;
                \draw[thick,red, fill=red] 
                    (b1) circle (0.15) (b2) circle (0.15) (b3) circle (0.15)
                    (b4) circle (0.15) (b5) circle (0.15) (b6) circle (0.15) (b7) circle (0.15);
                \fill[gray!10] (u2) to[bend left=10] (u1) to[bend left] cycle;  
                \draw[thick,black, fill=white] (r1) circle (0.15) (r2) circle (0.15);
              
                \draw[line width=1pt,red] (b1) to (b2);
                \draw[line width=1pt,red] (b1) to (b3);
                \draw[dashed,line width=1pt,red] (b3) to (b4);
                \draw[dashed,line width=1pt,red] (b2) to (b5);
                \draw[line width=1pt,red] (b4) to (b7);
                \draw[line width=1pt,red] (b6) to (b7);
                \draw[line width=1pt,red] (b6) to (b5);

                \draw[line width=0.8pt,black,bend right] (u2) to (u1);
                \draw[line width=0.8pt,black,dashed,bend left=10] (u2) to (u1);
                
                \draw[bluearrow] (0.3,5.85) arc[start angle=-10, end angle=-200, radius=0.35] arc[start angle=-200, end angle=-375, radius=0.45]  ;
                \draw[bluearrow] (-0.5,5.8)  arc[start angle=-160, end angle=-400, radius=0.65] arc[start angle=-400, end angle=-508, radius=0.75];
                \draw[bluearrow] (2.75,-1.75) arc[start angle=-180, end angle=-400, radius=0.35] arc[start angle=-400, end angle=-535, radius=0.45];
                 \draw[bluearrow] (3.2,-1.2)  arc[start angle=70, end angle=-190, radius=0.65] arc[start angle=-190, end angle=-290, radius=0.75];
                 \draw[bluearrow] (-2.75,-1.75) arc[start angle=0, end angle=-260, radius=0.35] arc[start angle=-260, end angle=-350, radius=0.45] ;
                 \draw[bluearrow] (-3.2,-1.2)  arc[start angle=110, end angle=-150, radius=0.6] arc[start angle=-150, end angle=-240, radius=0.7] ;

                \node[black] at (-2,5.8) {\scriptsize$\mathbf{x_{1}}$};
                \node[black] at (2,5.8)  {\scriptsize$\mathbf{x_{n}}$};
                \node[black] at (-4.5,-0.7) {\scriptsize$\mathbf{x_{i-1}}$};
                \node[black] at (4.5,-0.7){\scriptsize$\mathbf{x_{i+1}}$};
                \node[black] at (0,-2.3)  {\scriptsize$\mathbf{x_{i}}$};
            \end{scope}
            
            \begin{scope}[xshift=17cm,yshift=10cm]
                \draw[line width=1pt,fill=white] (0,0) circle (7);
                \path (-90:6) coordinate (b1);
                \draw[thick,red, fill=red] (b1) circle (0.15) ;
                \draw[thick,black, fill=white] (0,2) circle (0.15) ;
                
                \draw[line width=1pt,red] (b1) to[out=160,in=-120] (-5.5,3) to[out=60,in=110] (6,2) to[out=-70,in=20] (b1);

                \node[red] at (-1.2,-2) {\Huge\ldots};
                 \node[red] at (2.4,-2) {\Huge\ldots};
               
                \draw[line width=1.5pt] (-5.2, 0.2) to[out=70, in=110, looseness=2.5] (-4.2, 0.5); 
                \draw[line width=1.5pt] (-5.5,0.6) to[out=-60, in=-66, looseness=2] (-5.2, 0.2) to[out=-70, in=48, looseness=2] (-4.2, 0.5) to[out=48, in=50, looseness=2] (-3.9,0.8);
                \draw[line width=1pt, red] (b1) to(-4.2, 0.5);
                \path 
                    (-120:7) coordinate (u1) ;
                \draw[line width=1pt, red,dashed] (-4.2, 0.5) to[out=-130,in=120] (u1);
                \draw[line width=1pt, red] (u1) to[out=-30,in=20] (b1);
                \node[black] at (-2.65,0.5) {\Huge\ldots};

                \draw[line width=1.5pt] (-1.5, 0.2) to[out=70, in=110, looseness=2.5] (-0.5, 0.5); 
                \draw[line width=1.5pt] (-1.8,0.6) to[out=-60, in=-66, looseness=2] (-1.5, 0.2) to[out=-70, in=48, looseness=2] (-0.5, 0.5) to[out=48, in=50, looseness=2] (-0.2,0.8);
                \draw[line width=1pt, red] (b1) to(-0.5, 0.5);
                \path 
                    (-80:7) coordinate (u2) ;
                \draw[line width=1pt, red,dashed] (-0.5, 0.5) to[out=-60,in=60] (u2);
                \draw[bend right,line width=1pt, red] (b1) to(u2);
                
                \draw[line width=1.5pt] (1.0, 0.2) to[out=70, in=110, looseness=2.5] (2.0, 0.5); 
                \draw[line width=1.5pt] (0.7,0.6) to[out=-60, in=-66, looseness=2] (1.0, 0.2) to[out=-70, in=48, looseness=2] (2.0, 0.5) to[out=48, in=50, looseness=2] (2.3,0.8);
                \node[black] at (3.3,0.5) {\Huge\ldots};
                \draw[line width=1pt, red] (b1) to(2, 0.5);
                \path 
                    (-50:7) coordinate (u3) ;
                \draw[line width=1pt, red,dashed] (2, 0.5) to[out=-60,in=60] (u3);
                \draw[bend right,line width=1pt, red] (b1) to(u3);

                \draw[line width=1.5pt] (4.5, 0.2) to[out=70, in=110, looseness=2.5] (5.5, 0.5); 
                \draw[line width=1.5pt] (4.2,0.6) to[out=-60, in=-66, looseness=2] (4.5, 0.2) to[out=-70, in=48, looseness=2] (5.5, 0.5) to[out=48, in=50, looseness=2] (5.8,0.8);
                \draw[line width=1pt, red] (b1) to(5.5, 0.5);
                \path 
                    (-20:7) coordinate (u4) ;
                \draw[line width=1pt, red,dashed] (5.5, 0.5) to[out=-60,in=60] (u4);
                \draw[bend right,line width=1pt, red] (b1) to(u4);

                \draw[line width=1.2pt, blue, dashed, ->=bluearrow] (0.18,-5.55) arc[start angle=60, end angle=-110, radius=0.5] arc[start angle=-110, end angle=-200, radius=0.5] arc[start angle=-200, end angle=-270, radius=0.9];

                 \node[black] at (0,6) {\scriptsize$\mathbf{x_{i}}$};
                \node[black] at (-1,-1) {\scriptsize$\mathbf{x_{n}}$};
                \node[black] at (2.3,-1) {\scriptsize$\mathbf{x_{1}}$};
                \node[black] at (-4,-1) {\scriptsize$\mathbf{x_{i+1}}$};
                \node[black] at (5.2,-1) {\scriptsize$\mathbf{x_{i-1}}$};
            \end{scope}                                                
        \end{tikzpicture} 
        \caption{The (partial) geometric models of fully overlapping coupled cycles ($n=m=k$). The upper and lower pairs of diagrams correspond, respectively, to Cases \ref{C1} and \ref{C2} in Definition~\ref{def:couple_cycle}.}

        \label{fig:geo-sing algebra}
    \end{center}
\end{figure}
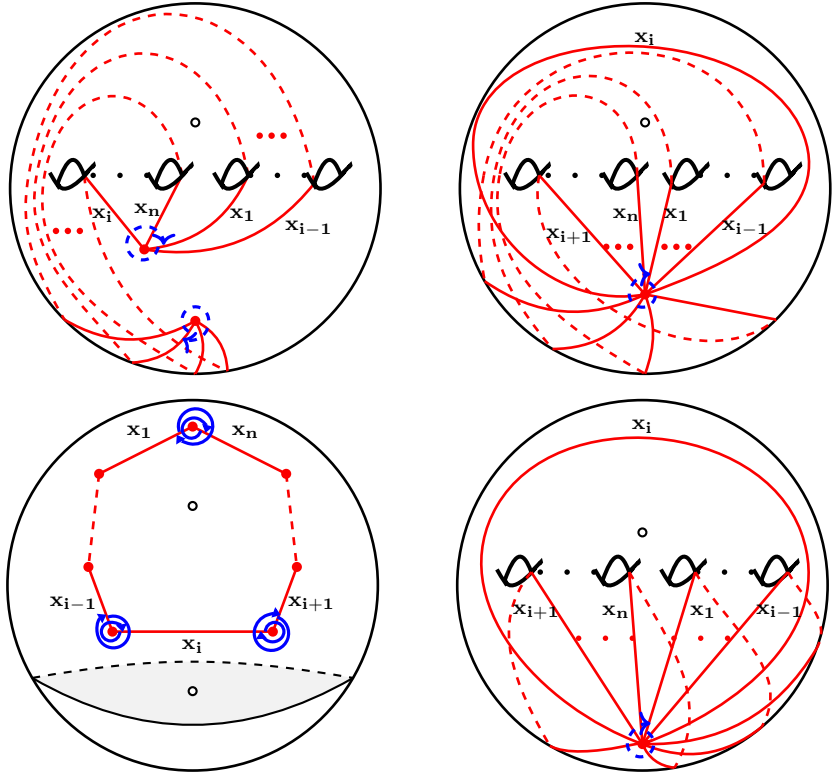

We now generalize this notion to assemblies of multiple coupled cycles through the following construction.

\noindent \textbf{Construction:}
Let \(\Lambda\) be a nonempty finite index set. For each \(\lambda\in\Lambda\), take a pair of primitive oriented cycles \(\mathcal{C}_{\lambda,1}, \mathcal{C}_{\lambda,2}\) with the overlap pattern prescribed in Definition~\ref{def:couple_cycle}, and denote their common vertex set by \(P_\lambda\). This pair satisfies all conditions of Definition~\ref{def:couple_cycle}, except that its non-shared vertices are not required to be stable vertices. We call the triple \(\mathfrak C_\lambda = \bigl(P_\lambda;\mathcal{C}_{\lambda,1},\mathcal{C}_{\lambda,2}\bigr)\) a \emph{coupled skeleton}.

Let \(Q_{\mathrm{mc}}\) denote the subquiver of \(Q\) spanned by all vertices and arrows occurring in \(\mathcal C_{\lambda,i}\) for \(\lambda\in\Lambda\) and \(i\in\{1,2\}\). Setting \(I_{\mathrm{mc}}:=I\cap \mathbf{k}Q_{\mathrm{mc}}\), we obtain the bound subquiver \((Q_{\mathrm{mc}},I_{\mathrm{mc}})\), which contains all coupled skeletons and inherits its relations from \(I\).

Assume that the sets \(P_\lambda\) are pairwise disjoint. In the
subquiver \(Q_{\mathrm{mc}}\), identify all vertices of each
\(P_\lambda\) with a single marked vertex \(q_\lambda\),
simultaneously for all \(\lambda\in\Lambda\). This operation is
called \emph{formal contraction}. Under this contraction, an
endpoint lying in \(P_\lambda\) is replaced by the corresponding
marked vertex \(q_\lambda\), while the remaining part of the quiver
is left unchanged. Denote the resulting formal contracted diagram by $\overline Q_{\mathrm{mc}}$. In particular, if both endpoints of an arrow \(\alpha\) belong to
the same set \(P_\lambda\), then
$\overline s(\alpha)=\overline t(\alpha)=q_\lambda$, so that \(\alpha\) is represented by a loop at \(q_\lambda\) in
\(\overline Q_{\mathrm{mc}}\). For example, in configuration
\ref{C1} or \ref{C2} of Figure \ref{fig:coupled cycle}, the arrow
directly joining \(v_k\) and \(v_1\) is represented by such a loop
after contraction.

The diagram \(\overline Q_{\mathrm{mc}}\) is introduced only to
record the combinatorial intertwining of the cycles. It is not
regarded as a bound quiver and does not represent either a quotient
of \(A\) or a contraction of the bound quiver
\((Q_{\mathrm{mc}},I_{\mathrm{mc}})\).

For each \(\lambda\in\Lambda\) and \(i\in\{1,2\}\), the image of
\(\mathcal C_{\lambda,i}\) under the formal contraction is an
oriented closed walk in \(\overline Q_{\mathrm{mc}}\), consisting of
the same arrows in the same cyclic order. We denote it by
\(\overline{\mathcal C}_{\lambda,i}\) and call it the
\emph{induced closed walk} of \(\mathcal C_{\lambda,i}\).
If every vertex of \(\mathcal C_{\lambda,i}\) outside $\bigcup_{\mu\in\Lambda}P_\mu$ is stable with respect to \(\mathcal C_{\lambda,i}\), then
\(\overline{\mathcal C}_{\lambda,i}\) is called a
\emph{relatively essential induced cycle}. In other words, to check relative essentiality, we only consider the original vertices that are not contracted. The common vertices represented by marked vertices are temporarily ignored. If, after all marked vertices are expanded, the corresponding vertices in the common regions are also stable, then the original cycle \(\mathcal C_{\lambda,i}\) is an essential cycle. For each $\lambda\in\Lambda$, set $E_\lambda:=\left\{[\overline{\mathcal C}_{\lambda,1}]_{\mathrm{cyc}},[\overline{\mathcal C}_{\lambda,2}]_{\mathrm{cyc}}\right\}$.

\begin{definition}\label{def:multi-coupled-cycle}
 The bound subquiver $(Q_{\mathrm{mc}},I_{\mathrm{mc}})$ is
called a \emph{multi-coupled cycle} if the following conditions hold:

\begin{enumerate}[label=(\text{MC}\arabic*), ref=(\text{MC}\arabic*)]
    \item \label{MC1} For every $\lambda\in\Lambda$, both induced
    cycles $\overline{\mathcal C}_{\lambda,1}$ and
    $\overline{\mathcal C}_{\lambda,2}$ are relatively essential
    induced cycles.

    \item \label{MC2} For any $\lambda,\mu\in\Lambda$, there exists
    a finite sequence $\lambda=\lambda_0,\lambda_1,\ldots,
      \lambda_s=\mu$ such that $E_{\lambda_j}\cap E_{\lambda_{j+1}}\neq\varnothing$.
    In other words, any two coupled skeletons can be joined by a finite chain in which each skeleton shares an induced cycle with the next
    one.
\end{enumerate}

\end{definition}

\begin{remark}\label{rem:multi_coupled}
For a multi-coupled cycle, the following two points should be noted.
\begin{itemize}
    \item If \(|\Lambda|=1\), a single coupled skeleton is exactly a coupled cycle. If \(|\Lambda|\geq2\), \((Q_{\mathrm{mc}},I_{\mathrm{mc}})\) is referred to as a \emph {non-trivial multi-coupled cycle}.
    \item  The vertex \(q_\lambda\) is merely a formal marker representing the contracted common region \(P_\lambda\). It is not a vertex of the original quiver \(Q\), and no gentle or non-gentle type is assigned to it. Moreover, no relations or permitted compositions are defined at \(q_\lambda\). The formal contraction is used solely to suppress the internal configuration of \(P_\lambda\) and thereby display more clearly how the induced cycles are intertwined. All statements concerning relations, permitted compositions, and locally gentle covers are understood in the original bound subquiver \((Q_{\mathrm{mc}},I_{\mathrm{mc}})\).
\end{itemize}

\end{remark}

The formal contraction suppresses the internal structure of each common
region \(P_\lambda\). The following lemma recovers its \(J\)-permitted
behaviour in the original bound quiver: after \(q_\lambda\) is expanded
back to \(P_\lambda\), a permitted path either leaves
\(P_\lambda\) or yields a primitive permitted cycle within it.

An arrow \(\xi\) with \(s(\xi)\notin P_\lambda\) and \(t(\xi)\in P_\lambda\) is termed a \emph{boundary incoming arrow} of \(P_\lambda\), while an arrow \(\zeta\) satisfying \(s(\zeta)\in P_\lambda\) and \(t(\zeta)\notin P_\lambda\) is called a \emph{boundary outgoing arrow}. 
\begin{lemma}\label{lem:cou_ske_con}
Let \(\mathfrak C_\lambda = \bigl(P_\lambda;\mathcal C_{\lambda,1},\mathcal C_{\lambda,2}\bigr)\) with \(P_\lambda = \{v_1,\dots,v_k\}\) be a coupled skeleton, and let \(J\subseteq I\) be a saturated locally gentle cover of \(A\). Take any arrow \(\gamma\in (Q_{\mathrm{mc}})_1\) satisfying \(t(\gamma)\in P_\lambda\), and examine its \(J\)-admissible path once the formal vertex \(q_\lambda\) is unfolded back to \(P_\lambda\). Exactly one of the following two cases holds:
\begin{enumerate}[label=\textup{(\arabic*)}]
  \item The path exits \(P_\lambda\) after finitely many arrows; the prefix up to the first outgoing arrow from \(P_\lambda\) forms a finite \(J\)-admissible path within the coupled skeleton.
  \item The path remains inside \(P_\lambda\) indefinitely, so \(P_\lambda\) contains a primitive \(J\)-admissible cycle.
\end{enumerate}
In particular, if \(\gamma\) is a boundary incoming arrow, then the path from case (1) starts at \(\gamma\) and exits \(P_\lambda\) via a boundary outgoing arrow.
\end{lemma}

\begin{proof}
All permitted compositions are considered within the original bound quiver \((Q,J)\). Let \(\eta\in(Q_{\mathrm{mc}})_1\) satisfy \(t(\eta)=v_j\in P_\lambda\), and let \(\beta_1,\beta_2\) denote the two arrows of the coupled skeleton with source at \(v_j\). By the local structure of a coupled skeleton together with Condition~\ref{S1}, these are precisely all arrows of \(Q\) starting at \(v_j\). Since \((Q,J)\) is locally gentle, Conditions~\ref{G1} and~\ref{S2} imply, respectively, that at most one among \(\eta\beta_1,\eta\beta_2\) lies in \(J\), and that at most one of them lies outside \(J\). Consequently, exactly one of these two compositions avoids \(J\). It follows that every arrow whose target lies in \(P_\lambda\) admits a unique \(J\)-permitted successor within the coupled skeleton.

Starting from \(\gamma\), successively append these uniquely determined successors. If the resulting path first exits \(P_\lambda\) after finitely many steps, then case (1) holds. Otherwise, the path stays within the finite subquiver of \(Q_{\mathrm{mc}}\) contained in the common region, so some arrow appears twice. The segment between two consecutive occurrences of this arrow forms a \(J\)-permitted oriented closed walk. Whenever this walk is a positive power of a shorter oriented closed walk, the shorter walk is also \(J\)-permitted. Iterating this reduction produces a primitive \(J\)-permitted cycle. This establishes case (2). The final claim follows directly from the definitions of the boundary arrows.

It remains only to treat the loop cases within the original quiver. If \(k>1\), every loop at the formal vertex \(q_\lambda\) comes from an arrow of \(Q\) whose two endpoints lie in \(P_\lambda\), and is thus already covered by the preceding argument. Now suppose \(P_\lambda=\{v\}\). If \(v\) admits exactly one loop \(\varepsilon\), then the \(J\)-permitted path either exits \(v\) along the other branch of the skeleton or stays at \(v\) and repeatedly traverses \(\varepsilon\). In the latter case \(\varepsilon^2\notin J\), so \(\varepsilon\) itself is a primitive \(J\)-permitted cycle. If \(v\) carries two distinct loops \(a\) and \(b\), Example~\ref{ex:two_loops} shows that at least one among \(a\), \(b\), and \(ab\) is a primitive \(J\)-permitted cycle.
\end{proof}

The following result ensures that the presence of a multi-coupled cycle invariably guarantees a red puncture.

\begin{proposition}\label{prop:multi_coupled}
Let \(A=\mathbf{k}Q/I\) be a string algebra. If \(Q\) contains a
multi-coupled cycle \((Q_{\mathrm{mc}},I_{\mathrm{mc}})\), then every
geometric model of \(A\) contains a red puncture.
\end{proposition}

\begin{proof}
Let \(J\subseteq I\) be an arbitrary saturated locally gentle cover of \(A\). All \(J\)-permitted compositions are considered within the original bound quiver \((Q,J)\); the formal contraction retains only the combinatorial configuration of the coupled skeletons.

Let \(\mathscr A\subseteq Q_1\) be the finite set of original arrows occurring in the coupled skeletons defining the multi‑coupled cycle, and set \(P=\bigcup_{\lambda\in\Lambda}P_\lambda\). Choose an arrow \(p_1\in\mathscr A\). We construct a \(J\)-permitted path recursively. Suppose that \(p_1p_2\cdots p_n\) has been constructed, with every \(p_i\in\mathscr A\), and put \(\alpha=p_n\) and \(v=t(\alpha)\). If \(v\notin P\), then \(\alpha\) occurs in some \(\mathcal C_{\lambda,i}\). Let \(\beta\) be the arrow immediately following \(\alpha\) on this cycle. By Condition~\ref{MC1}, the vertex \(v\) is stable with respect to \(\mathcal C_{\lambda,i}\). Hence \(\alpha\beta\notin J\), and the path extends by the arrow \(\beta\in\mathscr A\). Suppose that \(v\in P_\lambda\). Such a \(\lambda\) is unique because the common regions are pairwise disjoint. By Lemma~\ref{lem:cou_ske_con}, either \(P_\lambda\) contains a primitive \(J\)-permitted cycle, in which case there is nothing further to prove, or \(\alpha\) extends to a finite \(J\)-permitted path through \(P_\lambda\) whose final arrow leaves \(P_\lambda\). In the latter case, all arrows of this path belong to the coupled skeleton \(\mathfrak C_\lambda\), and hence to \(\mathscr A\). The construction may therefore be repeated from its final arrow.

Unless a primitive \(J\)-permitted cycle has already been obtained inside a common region, this procedure produces a right‑infinite sequence of composable arrows $p_1p_2p_3\cdots$ in \(Q\), such that
$p_i\in\mathscr A$, $p_ip_{i+1}\notin J$, $ \forall i\geq1$. Since \(\mathscr A\) is finite, choose \(s\) minimal such that \(p_s=p_r\) for some \(r<s\). Then $c=p_rp_{r+1}\cdots p_{s-1}$ is an oriented closed walk. Every cyclic length‑two subpath of \(c\) lies outside \(J\), since in particular $p_{s-1}p_r=p_{s-1}p_s\notin J$. Thus \(c\) is \(J\)-permitted. By the minimality of \(s\), its arrows are pairwise distinct; hence \(c\) is not a proper power of a shorter oriented closed walk and is therefore primitive.

We have proved that \((Q,J)\) contains a primitive \(J\)-permitted cycle for every saturated locally gentle cover \(J\) of \(A\). By Theorem~\ref{string algebra--labelled tiling algebra}, the geometric model determined by \(J\) contains a red puncture. Since \(J\) was arbitrary, every geometric model of \(A\) contains a red puncture.
\end{proof}

\subsection{Geometric models with red punctures}

We now present sufficient conditions which guarantee red punctures appear in all geometric models of a string algebra.

\begin{theorem}\label{thm: all-red puncture}
    Let \(A=\mathbf{k}Q/I\) be a string algebra. If \(Q\) contains an
essential cycle, a shuttle cycle, or a multi-coupled cycle, then every geometric model of \(A\) contains a red
puncture.
\end{theorem}
\begin{proof}
The three cases follow respectively from
Propositions~\ref{prop: essential cycle},
\ref{prop:shuttle_cycle}, and~\ref{prop:multi_coupled}.
\end{proof}
This naturally raises the following question.

\begin{question}
    Is the sufficient condition given in Theorem \ref{thm: all-red puncture} also necessary?
\end{question}
Although this question has not been fully resolved, the following lemma provides supporting evidence: if every saturated locally gentle cover contains a primitive permitted cycle, then such a cycle can always be chosen to avoid  non-gentle vertices of Type~\textup{(III)}.

\begin{lemma}\label{lem:type-III-elimination}
Let \(A=\mathbf{k}Q/I\) be a string algebra. If every admissible deletion datum \(\delta\in\mathfrak D(A)\) gives a bound quiver \((Q,J_\delta)\) possessing a primitive \(J_\delta\)-permitted cycle, then for any \(\delta\in\mathfrak D(A)\), there exists a primitive \(J_\delta\)-permitted cycle avoiding all  non-gentle vertices of Type~\textup{(III)}.
\end{lemma}

\begin{proof}
We first establish a local observation. Let \(v\) be a non-gentle vertex of Type~\textup{(III)}, and let \(\delta,\delta'\in\mathfrak D(A)\) satisfy \(\delta(u)=\delta'(u)\) for all \(u\neq v\) while \(\delta(v)\neq\delta'(v)\). We claim that \((Q,J_\delta)\) and \((Q,J_{\delta'})\) cannot both contain a permitted cycle passing through \(v\).

We only consider vertices of Type~\textup {(III.a)}, the case of Type~\textup {(III.b)} being similar. First assume \(v\) has no loop. Let \(a,b\) be the incoming arrows of \(v\) and \(c\) its unique outgoing arrow, so \(ac,bc\in I\). Swap \(a,b\) if needed, then the two local deletion rules are \(\delta(v)=\{ac\}\) and \(\delta'(v)=\{bc\}\). We get \(ac\notin J_\delta, bc\in J_\delta, ac\in J_{\delta'}, bc\notin J_{\delta'}\).
Suppose \((Q,J_\delta)\) and \((Q,J_{\delta'})\) both contain permitted cycles \(\mathcal C,\mathcal C'\) passing through \(v\). Only \(c\) can leave \(v\), so \(\mathcal C\) uses \(ac\) and \(\mathcal C'\) uses \(bc\) to cross \(v\). Track both cycles starting from \(c\) until returning to \(v\). No path of length two arising before returning to \(v\) has \(v\) as its middle vertex, so each such path either lies in both \(J_\delta\) and \(J_{\delta'}\) or lies outside both ideals. By \ref{S2}, every arrow has exactly one valid successor. The two paths must match fully back to \(v\), which contradicts their final incoming arrows being \(a\) and \(b\) separately. This logic still holds for other Type \textup{(III)} vertices, as \(\delta,\delta'\) perform the same deletions at those vertices. This reasoning applies to cycles sharing arrows or passing through many non-gentle vertices of Type \textup{(III)}.
Now let the unique outgoing arrow \(c\) of \(v\) be a loop. Any permitted cycle through \(v\) must contain \(c^2\), since traversing \(c\) leaves no other outgoing arrow available. The two local deletion choices toggle whether \(c^2\) belongs to the ideal, so at most one of $(Q,J_\delta)$ and $(Q,J_{\delta'})$ can possess a permitted cycle passing through \(v\). 

Let \(v_1,\ldots,v_m\) denote all Type~\textup{(III)} non-gentle vertices in \(Q\). We assume \(m\geq1\) and proceed by induction on \(k\) to show: for every \(1\leq k\leq m\) and every \(\delta\in\mathfrak D(A)\), \((Q,J_\delta)\) admits a primitive permitted cycle avoiding \(v_1,\ldots,v_k\).
For the base case \(k=1\), argue by contradiction. Suppose there exists \(\delta_0\in\mathfrak D(A)\) such that every primitive \(J_{\delta_0}\)-permitted cycle passes through \(v_1\). Construct \(\delta_1\) by only switching the local deletion choice at \(v_1\). By the local observation, \((Q,J_{\delta_1})\) has no permitted cycle through \(v_1\). By assumption, \((Q,J_{\delta_1})\) admits some primitive permitted cycle \(\mathcal C'\), which must avoid \(v_1\). Since \(\delta_0\) and \(\delta_1\) differ only at \(v_1\), every length‑2 subpath of \(\mathcal C'\) belongs to \(J_{\delta_0}\) precisely when it belongs to \(J_{\delta_1}\), so \(\mathcal C'\) is also \(J_{\delta_0}\)-permitted. Primitivity depends only on the underlying closed walk, not the ideal, giving a primitive \(J_{\delta_0}\)-permitted cycle avoiding \(v_1\), a contradiction.

Next fix \(1\leq k<m\) and assume the claim holds for k. Suppose the claim fails for \(k+1\). Then there exists \(\delta_0\in\mathfrak D(A)\) such that every primitive \(J_{\delta_0}\)-permitted cycle avoiding \(v_1,\dots,v_k\) passes through \(v_{k+1}\), and the induction hypothesis guarantees such cycles exist. Define \(\delta_1\) by changing only the deletion rule at \(v_{k+1}\). By the local observation, \((Q,J_{\delta_1})\) admits no permitted cycle through \(v_{k+1}\). Applying the inductive hypothesis to \(\delta_1\) yields a primitive \(J_{\delta_1}\)-permitted cycle \(\mathcal C'\) avoiding \(v_1,\dots,v_k\), and hence also \(v_{k+1}\). Since \(\delta_0\) and \(\delta_1\) differ only at \(v_{k+1}\), \(\mathcal C'\) is primitive and \(J_{\delta_0}\)-permitted, contradicting the assumption on \(\delta_0\).
The induction step closes, and setting \(k=m\) yields the desired conclusion.
\end{proof}

\section{Examples}
In this section, we will give some examples to explain our results.
\begin{example}\label{ex:label set}
Let $A = \mathbf{k}Q/I$ with $Q$  the quiver
\begin{tikzcd}[row sep=small, column sep=small]
1 \arrow[r, "d"]  & 5 \arrow[r, "b"] & 7 \arrow[r, "a"] & 8                 \\
2 \arrow[ru, "e"] & 4 \arrow[r, "f"] & 6 \arrow[u, "c"] & 3 \arrow[l, "g"']
\end{tikzcd}
 and admissible ideal $I = <db, fc, gc, ba, ca>$. After transforming the non-gentle vertices $\{6, 7\}  \subseteq  Q_{0}$, we obtain the corresponding four gentle algebras \( B_i = \mathbf{k}Q/J_i, i=1,2,3,4 \), where \( J_1 = < db, gc, ba >\), \( J_2 = < db, fc, ba > \), \( J_3 = < db, gc, ca > \), \( J_4 = < db, fc, ca > \). They are easily seen to be pairwise isomorphic. However, since the non-gentle vertex $7$ fails to satisfy \ref{U2}, the geometric models of algebra $A$ are not unique. Furthermore, as the quiver $Q$ contains no cycles, none of its geometric models admit a red puncture, see Figure \ref{fig:label set}.

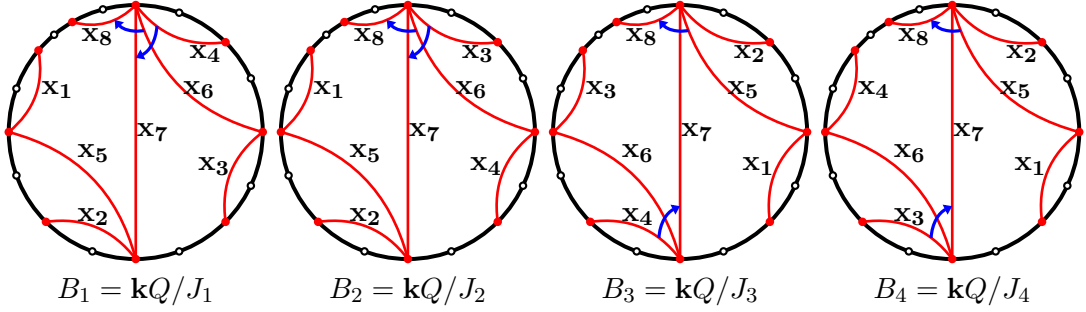
\begin{figure}[htbp]
	\begin{center}
		
\begin{tikzpicture}[scale=0.28]
    \draw[line width=1.5pt,fill=white] (0,0) circle (6cm);
    \path 
    (90:6) coordinate (b1) (45:6) coordinate (b2) (0:6) coordinate (b3)
    (-45:6) coordinate (b4) (-90:6) coordinate (b5) (-135:6) coordinate (b6)
    (180:6) coordinate (b7) (140:6) coordinate (b8) (120:6) coordinate (b9)
    (70:6) coordinate (r1) (25:6) coordinate (r2) (-25:6) coordinate (r3)
    (-70:6) coordinate (r4) (-110:6) coordinate (r5) (-160:6) coordinate (r6)
    (160:6) coordinate (r7) (130:6) coordinate (r8) (105:6) coordinate (r9);
    \draw[line width=1pt,red] (b1) to (b5);
    \draw[bend right,line width=1pt,red] (b1) to (b2);
    \draw[bend right,line width=1pt,red] (b1) to (b3);
    \draw[bend right,line width=1pt,red] (b3) to (b4);
    \draw[bend right,line width=1pt,red] (b9) to (b1);
    \draw[bend right,line width=1pt,red] (b5) to (b6);
    \draw[bend right,line width=1pt,red] (b5) to (b7);
    \draw[bend right,line width=1pt,red] (b7) to (b8);
    \draw[bend left,bluearrow] (0.4,4.8) to (-1,5.3);
    \draw[bend left,bluearrow] (1,5) to (0,3.5);
    \draw[thick,red, fill=red] 
    (b1) circle (0.15cm) (b2) circle (0.15cm) (b3) circle (0.15cm)
    (b4) circle (0.15cm) (b5) circle (0.15cm) (b6) circle (0.15cm)
    (b7) circle (0.15cm) (b8) circle (0.15cm) (b9) circle (0.15cm);	
    \draw[thick,black, fill=white] 
    (r1) circle (0.15cm) (r2) circle (0.15cm) (r3) circle (0.15cm)
    (r4) circle (0.15cm) (r5) circle (0.15cm) (r6) circle (0.15cm)
    (r7) circle (0.15cm) (r8) circle (0.15cm) (r9) circle (0.15cm);
    
    \draw(-3.7,2) node[black] {$\mathbf{x_{1}}$};
    \draw(-2,-4) node[black] {$\mathbf{x_{2}}$};
    \draw(3.3,3.7) node[black] {$\mathbf{x_{4}}$};
    \draw(3.7,-1.6) node[black] {$\mathbf{x_{3}}$};
    \draw(-2,-1) node[black] {$\mathbf{x_{5}}$};
    \draw(3,2) node[black] {$\mathbf{x_{6}}$};
    \draw(0.8,0) node[black] {$\mathbf{x_{7}}$};
    \draw(-1.8,4.5) node[black] {$\mathbf{x_{8}}$};
    
    \draw(0,-7.5) node[black] {$B_{1}=\mathbf{k}Q/J_{1}$};
\end{tikzpicture}
\begin{tikzpicture}[scale=0.28]
    \draw[line width=1.5pt,fill=white] (0,0) circle (6cm);
    \path 
    (90:6) coordinate (b1) (45:6) coordinate (b2) (0:6) coordinate (b3)
    (-45:6) coordinate (b4) (-90:6) coordinate (b5) (-135:6) coordinate (b6)
    (180:6) coordinate (b7) (140:6) coordinate (b8) (120:6) coordinate (b9)
    (70:6) coordinate (r1) (25:6) coordinate (r2) (-25:6) coordinate (r3)
    (-70:6) coordinate (r4) (-110:6) coordinate (r5) (-160:6) coordinate (r6)
    (160:6) coordinate (r7) (130:6) coordinate (r8) (105:6) coordinate (r9);
    \draw[line width=1pt,red] (b1) to (b5);
    \draw[bend right,line width=1pt,red] (b1) to (b2);
    \draw[bend right,line width=1pt,red] (b1) to (b3);
    \draw[bend right,line width=1pt,red] (b3) to (b4);
    \draw[bend right,line width=1pt,red] (b9) to (b1);
    \draw[bend right,line width=1pt,red] (b5) to (b6);
    \draw[bend right,line width=1pt,red] (b5) to (b7);
    \draw[bend right,line width=1pt,red] (b7) to (b8);
    \draw[bend left , bluearrow] (0.4,4.8) to (-1,5.3);
    \draw[bend left,bluearrow] (1,5) to (0,3.5);
    \draw[thick,red, fill=red] 
    (b1) circle (0.15cm) (b2) circle (0.15cm) (b3) circle (0.15cm)
    (b4) circle (0.15cm) (b5) circle (0.15cm) (b6) circle (0.15cm)
    (b7) circle (0.15cm) (b8) circle (0.15cm) (b9) circle (0.15cm);	
    \draw[thick,black, fill=white] 
    (r1) circle (0.15cm) (r2) circle (0.15cm) (r3) circle (0.15cm)
    (r4) circle (0.15cm) (r5) circle (0.15cm) (r6) circle (0.15cm)
    (r7) circle (0.15cm) (r8) circle (0.15cm) (r9) circle (0.15cm);
    
    \draw(-3.7,2) node[black] {$\mathbf{x_{1}}$};
    \draw(-2,-4) node[black] {$\mathbf{x_{2}}$};
    \draw(3.3,3.7) node[black] {$\mathbf{x_{3}}$};
    \draw(3.7,-1.6) node[black] {$\mathbf{x_{4}}$};
    \draw(-2,-1) node[black] {$\mathbf{x_{5}}$};
    \draw(3,2) node[black] {$\mathbf{x_{6}}$};
    \draw(0.8,0) node[black] {$\mathbf{x_{7}}$};
    \draw(-1.8,4.5) node[black] {$\mathbf{x_{8}}$};
    
    \draw(0,-7.5) node[black] {$B_{2}=\mathbf{k}Q/J_{2}$};
\end{tikzpicture}
\begin{tikzpicture}[scale=0.28]
    \draw[line width=1.5pt,fill=white] (0,0) circle (6cm);
    \path 
    (90:6) coordinate (b1) (45:6) coordinate (b2) (0:6) coordinate (b3)
    (-45:6) coordinate (b4) (-90:6) coordinate (b5) (-135:6) coordinate (b6)
    (180:6) coordinate (b7) (140:6) coordinate (b8) (120:6) coordinate (b9)
    (70:6) coordinate (r1) (25:6) coordinate (r2) (-25:6) coordinate (r3)
    (-70:6) coordinate (r4) (-110:6) coordinate (r5) (-160:6) coordinate (r6)
    (160:6) coordinate (r7) (130:6) coordinate (r8) (105:6) coordinate (r9);
    \draw[line width=1pt,red] (b1) to (b5);
    \draw[bend right,line width=1pt,red] (b1) to (b2);
    \draw[bend right,line width=1pt,red] (b1) to (b3);
    \draw[bend right,line width=1pt,red] (b3) to (b4);
    \draw[bend right,line width=1pt,red] (b9) to (b1);
    \draw[bend right,line width=1pt,red] (b5) to (b6);
    \draw[bend right,line width=1pt,red] (b5) to (b7);
    \draw[bend right,line width=1pt,red] (b7) to (b8);
    \draw[bend left,bluearrow] (0.4,4.8) to (-1,5.3);
    \draw[bend left,bluearrow] (-1,-5) to (0,-3.5);
    \draw[thick,red, fill=red] 
    (b1) circle (0.15cm) (b2) circle (0.15cm) (b3) circle (0.15cm)
    (b4) circle (0.15cm) (b5) circle (0.15cm) (b6) circle (0.15cm)
    (b7) circle (0.15cm) (b8) circle (0.15cm) (b9) circle (0.15cm);	
    \draw[thick,black, fill=white] 
    (r1) circle (0.15cm) (r2) circle (0.15cm) (r3) circle (0.15cm)
    (r4) circle (0.15cm) (r5) circle (0.15cm) (r6) circle (0.15cm)
    (r7) circle (0.15cm) (r8) circle (0.15cm) (r9) circle (0.15cm);
    
    \draw(-3.7,2) node[black] {$\mathbf{x_{3}}$};
    \draw(-2,-4) node[black] {$\mathbf{x_{4}}$};
    \draw(3.3,3.7) node[black] {$\mathbf{x_{2}}$};
    \draw(3.7,-1.6) node[black] {$\mathbf{x_{1}}$};
    \draw(-2,-1) node[black] {$\mathbf{x_{6}}$};
    \draw(3,2) node[black] {$\mathbf{x_{5}}$};
    \draw(0.8,0) node[black] {$\mathbf{x_{7}}$};
    \draw(-1.8,4.5) node[black] {$\mathbf{x_{8}}$};
    
    \draw(0,-7.5) node[black] {$B_{3}=\mathbf{k}Q/J_{3}$};
\end{tikzpicture}
\begin{tikzpicture}[scale=0.28]
    \draw[line width=1.5pt,fill=white] (0,0) circle (6cm);
    \path 
    (90:6) coordinate (b1) (45:6) coordinate (b2) (0:6) coordinate (b3)
    (-45:6) coordinate (b4) (-90:6) coordinate (b5) (-135:6) coordinate (b6)
    (180:6) coordinate (b7) (140:6) coordinate (b8) (120:6) coordinate (b9)
    (70:6) coordinate (r1) (25:6) coordinate (r2) (-25:6) coordinate (r3)
    (-70:6) coordinate (r4) (-110:6) coordinate (r5) (-160:6) coordinate (r6)
    (160:6) coordinate (r7) (130:6) coordinate (r8) (105:6) coordinate (r9);
    \draw[line width=1pt,red] (b1) to (b5);
    \draw[bend right,line width=1pt,red] (b1) to (b2);
    \draw[bend right,line width=1pt,red] (b1) to (b3);
    \draw[bend right,line width=1pt,red] (b3) to (b4);
    \draw[bend right,line width=1pt,red] (b9) to (b1);
    \draw[bend right,line width=1pt,red] (b5) to (b6);
    \draw[bend right,line width=1pt,red] (b5) to (b7);
    \draw[bend right,line width=1pt,red] (b7) to (b8);
    \draw[bend left,bluearrow] (0.4,4.8) to (-1,5.3);
    \draw[bend left,bluearrow] (-1,-5) to (0,-3.5);
    \draw[thick,red, fill=red] 
    (b1) circle (0.15cm) (b2) circle (0.15cm) (b3) circle (0.15cm)
    (b4) circle (0.15cm) (b5) circle (0.15cm) (b6) circle (0.15cm)
    (b7) circle (0.15cm) (b8) circle (0.15cm) (b9) circle (0.15cm);	
    \draw[thick,black, fill=white] 
    (r1) circle (0.15cm) (r2) circle (0.15cm) (r3) circle (0.15cm)
    (r4) circle (0.15cm) (r5) circle (0.15cm) (r6) circle (0.15cm)
    (r7) circle (0.15cm) (r8) circle (0.15cm) (r9) circle (0.15cm);
    
    \draw(-3.7,2) node[black] {$\mathbf{x_{4}}$};
    \draw(-2,-4) node[black] {$\mathbf{x_{3}}$};
    \draw(3.3,3.7) node[black] {$\mathbf{x_{2}}$};
    \draw(3.7,-1.6) node[black] {$\mathbf{x_{1}}$};
    \draw(-2,-1) node[black] {$\mathbf{x_{6}}$};
    \draw(3,2) node[black] {$\mathbf{x_{5}}$};
    \draw(0.8,0) node[black] {$\mathbf{x_{7}}$};
    \draw(-1.8,4.5) node[black] {$\mathbf{x_{8}}$};
    
    \draw(0,-7.5) node[black] {$B_{4}=\mathbf{k}Q/J_{4}$};
\end{tikzpicture}	
	\end{center}
	
	\caption{Example \ref{ex:label set}: The geometric models of \(B_1,B_2\) are equivalent, as are those of \(B_3,B_4\). Despite identical surface dissections, the two families aren't equivalent due to different label positions.}
	\label{fig:label set}
\end{figure}

\end{example}

\begin{example}\label{ex-sys I}
Let $A = \mathbf{k}Q/I$ with $Q$  the quiver
\begin{tikzcd}[row sep=small, column sep=small]
1 \arrow[rd, "a"'] &                                    & 3 \arrow[ll, "e"] \\
                   & 2 \arrow[ru, "c"'] \arrow[rd, "d"] &                    \\
4 \arrow[ru, "b"]  &                                    & 5 \arrow[ll, "f"']
\end{tikzcd}
 and admissible ideal $I = <ad, ac, bc, bd, ea, fb>$ . We transform non-gentle vertex $2 \in Q_{0}$. We get two gentle algebras $B_{1} = \mathbf{k}Q/J_{1}$ and $B_{2} = \mathbf{k}Q/J_{2}$, where $J_{1}= < ac, bd, ea, fb>$ and $J_{2} = <ad, bc,ea, fb>$. Since the vertex $2$ fails to satisfy \ref{U1}  (a$_2$), $B_1$ and $B_2$ are non-isomorphic. This means the geometric models of algebra $A$ are not unique. Furthermore, every cycle in the quiver $Q$ contains a relation of Type \ref{rel:R1}, so none of the geometric models of algebra $A$ admit a red puncture, see Figure \ref{fig:connect}.

\begin{figure}[htbp]
	\begin{center}

		\begin{tikzpicture}[scale=0.3]
			\begin{scope}[shift={(-8, 0)}]
				\draw[line width=1.5pt,fill=white] (0,0) circle (6cm);
				\draw[line width=1.5pt,fill=gray!50] (0,2.5) circle (0.8cm);
				\draw[line width=1.5pt,fill=gray!50] (0,-2.5) circle (0.8cm);
				
				\path 
				(0:6) coordinate (b1_1)
				(90:1.7) coordinate (b2_1)
				(180:6) coordinate (b3_1)
				(-90:1.7) coordinate (b4_1)
				
				(90:6) coordinate (r1_1)
				(90:3.3) coordinate (r2_1)
				(-90:3.3) coordinate (r3_1)
				(-90:6) coordinate (r4_1);
				
				\draw[line width=1pt,red] (b4_1) to (b3_1);
				\draw[line width=1pt,red] (b3_1) to (b1_1) ;
				\draw[line width=1pt,red] (b1_1) to (b2_1);
				\draw[,line width=1pt,red] (b3_1) to[out=80,in=180](0,4) to[out=0,in=110](b1_1);
				\draw[,line width=1pt,red] (b3_1) to[out=-80,in=180](0,-4) to[out=0,in=-110](b1_1);
				\draw[bend left,bluearrow] (-5.5,1.2) to (-4.5,-0.4);
				\draw[bend left,bluearrow] (5.5,-1.2) to (4.5,0.4);
				
				\draw[thick, red,  fill=red] 
				(b1_1) circle (0.2cm)
				(b2_1) circle (0.2cm)
				(b3_1) circle (0.2cm)
				(b4_1) circle (0.2cm);
				
				\draw[thick,black, fill=white] 
				(r1_1) circle (0.2cm)
				(r2_1) circle (0.2cm)
				(r3_1) circle (0.2cm)
				(r4_1) circle (0.2cm);
				
				\draw(0,4.6) node[black] {$\mathbf{x_{1}}$};
				\draw(-1,0.5) node[black] {$\mathbf{x_{2}}$};
				\draw(3.5,1.2) node[black] {$\mathbf{x_{3}}$};
				\draw(-3.3,-1.4) node[black] {$\mathbf{x_{5}}$};
				\draw(0,-4.9) node[black] {$\mathbf{x_{4}}$};
				\draw(0,-7) node[black] {$B_{1}=\mathbf{k}Q/J_{1}$};
			\end{scope}
			
			\begin{scope}[shift={(8, 0)}]
				\draw[line width=1.5pt,fill=white] (0,0) circle (6cm);
				\draw[line width=1.5pt,fill=gray!50] (0,2.5) circle (0.8cm);
				\draw[line width=1.5pt,fill=gray!50] (0,-2.5) circle (0.8cm);
				
				\path 
				(0:6) coordinate (b1_2)
				(90:1.7) coordinate (b2_2)
				(180:6) coordinate (b3_2)
				(-90:1.7) coordinate (b4_2)
				
				(90:6) coordinate (r1_2)
				(90:3.3) coordinate (r2_2)
				(-90:3.3) coordinate (r3_2)
				(-90:6) coordinate (r4_2);
				
				\draw[line width=1pt,red] (b4_2) to (b1_2);
				\draw[line width=1pt,red] (b2_2) to (b4_2) ;
				\draw[line width=1pt,red] (b2_2) to (b3_2);
				\draw[line width=1pt,red] (b2_2) to[out=-10,in=-90](1.2,2.5) to[out=90,in=0] (0,4) to[out=180,in=90] (-1.2,2.5) to[out=-90,in=-170](b2_2);
				\draw[line width=1pt,red ] (b4_2) to[out=10,in=90](1.2,-2.5) to[out=-90,in=0] (0,-4) to[out=180,in=-90] (-1.2,-2.5) to[out=90,in=170](b4_2);
				\draw[bend left,bluearrow] (-0.7,-1.7) to (0.8,-1.3);
				\draw[bend left,bluearrow] (0.7,1.7) to (-0.8,1.3);
				
				\draw[thick, red, fill=red] 
				(b1_2) circle (0.2cm)
				(b2_2) circle (0.2cm)
				(b3_2) circle (0.2cm)
				(b4_2) circle (0.2cm);
				
				\draw[thick,black, fill=white] 
				(r1_2) circle (0.2cm)
				(r2_2) circle (0.2cm)
				(r3_2) circle (0.2cm)
				(r4_2) circle (0.2cm);
				
				\draw(0,-4.7) node[black] {$\mathbf{x_{1}}$};
				\draw(-0.8,0) node[black] {$\mathbf{x_{2}}$};
				\draw(2.7,-0.2) node[black] {$\mathbf{x_{3}}$};
				\draw(-2.8,1.5) node[black] {$\mathbf{x_{5}}$};
				\draw(0,4.5) node[black] {$\mathbf{x_{4}}$};
				\draw(0,-7) node[black] {$B_{2}=\mathbf{k}Q/J_{2}$};
			\end{scope}
		\end{tikzpicture}
		
		\caption{Example \ref{ex-sys I}: The geometric models of $B_{1}$ and $B_{2}$ aren't equivalent to each other. }
		\label{fig:connect}
	\end{center}
\end{figure}
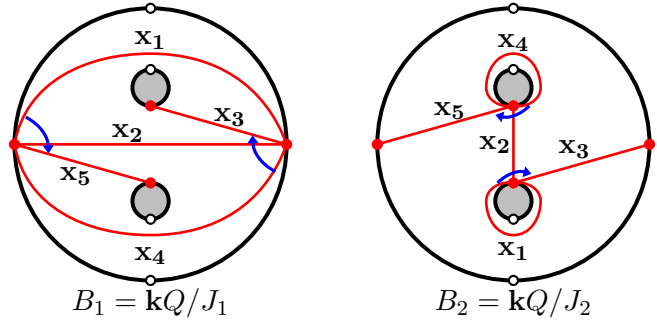
\end{example}

\begin{example}\label{ex:condition(3)}
Let $A = \mathbf{k}Q/I$ with $Q$  the quiver
\begin{tikzcd}[row sep=small, column sep=small]
  &                   & 3 \arrow[ld, "b"'] \arrow[rd, "d"] &                  &   \\
1 & 2 \arrow[l, "a"'] &                                    & 5 \arrow[r, "f"] & 6 \\
  &                   & 4 \arrow[lu, "c"] \arrow[ru, "e"'] &                  &  
\end{tikzcd}
and admissible ideal $I = <ba, ca, ef, df>$. We transform non-gentle vertices $\{2,5\} \subseteq Q_{0}$. We get the four gentle algebras $B_{i} = \mathbf{k}Q/J_{i}$, $i=1,2,3,4$, where $J_{1}= < ba, df>$, $J_{2} = <ca,ef>$, $J_{3}= <ba, ef>$ and $J_{4} = <ca,df>$. It is clear that \(B_{1} \cong B_{2}\) and \(B_{3} \cong B_{4}\). Since the vertices $2$ and $5$ fail to satisfy \ref{U3}, $B_2$ is not isomorphic to $B_3$. This implies that the geometric models of $A$ are not unique. In addition, the quiver $Q$ contains no cycles, so no geometric model of $A$ admits a red puncture, see Figure \ref{fig:condition3}.

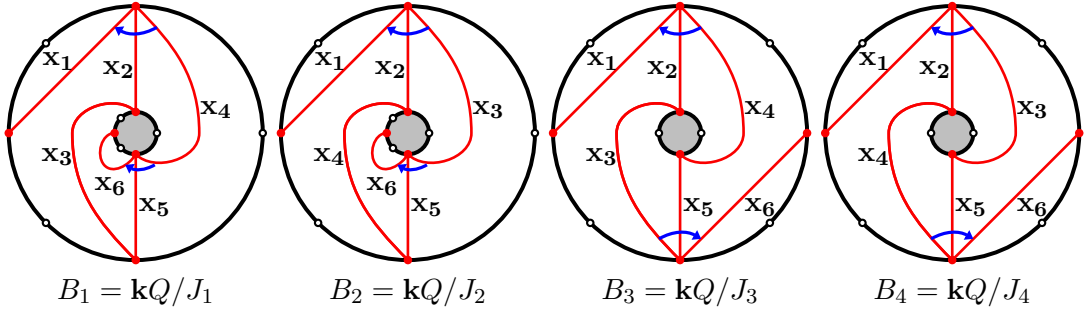
\begin{figure}[htbp]
	\begin{center}
\begin{tikzpicture}[scale=0.28]
	\draw[line width=1.5pt,fill=white] (0,0) circle (6cm);
	\draw[line width=1.5pt,fill=gray!50] (0,0) circle (1cm);
	
	\path 
	(90:6) coordinate (b1)
	(180:6) coordinate (b2)
	(-90:6) coordinate (b3)
	(90:1) coordinate (b4)
	(180:1) coordinate (b5)
	(-90:1) coordinate (b6)
	
	(135:6) coordinate (r1)
	(225:6) coordinate (r2)
	(0:6) coordinate (r3)
	(135:1) coordinate (r4)
	(225:1) coordinate (r5)
	(0:1) coordinate (r6);
	
	\draw[line width=1pt,red] (b1) to (b2);
	\draw[line width=1pt,red] (b1) to (b4);
	\draw[line width=1pt,red] (b3) to (b6);
	\draw[red,line width=1pt] (b1) to[out=-45,in=90](3,0) to[out=-90,in=-45](b6);
	\draw[red,line width=1pt] (b3) to[out=135,in=-90](-3,0) to[out=90,in=135](b4);
	\draw[red,line width=1pt] (b3) to[out=135,in=-90](-3,0) to[out=90,in=135](b4);
	\draw[red,line width=1pt] (b5) to[out=180,in=130] (-1.5,-1.5) to[out=-50,in=-120] (b6);
	
	\draw[bend left,bluearrow] (1,5) to (-1,5);
	\draw[bend left,bluearrow] (0.9,-1.5) to (-0.5,-1.5);
	
	\draw[thick,red, fill=red] 
	(b1) circle (0.15cm)
	(b2) circle (0.15cm)
	(b3) circle (0.15cm)
	(b4) circle (0.15cm)
	(b5) circle (0.15cm)
	(b6) circle (0.15cm);
	
	\draw[thick,black, fill=white] 
	(r1) circle (0.15cm)
	(r2) circle (0.15cm)
	(r3) circle (0.15cm)
	(r4) circle (0.15cm)
	(r5) circle (0.15cm)
	(r6) circle (0.15cm);
	
	\draw(-3.65,3.4) node[black] {$\mathbf{x_{1}}$};
	\draw(-0.8,3) node[black] {$\mathbf{x_{2}}$};
	\draw(3.8,1.1) node[black] {$\mathbf{x_{4}}$};
	\draw(-3.7,-1.1) node[black] {$\mathbf{x_{3}}$};
	\draw(0.9,-3.5) node[black] {$\mathbf{x_{5}}$};
	\draw(-1.2,-2.5) node[black] {$\mathbf{x_{6}}$};
	
	\draw(0,-7.5) node[black] {$B_{1}=\mathbf{k}Q/J_{1}$};
\end{tikzpicture}
\begin{tikzpicture}[scale=0.28]
	\draw[line width=1.5pt,fill=white] (0,0) circle (6cm);
	\draw[line width=1.5pt,fill=gray!50] (0,0) circle (1cm);
	
	\path 
	(90:6) coordinate (b1)
	(180:6) coordinate (b2)
	(-90:6) coordinate (b3)
	(90:1) coordinate (b4)
	(180:1) coordinate (b5)
	(-90:1) coordinate (b6)
	
	(135:6) coordinate (r1)
	(225:6) coordinate (r2)
	(0:6) coordinate (r3)
	(135:1) coordinate (r4)
	(225:1) coordinate (r5)
	(0:1) coordinate (r6);
	
	\draw[line width=1pt,red] (b1) to (b2);
	\draw[line width=1pt,red] (b1) to (b4);
	\draw[line width=1pt,red] (b3) to (b6);
	\draw[red,line width=1pt] (b1) to[out=-45,in=90](3,0) to[out=-90,in=-45](b6);
	\draw[red,line width=1pt] (b3) to[out=135,in=-90](-3,0) to[out=90,in=135](b4);
	\draw[red,line width=1pt] (b3) to[out=135,in=-90](-3,0) to[out=90,in=135](b4);
	\draw[red,line width=1pt] (b5) to[out=180,in=130] (-1.5,-1.5) to[out=-50,in=-120] (b6);
	
	\draw[bend left,bluearrow] (1,5) to (-1,5);
	\draw[bend left,bluearrow] (0.9,-1.5) to (-0.5,-1.5);
	
	\draw[thick, red, fill=red] 
	(b1) circle (0.15cm)
	(b2) circle (0.15cm)
	(b3) circle (0.15cm)
	(b4) circle (0.15cm)
	(b5) circle (0.15cm)
	(b6) circle (0.15cm);
	
	\draw[thick,black, fill=white] 
	(r1) circle (0.15cm)
	(r2) circle (0.15cm)
	(r3) circle (0.15cm)
	(r4) circle (0.15cm)
	(r5) circle (0.15cm)
	(r6) circle (0.15cm);
	
	\draw(-3.65,3.4) node[black] {$\mathbf{x_{1}}$};
	\draw(-0.8,3) node[black] {$\mathbf{x_{2}}$};
	\draw(3.8,1.1) node[black] {$\mathbf{x_{3}}$};
	\draw(-3.7,-1.1) node[black] {$\mathbf{x_{4}}$};
	\draw(0.9,-3.5) node[black] {$\mathbf{x_{5}}$};
	\draw(-1.2,-2.5) node[black] {$\mathbf{x_{6}}$};
	
	\draw(0,-7.5) node[black] {$B_{2}=\mathbf{k}Q/J_{2}$};
\end{tikzpicture}
\begin{tikzpicture}[scale=0.28]
	\draw[line width=1.5pt,fill=white] (0,0) circle (6cm);
	\draw[line width=1.5pt,fill=gray!50] (0,0) circle (1cm);
	
	\path 
	(90:6) coordinate (b1)
	(180:6) coordinate (b2)
	(-90:6) coordinate (b3)
	(90:1) coordinate (b4)
	(0:6) coordinate (b5)
	(-90:1) coordinate (b6)
	
	(135:6) coordinate (r1)
	(225:6) coordinate (r2)
	(-45:6) coordinate (r3)
	(0:1) coordinate (r4)
	(180:1) coordinate (r5)
	(45:6) coordinate (r6);
	
	\draw[line width=1pt,red] (b1) to (b2);
	\draw[line width=1pt,red] (b1) to (b4);
	\draw[line width=1pt,red] (b3) to (b6);
	\draw[red,line width=1pt] (b1) to[out=-45,in=90](3,0) to[out=-90,in=-45](b6);
	\draw[red,line width=1pt] (b3) to[out=135,in=-90](-3,0) to[out=90,in=135](b4);
	\draw[red,line width=1pt] (b3) to[out=135,in=-90](-3,0) to[out=90,in=135](b4);
	\draw[red,line width=1pt] (b5) to (b3);
	
	\draw[bend left,bluearrow] (1,5) to (-1,5);
	\draw[bend left,bluearrow] (-1,-5) to (1,-5);
	
	\draw[thick, red, fill=red] 
	(b1) circle (0.15cm)
	(b2) circle (0.15cm)
	(b3) circle (0.15cm)
	(b4) circle (0.15cm)
	(b5) circle (0.15cm)
	(b6) circle (0.15cm);
	
	\draw[thick,black, fill=white] 
	(r1) circle (0.15cm)
	(r2) circle (0.15cm)
	(r3) circle (0.15cm)
	(r4) circle (0.15cm)
	(r5) circle (0.15cm)
	(r6) circle (0.15cm);
	
	\draw(-3.65,3.4) node[black] {$\mathbf{x_{1}}$};
	\draw(-0.8,3) node[black] {$\mathbf{x_{2}}$};
	\draw(3.8,1.1) node[black] {$\mathbf{x_{4}}$};
	\draw(-3.7,-1.1) node[black] {$\mathbf{x_{3}}$};
	\draw(0.9,-3.5) node[black] {$\mathbf{x_{5}}$};
	\draw(3.8,-3.5) node[black] {$\mathbf{x_{6}}$};
	
	\draw(0,-7.5) node[black] {$B_{3}=\mathbf{k}Q/J_{3}$};
\end{tikzpicture}
\begin{tikzpicture}[scale=0.28]
	\draw[line width=1.5pt,fill=white] (0,0) circle (6cm);
	\draw[line width=1.5pt,fill=gray!50] (0,0) circle (1cm);
	
	\path 
	(90:6) coordinate (b1)
	(180:6) coordinate (b2)
	(-90:6) coordinate (b3)
	(90:1) coordinate (b4)
	(0:6) coordinate (b5)
	(-90:1) coordinate (b6)
	
	(135:6) coordinate (r1)
	(225:6) coordinate (r2)
	(-45:6) coordinate (r3)
	(0:1) coordinate (r4)
	(180:1) coordinate (r5)
	(45:6) coordinate (r6);
	
	\draw[line width=1pt,red] (b1) to (b2);
	\draw[line width=1pt,red] (b1) to (b4);
	\draw[line width=1pt,red] (b3) to (b6);
	\draw[red,line width=1pt] (b1) to[out=-45,in=90](3,0) to[out=-90,in=-45](b6);
	\draw[red,line width=1pt] (b3) to[out=135,in=-90](-3,0) to[out=90,in=135](b4);
	\draw[red,line width=1pt] (b3) to[out=135,in=-90](-3,0) to[out=90,in=135](b4);
	\draw[red,line width=1pt] (b5) to (b3);
	
	\draw[bend left,bluearrow] (1,5) to (-1,5);
	\draw[bend left,bluearrow] (-1,-5) to (1,-5);
	
	\draw[thick,red, fill=red] 
	(b1) circle (0.15cm)
	(b2) circle (0.15cm)
	(b3) circle (0.15cm)
	(b4) circle (0.15cm)
	(b5) circle (0.15cm)
	(b6) circle (0.15cm);
	
	\draw[thick,black, fill=white] 
	(r1) circle (0.15cm)
	(r2) circle (0.15cm)
	(r3) circle (0.15cm)
	(r4) circle (0.15cm)
	(r5) circle (0.15cm)
	(r6) circle (0.15cm);
	
	\draw(-3.65,3.4) node[black] {$\mathbf{x_{1}}$};
	\draw(-0.8,3) node[black] {$\mathbf{x_{2}}$};
	\draw(3.8,1.1) node[black] {$\mathbf{x_{3}}$};
	\draw(-3.7,-1.1) node[black] {$\mathbf{x_{4}}$};
	\draw(0.9,-3.5) node[black] {$\mathbf{x_{5}}$};
	\draw(3.8,-3.5) node[black] {$\mathbf{x_{6}}$};
	
	\draw(0,-7.5) node[black] {$B_{4}=\mathbf{k}Q/J_{4}$};
\end{tikzpicture}
	\end{center}
	\caption{Example \ref{ex:condition(3)}: The left pair of geometric models are equivalent, as are the right pair, but the left and right families aren't equivalent.}
	\label{fig:condition3}
\end{figure}

\end{example}

\begin{example}\label{ex:cycle-every surface}
    Let $A = \mathbf{k}Q/I$ with $Q$  the quiver
\begin{tikzcd}
1 \arrow[rr, "c" description] &  & 2 \arrow[ll, "a"', bend right] \arrow[ll, "b", bend left]
\end{tikzcd} and admissible ideal $I = <ac,bc,ca,cb>$. We transform non-gentle vertices $\{1,2\}$. We get the four (locally) gentle algebras $B_{i} = \mathbf{k}Q/J_{i}$, $i=1,2,3,4$, where $J_{1}= < ac, ca>$, $J_{2} = <bc, cb>$, $J_{3}= <bc, ca>$ and $J_{4} = <ac, cb> $. Clearly, $B_{1} \cong B_{2}$ are infinite-dimensional, whereas $B_{3} \cong B_{4}$ are finite-dimensional. Then the geometric models of $A$ are not unique. In addition, there exists an infinite path ${}^\infty (bc)^\infty$ (resp. ${}^\infty (ac)^\infty$) in the quiver corresponding to the algebra $B_1$ (resp. $B_2$), so its associated geometric model contains a red puncture. See Figure \ref{fig:2-point,3-arrow} for the concrete geometric models.

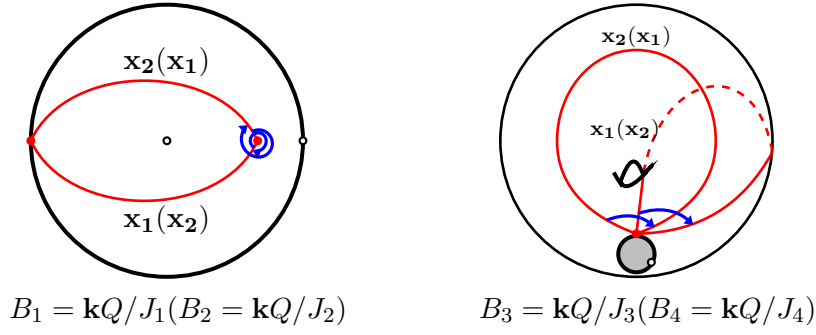
\begin{figure}[htbp]
	\begin{center}
		\begin{tikzpicture}[scale=0.3]
			\begin{scope}[xshift=-20,yshift=0cm]
                \draw[line width=1.5pt,fill=white] (0,0) circle (6cm);
				
				\path 
				(0:4) coordinate (b1)
				(180:6) coordinate (b2)
                (0:6) coordinate (r1)
                (0:0) coordinate (r2)
				;

                \draw[thick,red, fill=red] (b1) circle (0.15cm) (b2) circle (0.15cm);
                \draw[thick,black, fill=white] 
	              (r1) circle (0.15cm) (r2) circle (0.15cm);
                
				\draw[line width=1pt,red, bend left=65] (b2) to(b1);
				\draw[line width=1pt,red, bend right=65] (b2) to(b1);

                \draw[bluearrow] (3.85,-0.3) arc[start angle=-120, end angle=-350, radius=0.35] arc[start angle=-350, end angle=-470, radius=0.45] ;
                 \draw[bluearrow] (3.75,0.4) arc[start angle=120, end angle=-100, radius=0.6] arc[start angle=-100, end angle=-230, radius=0.8] ;

                 \draw(0.5,-7.5) node[black] {$B_{1}=\mathbf{k}Q/J_{1}(B_{2}=\mathbf{k}Q/J_{2})$};               
				\draw(0,3.5) node[black] {$\mathbf{x_{2}(x_{1})}$};
			    \draw(0,-3.5) node[black] {$\mathbf{x_{1}(x_{2})}$};
			\end{scope}
			
\begin{scope}[xshift=20cm,yshift=0cm]
\draw[line width=1pt,fill=white] (0,0) circle (6cm);
\draw[line width=1.5pt,fill=gray!50] (0, -5) circle (0.8cm);

\path 
    (-90:4.1) coordinate (b1)
    (-83:5.4) coordinate (r1);
\draw[thick,red, fill=red] (b1) circle(0.15cm) ;
\draw[thick,black, fill=white] (r1) circle(0.15cm);

\draw[line width=1pt,red] (b1)
to[out=160,in=-90] (-3.5,0)
to[out=90,in=180] (0,4)
to[out=0,in=90] (3.5,0)
to[out=-90,in=20](b1);
\draw[line width=1pt, red] (b1) to[out=80, in=80, looseness=0] (0.32, -1.4);
\draw[line width=1pt,dashed, red] (0.32, -1.4)to[out=80, in=100, looseness=2] (6,-0.5);
\draw[bend right,line width=1pt,red] (b1) to (6,-0.5);

\draw[line width=1.5pt] (-0.7, -1.8) to[out=70, in=110, looseness=2.5] (0.3, -1.5); 
\draw[line width=1.5pt] 
(-1.0,-1.4) to[out=-60, in=-66, looseness=2] (-0.7, -1.8) 
to[out=-70, in=48, looseness=2] (0.3, -1.5)
to[out=48, in=50, looseness=2](0.6,-1.2); 

\draw[bend left,bluearrow] (-1.3,-3.5) to (0.8,-3.8);
\draw[bend left,bluearrow] (0.1,-3.2) to (2.5,-3.7);

\draw(0,4.6) node[black] {\scriptsize$\mathbf{x_{2}(x_{1})}$};
\draw(-0.5,0.5) node[black] {\scriptsize$\mathbf{x_{1}(x_{2})}$};

\draw(0.5,-7.5) node[black] {$B_{3}=\mathbf{k}Q/J_{3}(B_{4}=\mathbf{k}Q/J_{4})$};
			\end{scope}
		\end{tikzpicture}
        
		\caption{Example \ref{ex:cycle-every surface}: The geometric models on the left contain a red puncture, while those on the right have no red puncture.}
		\label{fig:2-point,3-arrow}
	\end{center}
\end{figure}
\end{example}

\begin{example}\label{ex:cycle-every surface1}
    Let $A = \mathbf{k}Q/I$ with $Q$  the quiver
\begin{tikzcd}
1 \arrow[rr, "a" description, bend left=49] \arrow[rr, "b" description, bend left=19] &  & 2 \arrow[ll, "c" description, bend left=19] \arrow[ll, "d" description, bend left=49]
\end{tikzcd} and admissible ideal $I = <ac,ad,bc,bd,ca,cb,da,db>$. We transform non-gentle vertices $\{1,2\}$. We get the four locally gentle algebras $B_{i} = \mathbf{k}Q/J_{i}$, $i=1,2,3,4$, where $J_{1}= < ac,bd,ca,db>$, $J_{2} = <ad,bc,cb,da>$, $J_{3}= <ac,bd,cb,da>$ and $J_{4} = <ad,bc,ca,db > $. It is clear that \(B_{1} \cong B_{2}\) and \(B_{3} \cong B_{4}\). Since the vertices $1$ and $2$ fail to satisfy \ref{U1}(a$_2$), $B_1$ and $B_3$ are non-isomorphic. This implies that the geometric models of $A$ are not unique. Furthermore, $B_1,B_2$ carry two infinite paths apiece: ${}^\infty (ad)^\infty,{}^\infty (bc)^\infty$ and ${}^\infty (ac)^\infty,{}^\infty (bd)^\infty$; $B_3,B_4$ each have one: ${}^\infty (adbc)^\infty$ and ${}^\infty (acbd)^\infty$. Consequently, every geometric model of $A$ has a red puncture, see Figure \ref{fig: 2-point,4-arrow}.

\begin{figure}[htbp]
	\begin{center}

\begin{tikzpicture}[scale=0.3]
\begin{scope}[xshift=-17,yshift=0cm]
    \draw[line width=1,fill=white] (0,0) circle (7cm);
			\path 
				(40:5) coordinate (b1)
				(140:5) coordinate (b2)
                (90:3) coordinate (r1)
                (-90:4) coordinate (r2)
                (-30:7)  coordinate (u1)
                 (-150:7)  coordinate (u2)
				;
                \fill[gray!10] (u2) to[bend left] (u1) to[bend left] cycle;

                \draw[thick,red, fill=red] (b1) circle (0.15cm) (b2) circle (0.15cm);
                 \draw[thick,black, fill=white] (r1) circle (0.15cm) (r2) circle (0.15cm);

				\draw[,line width=1pt,red] (b2) to[out=80,in=180](0,5.7) to[out=0,in=110](b1);
				\draw[,line width=1pt,red] (b2) to[out=-80,in=180](0,1) to[out=0,in=-110](b1);

                \draw[line width=0.8pt,black, bend right] (u2) to (u1);
               \draw[line width=0.8pt,black,dashed,bend left] (u2) to (u1);

\draw[bluearrow] 
(-3.83,3.21) +(-80:0.42) arc[start angle=-90, end angle=-335, radius=0.42] arc[start angle=-335, end angle=-420, radius=0.6];
\draw[bluearrow] 
(-3.6,3.8)  arc[start angle=80, end angle=-200, radius=0.7] 
arc[start angle=-200, end angle=-260, radius=1];

\draw[bluearrow] 
(3.75,2.9) arc[start angle=-90, end angle=-335, radius=0.35] arc[start angle=-335, end angle=-440, radius=0.6];
\draw[bluearrow] 
(3.6,3.75)  arc[start angle=110, end angle=-200, radius=0.7] arc[start angle=-200, end angle=-225, radius=1.5] 
;

                 \draw(0.5,-8) node[black] {$B_{1}=\mathbf{k}Q/J_{1}(B_{2}=\mathbf{k}Q/J_{2})$};  
				\draw(0,4.6) node[black] {\scriptsize$\mathbf{x_{1}(x_{2})}$};
			    \draw(0,0.6) node[black] {\scriptsize$\mathbf{x_{2}(x_{1})}$};
				
			\end{scope}
			
			\begin{scope}[xshift=17cm,yshift=0cm]
				\draw[line width=1pt,fill=white] (0,0) circle (7cm);
				\path 
				(-90:6) coordinate (b1);
              
                 \draw[thick,red, fill=red] (b1) circle(0.15cm) ;
                 \draw[thick,black, fill=white] (0,2) circle(0.15cm) ;
                
                \draw[line width=1pt,red] (b1)
               to[out=130,in=-120] (-4,3)
               to[out=60,in=110] (4,3.6)
                to[out=-70,in=70] (b1);
                \draw[line width=1pt, red] (b1) to[out=80, in=80, looseness=0] (0.32, -1.4);
                \draw[line width=1pt,dashed, red] (0.32, -1.4)to[out=80, in=100, looseness=2] (6.9,-0.5);
                \draw[bend right=10,line width=1pt,red] (b1) to (6.9,-0.5);

                 \draw[line width=1.5pt] (-0.7, -1.8) to[out=70, in=110, looseness=2.5] (0.3, -1.5); 
                 \draw[line width=1.5pt] 
                 (-1.0,-1.4) to[out=-60, in=-66, looseness=2] (-0.7, -1.8) 
                 to[out=-70, in=48, looseness=2] (0.3, -1.5)
                 to[out=48, in=50, looseness=2](0.6,-1.2); 

                \draw[bend left,bluearrow] (-0.35,-5.6) to (0.4,-5.2);
                \draw[bend left,bluearrow] (0.1,-4.5) to (1.8,-5);

                \draw[bluearrow] 
(0.2,-5.65)arc[start angle=60, end angle=-230, radius=0.4] 
;

                  \draw[bluearrow] 
(0.8,-5.6)arc[start angle=0, end angle=-130, radius=1]arc[start angle=-130, end angle=-280, radius=1.2]  
;
           
         \draw(0.5,-8) node[black] {$B_{3}=\mathbf{k}Q/J_{3}(B_{4}=\mathbf{k}Q/J_{4})$}; 
				\draw(0,4.6) node[black] {\scriptsize$\mathbf{x_{2}(x_{1})}$};
			    \draw(-1,-3) node[black] {\scriptsize$\mathbf{x_{1}(x_{2})}$};
			\end{scope}
    
		\end{tikzpicture}  
				\caption{Example \ref{ex:cycle-every surface1}: The geometric models on the left carry two red punctures, while those on the right possess only one red puncture.}
		\label{fig: 2-point,4-arrow}
	\end{center}
\end{figure}
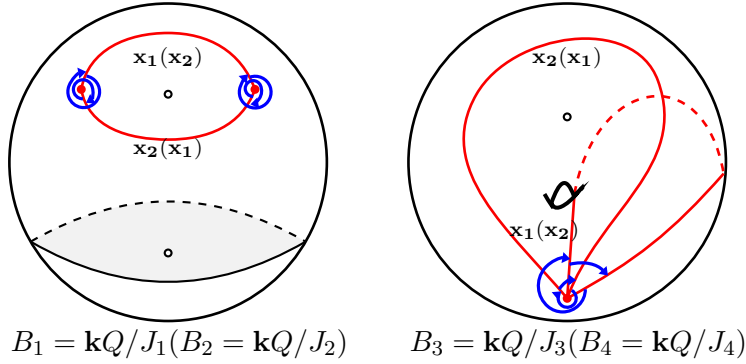

\end{example}
\begin{example}\label{ex:Type I+II+III}
    Let $A = \mathbf{k}Q/I$ with $Q$  the quiver
\begin{tikzcd}[row sep=small, column sep=small]
  1 \arrow[rr, "a_1" description, bend left] & & 2 \arrow[dd, "b_1" description, bend left] \arrow[ll, "a_2" description, shift left] \\
  \\
  4 \arrow[uu, "d" description, bend left] \arrow[rr, "c_2" description, shift left] & & 3 \arrow[ll, "c_1" description, bend left] \arrow[uu, "b_2" description, shift left]
\end{tikzcd}
and admissible ideal $I = <a_{1}a_{2}, a_{1}b_{1}, a_{2}a_{1}, b_{2}a_{2}, b_{2}b_{1}, b_{1}c_{1}, b_{1}b_{2}, c_{2}b_{2}, c_{1}c_{2}, c_{1}d, da_{1}>$. We transform the non-gentle vertices $\{1,2,3,4\}$ and obtain eight (locally) gentle algebras $B_i = \mathbf{k}Q/J_i$, $i=1,2,\dots,8$. This is because the quiver $Q$ fails to satisfy all conditions in Lemma \ref{three-conditions}, which implies that the geometric models of $A$ are not unique. Among them, $B_1 = \mathbf{k}Q/J_1$ with $J_1=<da_1,a_1b_1,b_2a_2,b_1c_1,c_2b_2,c_1d>$ is an infinite-dimensional algebra. There exist three infinite paths ${}^\infty (a_1a_2)^\infty, {}^\infty (b_1b_2)^\infty, {}^\infty (c_1c_2)^\infty,$ in its quiver, which correspond to three red punctures in its associated geometric model. Meanwhile, $B_2 = \mathbf{k}Q/J_2$ with $J_2=< a_2a_1,a_1a_2,b_2b_1,b_1c_1,c_2b_2,c_1c_2>$ is a finite-dimensional algebra. Every cycle in its quiver contains a relation of Type \ref{rel:R1}, so its geometric model admits no red puncture. The concrete geometric models are presented in Figure \ref{fig: Type I+II+III}.

\begin{figure}[htbp]
	\begin{center}
		\begin{tikzpicture}[scale=0.35]
			
			\begin{scope}[xshift=-15,yshift=10cm]
                \draw[line width=1.5pt,fill=white] (0,0) circle (6cm);
				
				\path 
				(0:4) coordinate (b1)
				(-90:6) coordinate (b2)
                (180:4) coordinate (b3)
				(90:4) coordinate (b4)
                (180:6) coordinate (r1)
                (0:0) coordinate (r2)
				;

                \draw[thick,red, fill=red] (b1) circle (0.15cm) (b2) circle (0.15cm)
                (b3) circle (0.15cm) (b4) circle (0.15cm);
                \draw[thick,black, fill=white] 
	              (r1) circle (0.15cm) (r2) circle (0.15cm);
                
				\draw[line width=1pt,red] (b2) to(b1);
                \draw[line width=1pt,red] (b4) to(b1);
                 \draw[line width=1pt,red] (b4) to(b3);
                  \draw[line width=1pt,red] (b2) to(b3);

               \draw[bluearrow] (-3.8,0.2)arc[start angle=60, end angle=-160, radius=0.3] arc[start angle=-160, end angle=-295, radius=0.45];

                \draw[bluearrow] (0.2,3.8)arc[start angle=-30, end angle=-260, radius=0.3] arc[start angle=-260, end angle=-380, radius=0.45] ;
                \draw[bluearrow] (-0.5,3.5)arc[start angle=-150, end angle=-400, radius=0.65] arc[start angle=-400, end angle=-500, radius=0.75];

                \draw[bluearrow] (3.85,-0.2)arc[start angle=-100, end angle=-310, radius=0.3] arc[start angle=-310, end angle=-460, radius=0.45]  ;
                \draw[bluearrow] (3.6,0.45)arc[start angle=120, end angle=-120, radius=0.65] arc[start angle=-120, end angle=-225, radius=0.75];

                 \draw(0,-7) node[black] {$B_{1}=\mathbf{k}Q/J_{1}$}; 
                 
               \draw(3,-3) node[black] {\scriptsize$\mathbf{x_{1}}$};
\draw(3,2) node[black] {\scriptsize$\mathbf{x_{2}}$};
\draw(-3,2) node[black] {\scriptsize$\mathbf{x_{3}}$};
\draw(-3,-3) node[black] {\scriptsize$\mathbf{x_{4}}$};
			\end{scope}
			
\begin{scope}[xshift=15cm,yshift=10cm]
\draw[line width=1pt,fill=white] (0,0) circle (6cm);
\draw[line width=1.5pt,fill=gray!50] (0, -5) circle (0.8cm);

\path 
    (0, -5)+(90:0.8) coordinate (b1)
    (0,-5)+(-60:0.8) coordinate (r1)
    (70:6) coordinate (m1);
    
\draw[thick,red, fill=red] (b1) circle(0.15cm) ;
\draw[thick,black, fill=white] (r1) circle(0.15cm);

\draw[line width=1pt,red] (b1)
to[out=-180,in=-90] (-5.5,0)to[out=90,in=100,looseness=1.5] (b1);

\draw[line width=1pt,red] (b1) to[out=90,in=150,looseness=1.5] (4,3) to[out=-30,in=10,looseness=1.4] (b1);

\draw[line width=1pt,red] (b1) to (1.8,0.5);
\draw[bend right=80,line width=1pt,red,dashed] (1.8,0.5) to (4.8, 0.5);
\draw[bend right,line width=1pt,red] (b1) to (4.8,0.5);

\draw[bend left,line width=1pt,red] (b1) to(-2.2, -1.5);
\draw[bend left=50,line width=1pt,red,dashed] (-2.2, -1.5) to(m1);
\draw[bend right,line width=1pt,red] (b1) to (3.8,0.2);
\draw[bend right,line width=1pt,red,dashed] (3.8,0.2) to (m1);

\draw[line width=1.5pt] (-3.2, -1.8) to[out=70, in=110, looseness=2.5] (-2.2, -1.5); 
\draw[line width=1.5pt] 
(-3.5,-1.4) to[out=-60, in=-66, looseness=2] (-3.2, -1.8) 
to[out=-70, in=48, looseness=2] (-2.2, -1.5)
to[out=48, in=50, looseness=2](-1.9,-1.2);

\draw[line width=1.5pt] (0.8, 0.2) to[out=70, in=110, looseness=2.5] (1.8, 0.5); 
\draw[line width=1.5pt] 
(0.5,0.6) to[out=-60, in=-66, looseness=2] (0.8, 0.2) 
to[out=-70, in=48, looseness=2] (1.8, 0.5)
to[out=48, in=50, looseness=2](2.1,0.8);

\draw[line width=1.5pt] (3.8, 0.2) to[out=70, in=110, looseness=2.5] (4.8, 0.5); 
\draw[line width=1.5pt] 
(3.5,0.6) to[out=-60, in=-66, looseness=2] (3.8, 0.2) 
to[out=-70, in=48, looseness=2] (4.8, 0.5)
to[out=48, in=50, looseness=2](5.1,0.8);

\draw[bend left,bluearrow] (-0.75, -3.8) to (0, -3.3); 
\draw[bend left=30,bluearrow] (-0.3, -3.1) to (0.35, -3.3); 
\draw[bend left=30,bluearrow] (0, -2.5) to (1.6, -3.4);
\draw[bend left=30,bluearrow] (0.5, -2.9) to (1.4, -3.7); 
\draw[bend left,bluearrow] (2.1, -3) to (2.7,-3.3) ; 

\draw(0,-7) node[black] {$B_{2}=\mathbf{k}Q/J_{2}$}; 

\draw(-0.5,0.5) node[black] {\scriptsize$\mathbf{x_{1}}$};
\draw(1.8,-1.5) node[black] {\scriptsize$\mathbf{x_{2}}$};
\draw(-2.5,-2.5) node[black] {\scriptsize$\mathbf{x_{3}}$};
\draw(-4.5,2) node[black] {\scriptsize$\mathbf{x_{4}}$};

            \end{scope}
		\end{tikzpicture}
        
		\caption{Example \ref{ex:Type I+II+III}: The geometric models on the left contain  red punctures, while those on the right have no red puncture.}
		\label{fig: Type I+II+III}
	\end{center}
\end{figure}
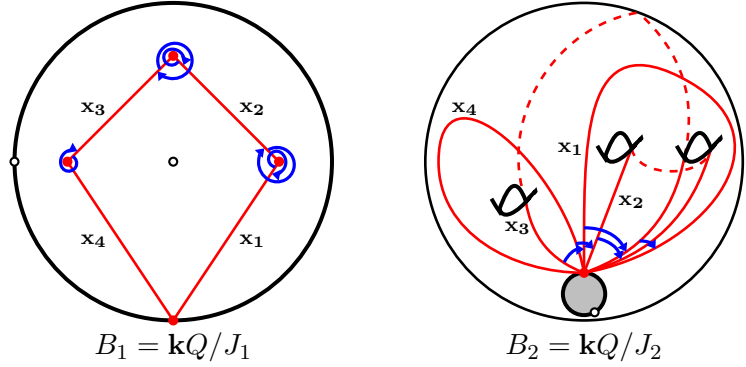
\end{example}

\begin{example}\label{ex:9-point}
Let $A = \mathbf{k}Q/I$ with $Q$  the quiver
\begin{tikzcd}
                               & 4 \arrow[ld, "b_1" description] \arrow[rd, "b_4" description] &                                \\
1 \arrow[r, "a_1" description] & 5 \arrow[u, "a_4" description] \arrow[d, "a_2" description]   & 3 \arrow[l, "a_3" description] \\
                               & 2 \arrow[ru, "b_3" description] \arrow[lu, "b_2" description] &                               
\end{tikzcd}
and admissible ideal 
$I = <b_{1}a_{1}, b_{2}a_{1}, a_{2}b_{2}, a_{2}b_{3}, b_{3}a_{3}, b_{4}a_{3}, a_{4}b_{1}, a_{4}b_{4}, a_{1}a_{4}, a_{1}a_{2}, a_{3}a_{4}, a_{3}a_{2}>.$ 
We transform the non-gentle vertices 
$\{ 1,2,3,4,5\} \subseteq Q_{0}$ and obtain thirty-two (locally) gentle algebras $B_i = \mathbf{k}Q/J_i$, $i=1,2,\dots,32$. Since the quiver $Q$ fails to satisfy \ref{U1}(a$_2$), the geometric models of $A$ are not unique . In addition, $B_1 = \mathbf{k}Q/J_1$ with $J_1 = < b_2a_1,a_2b_2,b_4a_3,a_4b_4,a_1a_2,a_3a_4>$ is an infinite-dimensional algebra. There exist two infinite paths ${}^\infty (a_1a_4b_1)^\infty, {}^\infty (a_3a_2b_3)^\infty$ in its quiver, which correspond to two red punctures in its associated geometric model. Meanwhile, $B_2 = \mathbf{k}Q/J_2$ with $J_2 = < b_1a_1,a_2b_3,b_4a_3,a_4b_1,a_1a_2,a_3a_4>$ is a finite-dimensional algebra. Every cycle in its quiver contains a relation of Type \ref{rel:R1}, so its geometric model admits no red puncture, see Figure \ref{fig: 9-point}.

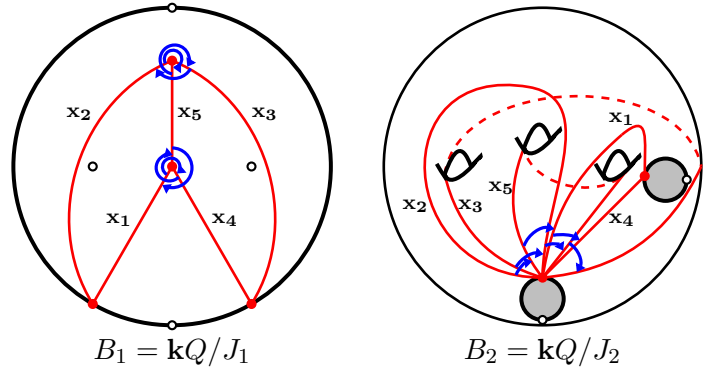
\begin{figure}[htbp]
	\begin{center}
		\begin{tikzpicture}[scale=0.35]
			
			\begin{scope}[xshift=0,yshift=10cm]
                \draw[line width=1.5pt,fill=white] (0,0) circle (6cm);
				
				\path 
				(0:0) coordinate (b1)
				(90:4) coordinate (b2)
                (-60:6) coordinate (b3)
				(-120:6) coordinate (b4)
                (90:6) coordinate (r1)
                (-90:6) coordinate (r2)
                (0:3) coordinate (r3)
                (180:3) coordinate (r4)
				;

                \draw[thick,red, fill=red] (b1) circle (0.15cm) (b2) circle (0.15cm)
                (b3) circle (0.15cm) (b4) circle (0.15cm);
                \draw[thick,black, fill=white] 
	              (r1) circle (0.15cm) (r2) circle (0.15cm) (r3) circle (0.15cm) (r4) circle (0.15cm);
                
				\draw[line width=1pt,red] (b2) to(b1);
                \draw[line width=1pt,red] (b3) to(b1);
				\draw[line width=1pt,red] (b4) to(b1);
                \draw[bend left=50,line width=1pt,red] (b4) to(b2);
                \draw[bend right=50,line width=1pt,red] (b3) to(b2);

                \draw[bluearrow] (-0.2,-0.3)arc[start angle=-120, end angle=-410, radius=0.33];
                \draw[bluearrow] (0.25,-0.4)arc[start angle=-45, end angle=-285, radius=0.5];
                \draw[bluearrow] (0,0.7)arc[start angle=90, end angle=-120, radius=0.75];

                \draw[bluearrow] (-0.25,3.9)arc[start angle=-150, end angle=-450, radius=0.33];
                \draw[bluearrow] (0,3.5)arc[start angle=-100, end angle=-390, radius=0.55];
                 \draw[bluearrow] (0.7,3.7)arc[start angle=-30, end angle=-150, radius=0.75];

                 \draw(0,-7) node[black] {$B_{1}=\mathbf{k}Q/J_{1}$}; 
                 
                \draw(-2,-2) node[black] {\scriptsize$\mathbf{x_{1}}$};
                \draw(-3.5,2) node[black] {\scriptsize$\mathbf{x_{2}}$};
                 \draw(3.5,2) node[black] {\scriptsize$\mathbf{x_{3}}$};
                \draw(2,-2) node[black] {\scriptsize$\mathbf{x_{4}}$};
			    \draw(0.7,2) node[black] {\scriptsize$\mathbf{x_{5}}$};
			\end{scope}
			
\begin{scope}[xshift=14cm,yshift=10cm]
\draw[line width=1pt,fill=white] (0,0) circle (6cm);
\draw[line width=1.5pt,fill=gray!50] (0, -5) circle (0.8cm);
\draw[line width=1.5pt,fill=gray!50] (4.65, -0.5) circle (0.8cm);

\path 
    (0, -5)+(90:0.8) coordinate (b1)
    (4.65,-0.5)+(170:0.8) coordinate (b2)
    (0,-5)+(-90:0.8) coordinate (r1)
    (4.65,-0.5)+(0:0.8) coordinate (r2)
    (0:6) coordinate (m1);
    
\draw[thick,red, fill=red] (b1) circle(0.15cm) (b2) circle(0.15cm) ;
\draw[thick,black, fill=white] (r1) circle(0.15cm) (r2) circle(0.15cm);

\draw[line width=1pt,red] (b1)
to [out=-180,in=170,looseness=2] (-0.5,3) to[out=-10,in=90] (b1);

\draw[bend left,line width=1pt,red] (b1) to (-3.7,-0.3);
\draw[bend right,line width=1pt,red] (b1) to (m1);
\draw[bend left=80,line width=1pt,red,dashed] (-3.7,-0.3) to (m1);

\draw[,line width=1pt,red] (b1) to (b2);

\draw[bend left,line width=1pt,red] (b1) to (-0.7,0.7);
\draw[bend left=70,line width=1pt,red,dashed] (3.3, 0) to  (-0.7,0.7);
\draw[,line width=1pt,red] (b1) to (3.3, 0);

\draw[bend left=20,line width=1pt,red] (b1) to (3,1.2);
\draw[line width=1pt,red] (b2) to[out=90,in=40,looseness=2] (3,1.2);

\draw[line width=1.5pt] (-3.7, -0.3) to[out=70, in=110, looseness=2.5] (-2.7, 0); 
\draw[line width=1.5pt] 
(-4.0,0.1) to[out=-60, in=-66, looseness=2] (-3.7, -0.3) 
to[out=-70, in=48, looseness=2] (-2.7, 0)
to[out=48, in=50, looseness=2](-2.4,0.3);

\draw[line width=1.5pt] (-0.7, 0.7) to[out=70, in=110, looseness=2.5] (0.3, 1); 
\draw[line width=1.5pt] 
(-1.0,1.1) to[out=-60, in=-66, looseness=2] (-0.7, 0.7) 
to[out=-70, in=48, looseness=2] (0.3, 1)
to[out=48, in=50, looseness=2](0.6,1.3);

\draw[line width=1.5pt] (2.3, -0.3) to[out=70, in=110, looseness=2.5] (3.3, 0); 
\draw[line width=1.5pt] 
(2,0.1) to[out=-60, in=-66, looseness=2] (2.3, -0.3) 
to[out=-70, in=48, looseness=2] (3.3, 0)
to[out=48, in=50, looseness=2](3.6,0.3);

\draw[bluearrow] (0.1,-3)arc[start angle=120, end angle=30, radius=0.5];
\draw[bluearrow] (-0.7,-3)arc[start angle=160, end angle=78, radius=1];
\draw[bluearrow] (-1,-4.1)arc[start angle=160, end angle=100, radius=0.8];
\draw[bluearrow] (1,-2.9)arc[start angle=60, end angle=-30, radius=0.8];
\draw[bluearrow] (0.4,-2.6)arc[start angle=100, end angle=60, radius=1.5];
\draw[bluearrow] (-1,-3.8)arc[start angle=150, end angle=80, radius=1];

                 \draw(0,-7) node[black] {$B_{2}=\mathbf{k}Q/J_{2}$}; 
                 
\draw(3,1.8) node[black] {\scriptsize$\mathbf{x_{1}}$};
\draw(-4.8,-1.5) node[black] {\scriptsize$\mathbf{x_{2}}$};
\draw(-2.7,-1.5) node[black] {\scriptsize$\mathbf{x_{3}}$};
\draw(3,-2) node[black] {\scriptsize$\mathbf{x_{4}}$};
\draw(-1.55,-0.8) node[black] {\scriptsize$\mathbf{x_{5}}$};

            \end{scope}
		\end{tikzpicture}
        
		\caption{Example \ref{ex:9-point}: The geometric models on the left contain red punctures, while those on the right have no red puncture.}
		\label{fig: 9-point}
	\end{center}
\end{figure}
\end{example}
\begin{example}\label{ex:genus}
Let $A = \mathbf{k}Q/I$ with $Q$  the quiver
\begin{tikzcd}[row sep=small, column sep=small]
1 \arrow[r, "b"'] & 3 \arrow[rd, "a"'] &                    &  &                                                                  \\
                 &                    & 5 \arrow[rr, "d"]  &  & 6 \arrow[lllu, "c"', bend right=15] \arrow[llld, "g", bend left=15] \\
2 \arrow[r, "e"] & 4 \arrow[ru, "f"]  &                    &  &                                                                  
\end{tikzcd}
and admissible ideal $I = <ba,ef,ad,fd,dg,dc>$. We transform non-gentle vertices $\{5,6\} \subseteq Q_{0}$. We get the four (locally) gentle algebras $B_{i} = \mathbf{k}Q/J_{i}$, $i=1,2,3,4$, where $J_{1}= < ba, ef, ad, dg>$, $J_{2} = <ba, ef, fd, dc>$, $J_{3}= <ba, ef, ad, dc>$ and $J_{4} = <ba, ef, fd, dg>$. Clearly, $B_{1} \cong B_{2}$ are finite-dimensional, whereas $B_{3} \cong B_{4}$ are infinite-dimensional, and $B_1$ is not isomorphic to $B_3$. This is because the vertices $5$ and $6$ fail to satisfy \ref{U3}, which implies that the geometric models of $A$ are not unique . In addition, every cycle in the quiver corresponding to the algebra $B_1$ (resp. $B_2$) contains a relation of Type \ref{rel:R1}, so its associated geometric model admits no red puncture. However, there exists an infinite path ${}^\infty (fdg)^\infty$(resp.~ ${}^\infty (adc)^\infty$) in the quiver corresponding to the algebra $B_3$ (resp. $B_4$), and thus its associated geometric model contains a red puncture, see Figure \ref{fig:genus}.

\begin{figure}[htbp]
  \centering
\begin{tikzpicture}[scale=0.25]
    \draw[line width=1pt,fill=white] (0,-1) circle (7.2cm);
    \draw[line width=1.5pt,fill=gray!50] (0, -3) circle (2.5cm);

    \coordinate (b1) at ($(0,-3)+(90:2.5)$);
    \coordinate (r1) at ($(0,-3)+(54:2.5)$);
    \coordinate (b2) at ($(0,-3)+(18:2.5)$);
    \coordinate (r2) at ($(0,-3)+(-18:2.5)$);
    \coordinate (b3) at ($(0,-3)+(-54:2.5)$);
    \coordinate (r3) at ($(0,-3)+(-90:2.5)$);
    \coordinate (b4) at ($(0,-3)+(-126:2.5)$);
    \coordinate (r4) at ($(0,-3)+(-162:2.5)$);
    \coordinate (b5) at ($(0,-3)+(162:2.5)$);
    \coordinate (r5) at ($(0,-3)+(126:2.5)$);

    \draw[thick, black, fill=white]
         (r1) circle (0.2cm)
         (r2) circle (0.2cm)
         (r3) circle (0.2cm)
         (r4) circle (0.2cm)
         (r5) circle (0.2cm);
    \draw[thick, red, fill=red]
         (b1) circle (0.2cm)
         (b2) circle (0.2cm)
         (b3) circle (0.2cm)
         (b4) circle (0.2cm)
         (b5) circle (0.2cm);

    \draw[line width=1.5pt] (1.1, 1.7) to[out=70, in=110, looseness=2.5] (2.1, 2.0);
    \draw[line width=1.5pt] (0.8,2.1) to[out=-60, in=-66, looseness=2] (1.1, 1.7) to[out=-70, in=48, looseness=2] (2.1, 2.0)to[out=48, in=50, looseness=2](2.4,2.3);
    \draw[line width=1.5pt] (-2.2, 2.2) to[out=70, in=110, looseness=2.5] (-1.2, 2.5);
    \draw[line width=1.5pt] (-2.5,2.6) to[out=-60, in=-66, looseness=2] (-2.2, 2.2) to[out=-70, in=48, looseness=2] (-1.2, 2.5)to[out=48, in=50, looseness=2](-0.9,2.8);

    \draw[line width=1pt, red] (b1) to (-2.2, 2.15);
    \coordinate (m) at ($(0,-1)+(-150:7.2)$);
    \draw[bend right=50,line width=1pt, red,dashed] (-2.2, 2.15) to (m);
    \draw[bend right=50,line width=1pt, red]  (m) to (b3);
    \draw[line width=1pt, red] (b1) to[out=10, in=0, looseness=2.5] (0,5.8) to[out=180,in=60] (-5,2.5) to [out=-120,in=160] (b5);
    \draw[line width=1pt, red] (b1) to (2.1, 2.0);
    \coordinate (n) at ($(0,-1)+(20:7.2)$);
    \draw[bend left=50,line width=1pt, red,dashed]  (2.1, 2.0) to (n);
    \draw[line width=1pt, red]  (b1) to[out=0,in=-70] (n);
    \draw[line width=1pt, red] (b1) to[out=170, in=180, looseness=2] (0,5) to[out=0, in=10, looseness=2] (b1);
    \draw[line width=1pt,red] (b2) to[out=-30,in=10,looseness=2]  (b3);
    \draw[line width=1pt,red] (b2) to[out=-30,in=10,looseness=2]  (b3);
    \draw[line width=1pt,red] (b4) to[out=180,in=-150,looseness=2]  (b5);
    
\node[black] at (-4.5,-3.5) {\scriptsize$\mathbf{x_{1}}$};
\node[black] at (4.5,-3.5) {\scriptsize$\mathbf{x_{2}}$};
\node[black] at (-4.1,0) {\scriptsize$\mathbf{x_{3}}$};
\node[black] at (-1,1.7) {\scriptsize$\mathbf{x_{4}}$};
\node[black] at (0.7,1.3) {\scriptsize$\mathbf{x_{5}}$};
\node[black] at (0,4.5) {\scriptsize$\mathbf{x_{6}}$};
    \draw[bluearrow] (-0.35,0)arc[start angle=150, end angle=37, radius=1.2];
    \draw[bluearrow] (1.13,0.75)arc[start angle=100, end angle=20, radius=0.9];
    \draw(0,-9.5) node[black] {$B_{1}=\mathbf{k}Q/J_{1}$};
\end{tikzpicture}
\hspace{-3.5mm}
\begin{tikzpicture}[scale=0.25]
    \draw[line width=1pt,fill=white] (0,-1) circle (7.2cm);
    \draw[line width=1.5pt,fill=gray!50] (0, -3) circle (2.5cm);

    \coordinate (b1) at ($(0,-3)+(90:2.5)$);
    \coordinate (r1) at ($(0,-3)+(54:2.5)$);
    \coordinate (b2) at ($(0,-3)+(18:2.5)$);
    \coordinate (r2) at ($(0,-3)+(-18:2.5)$);
    \coordinate (b3) at ($(0,-3)+(-54:2.5)$);
    \coordinate (r3) at ($(0,-3)+(-90:2.5)$);
    \coordinate (b4) at ($(0,-3)+(-126:2.5)$);
    \coordinate (r4) at ($(0,-3)+(-162:2.5)$);
    \coordinate (b5) at ($(0,-3)+(162:2.5)$);
    \coordinate (r5) at ($(0,-3)+(126:2.5)$);

    \draw[thick, black, fill=white]
         (r1) circle (0.2cm)
         (r2) circle (0.2cm)
         (r3) circle (0.2cm)
         (r4) circle (0.2cm)
         (r5) circle (0.2cm);
    \draw[thick, red, fill=red]
         (b1) circle (0.2cm)
         (b2) circle (0.2cm)
         (b3) circle (0.2cm)
         (b4) circle (0.2cm)
         (b5) circle (0.2cm);

    \draw[line width=1.5pt] (1.1, 1.7) to[out=70, in=110, looseness=2.5] (2.1, 2.0);
    \draw[line width=1.5pt] (0.8,2.1) to[out=-60, in=-66, looseness=2] (1.1, 1.7) to[out=-70, in=48, looseness=2] (2.1, 2.0)to[out=48, in=50, looseness=2](2.4,2.3);
    \draw[line width=1.5pt] (-2.2, 2.2) to[out=70, in=110, looseness=2.5] (-1.2, 2.5);
    \draw[line width=1.5pt] (-2.5,2.6) to[out=-60, in=-66, looseness=2] (-2.2, 2.2) to[out=-70, in=48, looseness=2] (-1.2, 2.5)to[out=48, in=50, looseness=2](-0.9,2.8);

    \draw[line width=1pt, red] (b1) to (-2.2, 2.15);
    \coordinate (m) at ($(0,-1)+(-150:7.2)$);
    \draw[bend right=50,line width=1pt, red,dashed] (-2.2, 2.15) to (m);
    \draw[bend right=50,line width=1pt, red]  (m) to (b3);
    \draw[line width=1pt, red] (b1) to[out=10, in=0, looseness=2.5] (0,5.8) to[out=180,in=60] (-5,2.5) to [out=-120,in=160] (b5);
    \draw[line width=1pt, red] (b1) to (2.1, 2.0);
    \coordinate (n) at ($(0,-1)+(20:7.2)$);
    \draw[bend left=50,line width=1pt, red,dashed]  (2.1, 2.0) to (n);
    \draw[line width=1pt, red]  (b1) to[out=0,in=-70] (n);
    \draw[line width=1pt, red] (b1) to[out=170, in=180, looseness=2] (0,5) to[out=0, in=10, looseness=2] (b1);
    \draw[line width=1pt,red] (b2) to[out=-30,in=10,looseness=2]  (b3);
    \draw[line width=1pt,red] (b2) to[out=-30,in=10,looseness=2]  (b3);
    \draw[line width=1pt,red] (b4) to[out=180,in=-150,looseness=2]  (b5);

   \node[black] at (-4.5,-3.5) {\scriptsize$\mathbf{x_{2}}$};
\node[black] at (4.5,-3.5) {\scriptsize$\mathbf{x_{1}}$};
\node[black] at (-4.1,0) {\scriptsize$\mathbf{x_{4}}$};
\node[black] at (-1,1.7) {\scriptsize$\mathbf{x_{3}}$};
\node[black] at (0.7,1.3) {\scriptsize$\mathbf{x_{5}}$};
\node[black] at (0,4.5) {\scriptsize$\mathbf{x_{6}}$};

    \draw[bluearrow] (-0.35,0)arc[start angle=150, end angle=37, radius=1.2];
    \draw[bluearrow] (1.13,0.75)arc[start angle=100, end angle=20, radius=0.9];
    \draw(0,-9.5) node[black] {$B_{2}=\mathbf{k}Q/J_{2}$};
\end{tikzpicture}
\hspace{-7.5mm}
\begin{tikzpicture}[scale=0.25]
    \draw[line width=1.5pt,fill=white] (0,-1) circle (7.2cm);
    \draw[line width=1.5pt,fill=gray!50] (-2.5, -1) circle (1cm);

    \draw[thick, red, fill=red] ($(-2.5,-1)+(90:1)$) circle (0.2cm) coordinate (b4);
    \draw[thick, black, fill=white] ($(-2.5,-1)+(0:1)$) circle (0.2cm) coordinate (r4);
    \draw[thick, red, fill=red] ($(-2.5,-1)+(-90:1)$) circle (0.2cm) coordinate (b5);
    \draw[thick, black, fill=white] ($(-2.5,-1)+(180:1)$) circle (0.2cm) coordinate (r5);

    \draw[thick, red, fill=red] (50:1) circle (0.2cm) coordinate (b0);
    \draw[thick, red, fill=red] (0,-1) +(60:7.2) circle (0.2cm) coordinate (b1);
    \draw[thick, red, fill=red] (0,-1) +(-15:7.2) circle (0.2cm) coordinate (b2);
    \draw[thick, red, fill=red] (0,-1)+(-90:7.2) circle (0.2cm) coordinate (b3);

    \draw[thick, black, fill=white] (0,-1)+(30:7.2) circle (0.2cm) coordinate (r1);
    \draw[thick, black, fill=white] (0,-1)+(-70:7.2) circle (0.2cm) coordinate (r2);
    \draw[thick, black, fill=white] (0,-1)+(150:7.2) circle (0.2cm) coordinate (r3);

    \draw[bend left,line width=1pt,red] (b1) to (b0);
    \draw[line width=1pt,red] (b1) to (b2);
    \draw[line width=1pt,red] (b0) to (b3);
    \draw[line width=1pt,red] (b4) to[out=0, in=0, looseness=3] (b5);
    \draw[line width=1pt,red] (b5) to (b3);
    \draw[line width=1pt, red] (b3) to[out=160, in=145, looseness=2.7] (b0);

    \draw[bluearrow] ($(b0)+(35:1)$) arc[start angle=20, end angle=-200, radius=1];
    \draw[bluearrow] ($(b0)+(-100:0.5)$) arc[start angle=-95, end angle=-320, radius=0.5];

    \node[black] at (-0.6,-2.3) {\scriptsize$\mathbf{x_{1}}$};
\node[black] at (1.5,2.5) {\scriptsize$\mathbf{x_{4}}$};
\node[black] at (1.2,-2) {\scriptsize$\mathbf{x_{5}}$};
\node[black] at (-2.8,-3.5) {\scriptsize$\mathbf{x_{3}}$};
\node[black] at (4,1.5) {\scriptsize$\mathbf{x_{2}}$};
\node[black] at (-2.5,2.7) {\scriptsize$\mathbf{x_{6}}$};
    \draw(0,-9.5) node[black] {$B_{3}=\mathbf{k}Q/J_{3}$};
\end{tikzpicture}
\hspace{-6mm}
\begin{tikzpicture}[scale=0.25]
    \draw[line width=1.5pt,fill=white] (0,-1) circle (7.2cm);
    \draw[line width=1.5pt,fill=gray!50] (-2.5, -1) circle (1cm);

    \draw[thick, red, fill=red] ($(-2.5,-1)+(90:1)$) circle (0.2cm) coordinate (b4);
    \draw[thick, black, fill=white] ($(-2.5,-1)+(0:1)$) circle (0.2cm) coordinate (r4);
    \draw[thick, red, fill=red] ($(-2.5,-1)+(-90:1)$) circle (0.2cm) coordinate (b5);
    \draw[thick, black, fill=white] ($(-2.5,-1)+(180:1)$) circle (0.2cm) coordinate (r5);

    \draw[thick, red, fill=red] (50:1) circle (0.2cm) coordinate (b0);
    \draw[thick, red, fill=red] (0,-1) +(60:7.2) circle (0.2cm) coordinate (b1);
    \draw[thick, red, fill=red] (0,-1) +(-15:7.2) circle (0.2cm) coordinate (b2);
    \draw[thick, red, fill=red] (0,-1)+(-90:7.2) circle (0.2cm) coordinate (b3);

    \draw[thick, black, fill=white] (0,-1)+(30:7.2) circle (0.2cm) coordinate (r1);
    \draw[thick, black, fill=white] (0,-1)+(-70:7.2) circle (0.2cm) coordinate (r2);
    \draw[thick, black, fill=white] (0,-1)+(150:7.2) circle (0.2cm) coordinate (r3);

    \draw[bend left,line width=1pt,red] (b1) to (b0);
    \draw[line width=1pt,red] (b1) to (b2);
    \draw[line width=1pt,red] (b0) to (b3);
    \draw[line width=1pt,red] (b4) to[out=0, in=0, looseness=3] (b5);
    \draw[line width=1pt,red] (b5) to (b3);
    \draw[line width=1pt, red] (b3) to[out=160, in=145, looseness=2.7] (b0);

    \draw[bluearrow] ($(b0)+(35:1)$) arc[start angle=20, end angle=-200, radius=1];
    \draw[bluearrow] ($(b0)+(-100:0.5)$) arc[start angle=-95, end angle=-320, radius=0.5];

    \node[black] at (-0.6,-2.3) {\scriptsize$\mathbf{x_{2}}$};
\node[black] at (1.5,2.5) {\scriptsize$\mathbf{x_{3}}$};
\node[black] at (1.2,-2.6) {\scriptsize$\mathbf{x_{5}}$};
\node[black] at (-2.6,-3.7) {\scriptsize$\mathbf{x_{4}}$};
\node[black] at (4.5,1.1) {\scriptsize$\mathbf{x_{1}}$};
\node[black] at (-2.5,2.6) {\scriptsize$\mathbf{x_{6}}$};
    \draw(0,-9.5) node[black] {$B_{4}=\mathbf{k}Q/J_{4}$};
\end{tikzpicture}
  \caption{Example \ref{ex:genus}: The two left geometric models are equivalent, as are the two right ones, yet the left and right models aren't equivalent.}
  \label{fig:genus}
\end{figure}
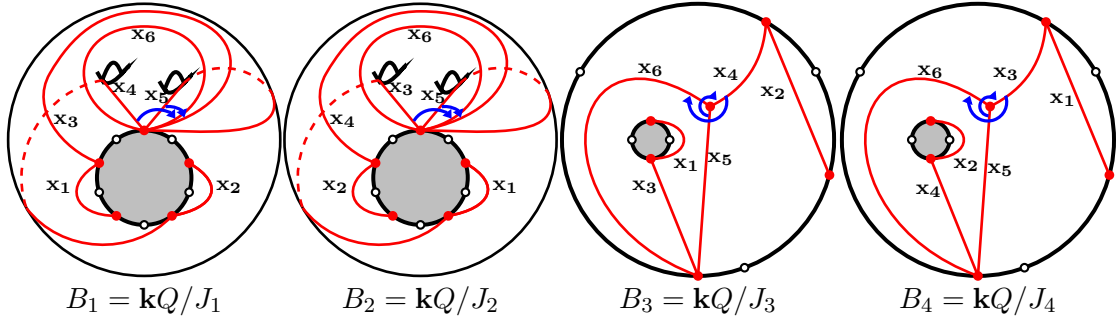

\end{example}

\begin{example}\label{ex:cycle-every surface2}
Let $A = \mathbf{k}Q/I$ with $Q$  the quiver
\begin{tikzcd}[row sep=small, column sep=small]
                        & 3 \arrow[ld, "f"', bend right] & \\
1 \arrow[rr, "b" description, shift right=2] 
  \arrow[rr, "a" description, shift left=2] 
                        &                               & 2 \arrow[lu, "c"', bend right] 
                        \arrow[ld, "d", bend left] \\
                        & 4 \arrow[lu, "e", bend left]    &
\end{tikzcd}
and admissible ideal $I = <ac,ad,bc,bd,ea,eb,fa,fb>$. We transform non-gentle vertices $\{1,2\} \subseteq Q_{0}$. We get the four locally gentle algebras $B_{i} = \mathbf{k}Q/J_{i}$, $i=1,2,3,4$, where $J_{1}= <  fa,eb,ad,bc>$, $J_{2} = < fb,ea,ac,bd>$, $J_{3}= <fa,eb,ac,bd>$ and $J_{4} = < fb,ea,ad,bc> $. Clearly, $B_1 \cong B_2$. The quiver corresponding to $B_1$ (resp. $B_2$) contains an infinite path ${}^\infty (fbdeac)^\infty$ (resp. ${}^\infty (fadebc)^\infty$), which gives rise to one red puncture in its associated geometric model. Meanwhile, $B_3 \cong B_4$. The quiver corresponding to $B_3$ (resp. $B_4$) contains two infinite paths ${}^\infty (fbc)^\infty, {}^\infty (ead)^\infty$ (resp. ${}^\infty (fac)^\infty, {}^\infty (ebd)^\infty$), which yield two red punctures in its associated geometric model. The geometric models of $A$ are not unique , since the quiver $Q$ fails to satisfy condition \ref{U1}(a$_2$) in Lemma \ref{three-conditions}. The concrete geometric models are presented in Figure \ref{fig: dis-cycle}.

\begin{figure}[htbp]
	\begin{center}
		\begin{tikzpicture}[scale=0.35]
			\begin{scope}[xshift=0,yshift=10cm]
    \draw[line width=1pt,fill=white] (0,0) circle (6cm);

    \begin{scope}[xshift=1.5cm, yshift=1.5cm]
        \path 
        (0:0) coordinate (b1)
        (0,-3.5)+(-90:1) coordinate (b2)
        (-90:2.5) coordinate (b3) 
       (0,-3.5)+(180:1) coordinate (r2)
       (0,-3.5)+(0:1) coordinate (r1) 
       (-100:7.6) coordinate (p)
        ;

        \draw[line width=1.5pt,fill=gray!50] (0, -3.5) circle (1cm);
        
\draw[line width=1.5pt] (-3.7, -3.8) to[out=70, in=110, looseness=2.5] (-2.7, -3.5); 
\draw[line width=1.5pt] 
(-4.0,-3.4) to[out=-60, in=-66, looseness=2] (-3.7, -3.8) 
to[out=-70, in=48, looseness=2] (-2.7, -3.5)
to[out=48, in=50, looseness=2](-2.4,-3.2);

        \draw[thick,red, fill=red] (b1) circle (0.15cm) (b2) circle (0.15cm) (b3) circle (0.15cm);
       \draw[thick,black, fill=white] (r1) circle (0.15cm) (r2) circle (0.15cm);

        
        \draw[bend right,line width=1pt,red,dashed] (-3.7, -3.8) to(p);
        \draw[bend right,line width=1pt,red] (p)to(b2);
        \draw[line width=1pt,red] (b3) to (b1);
        \draw[line width=1pt,red] (b1)to[out=180,in=90](-3.7, -3.8) ;
        \draw[line width=1pt,red] (b1) to[out=-160,in=70] (-2.7,-3.37);
        \draw[line width=1pt,red,dashed]  (-2.7,-3.37)to[out=-110,in=170](0.5,-7.15) ; 
        
        \draw[line width=1pt,red] (b1)to[out=0,in=90](3,-2) to[out=-90,in=20](0.5,-7.05);
        \draw[line width=1pt,red] (b1) to[out=-110,in=0]  (-4,-5) to[out=180,in=-90](-6.5,-1.5)to[out=90,in=180](-2.5,3)to[out=0,in=60](b1); 

           \draw[bluearrow] (-0.3,-0.8)arc[start angle=-110, end angle=-200, radius=0.6];

           \draw[bluearrow,bend left=60] (0.2,0.5) to (0,-0.7);

           \draw[bluearrow] (-1,-0.1)arc[start angle=-200, end angle=-345, radius=1.2];

           \draw[bluearrow] (0,-1.2)arc[start angle=-110, end angle=-144, radius=2];
           
				\draw(-3.8,-5.8) node[black] {\scriptsize$\mathbf{x_{1}}$};
                
			    \draw(1.4,-1.5) node[black] {\scriptsize$\mathbf{x_{3}(x_{4})}$};

                \draw(3.5,-1.3) node[black] {\scriptsize$\mathbf{x_{2}}$};

                \draw(-3.8,1.3) node[black] {\scriptsize$\mathbf{x_{4}(x_{3})}$};

                 \draw(0,-8.5) node[black] {$B_{1}=\mathbf{k}Q/J_{1}(B_{2}=\mathbf{k}Q/J_{2})$}; 
    \end{scope}
\end{scope}
			
			\begin{scope}[xshift=17cm,yshift=10cm]
				\draw[line width=1.5pt,fill=white] (0,0) circle (6cm);
				\path 
				(0:4) coordinate (b1)
				(150:4) coordinate (b2)
                (20:6) coordinate (b3)
                (-90:2.5)coordinate (b4)
                 (-160:6)coordinate (r1)
                 (-90:4.5)coordinate (r2)
                ;
                 \draw[line width=1.5pt,fill=gray!50] (0, -3.5) circle (1cm);
                 \draw[thick,red, fill=red] (b1) circle(0.15cm) (b2) circle(0.15cm)(b3) circle(0.15cm) (b4) circle(0.15cm);

                 \draw[thick,black, fill=white] (r1) circle(0.15cm) (r2) circle(0.15cm);
                
              \draw[,line width=1pt,red] (b2) to[out=-110,in=180](0,-5) to[out=0,in=-110](b1);
              \draw[line width=1pt,red] (b1) to(b3);
              \draw[line width=1pt,red] (b2) to(b4);
              \draw[,line width=1pt,red] (b2) to(b1);

                \draw[bluearrow] (-3.2,1.65)arc[start angle=-50, end angle=-365, radius=0.4];
                \draw[bluearrow] (-3.65,1.45)arc[start angle=-110, end angle=-420, radius=0.7];

                 \draw[bluearrow] (3.75,-0.6)arc[start angle=-120, end angle=-205, radius=0.9]arc[start angle=-205, end angle=-320, radius=0.7];
                 \draw[bluearrow] (4.25,0.3)arc[start angle=40, end angle=-195, radius=0.4];
				\draw(0,-5.5) node[black] {\scriptsize$\mathbf{x_{1}}$};
                
			    \draw(0,1.5) node[black] {\scriptsize$\mathbf{x_{2}}$};

                \draw(4,2) node[black] {\scriptsize$\mathbf{x_{3}(x_{4})}$};

                \draw(0,-0.5) node[black] {\scriptsize$\mathbf{x_{4}(x_{3})}$};

                  \draw(0.5,-7) node[black] {$B_{3}=\mathbf{k}Q/J_{3}(B_{4}=\mathbf{k}Q/J_{4})$}; 
			\end{scope}

		\end{tikzpicture}
		\caption{Example \ref{ex:cycle-every surface2}: Both of the above geometric models admit a red puncture, but they are not equivalent.}

        \label{fig: dis-cycle}
	\end{center}
\end{figure}
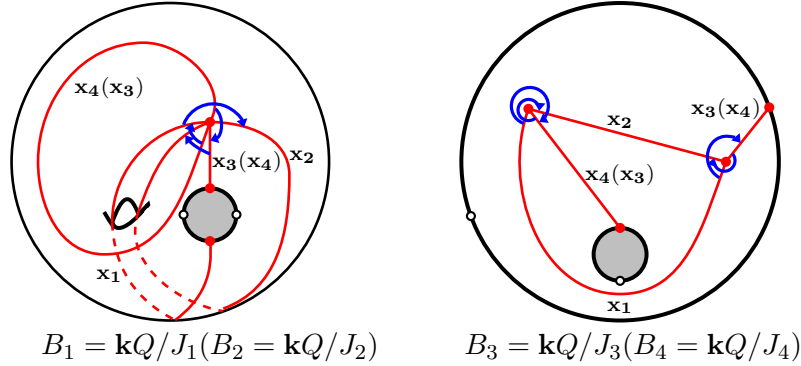

\end{example}

\end{document}